\documentclass[reqno]{amsart}
\usepackage[utf8]{inputenc}
\usepackage{amsmath,amssymb}
\usepackage{amsfonts}
\usepackage{xcolor}
\usepackage{hyperref}
\hypersetup{
    colorlinks=true,
    linkcolor=blue,
    filecolor=magenta,      
    urlcolor=cyan,
}
\usepackage{geometry}
\def\Xint#1{\mathchoice
{\XXint\displaystyle\textstyle{#1}}%
{\XXint\textstyle\scriptstyle{#1}}%
{\XXint\scriptstyle\scriptscriptstyle{#1}}%
{\XXint\scriptscriptstyle\scriptscriptstyle{#1}}%
\!\int}
\def\XXint#1#2#3{{\setbox0=\hbox{$#1{#2#3}{\int}$ }
\vcenter{\hbox{$#2#3$ }}\kern-.6\wd0}}

\def\dashint{\Xint-}

\usepackage{amsthm}
\newtheorem{theorem}{Theorem}[section]
\newtheorem{corollary}[theorem]{Corollary}
\newtheorem{lemma}[theorem]{Lemma}
\newtheorem{prop}[theorem]{Proposition}
\newtheorem{rem}[theorem]{Remark}
\newtheorem{claim}[theorem]{Claim}

\numberwithin{equation}{section}

\author{Anudeep K. Arora}
\address{Department of Mathematics, Statistics \& Computer Science, University of Illinois at Chicago, Chicago, IL 60607, USA}
\email{anudeep@uic.edu}

\author{Oscar Ria\~no }
\address{Department of Mathematics \& Statistics, Florida International University, Miami, FL 33199, USA}\email{orianoca@fiu.edu}

\author{Svetlana Roudenko}
\address{Department of Mathematics \& Statistics, Florida International University, Miami, FL 33199, USA}\email{sroudenko@fiu.edu}

\date{}

\title[Well-posedness in generalized Hartree, $p<2$ ]{Well-posedness in weighted spaces\\ for the generalized Hartree equation with $p<2$}

\subjclass[2010]{Primary: 35Q55, 35A01, 35B40; secondary: 42B37}

\keywords{Hartree equation, nonlocal potential, convolution nonlinearity, well-posedness, global existence, scattering, blow-up}

\begin{document}

\begin{abstract} 
We investigate the well-posedness in the generalized Hartree equation $iu_t + \Delta u  + (|x|^{-(N-\gamma)} \ast |u|^p)|u|^{p-2}u=0$, $x \in \mathbb{R}^N$, $0<\gamma<N$, for low powers of  nonlinearity, $p<2$. 
We establish the local well-posedness for a class of data in weighted Sobolev spaces, following ideas of Cazenave and Naumkin \cite{CazNaum2016}. This crucially relies on the boundedness of the Riesz transform in weighted Lebesgue spaces. 
As a consequence, we obtain a class of data that exists globally, moreover, scatters in positive time. Furthermore, in the focusing case in the $L^2$-supercritical setting we obtain a subset of locally well-posed data with positive energy, which blows up in finite time.  
\end{abstract}

\maketitle

\section{Introduction}

We consider the initial value problem associated to the generalized Hartree (gHartree), or Schr\"odinger-Hartree equation, 
\begin{equation}\label{gHartree}
\left\{\begin{aligned}
    &i u_t+\Delta u+\mu\Big(\frac{1}{|x|^{N-\gamma}}\ast |u|^{p}\Big)|u|^{p-2}u=0, \, \,  x \in \mathbb{R}^N, \, \, t\in \mathbb{R},\\
    & u(x,0)=u_0(x), 
\end{aligned}\right.
\end{equation}
where $0<\gamma<N$, $p<2$, $\mu\in \mathbb{C}\setminus\{0\}$ and $u=u(x,t)$ is a complex-valued function.

For general nonlinearity $p\geq 2$, the local well-posedness in $H^1$ 
was established in \cite{2019gHAnudRoud}. In this regard, our aim is to extend the well-posedness theory for nonlinearity with $p<2$. 
Such low powers appear, for example, when studying collapse or finite time blow-up in the focusing generalized Hartree equation, e.g., see \cite{YRZ2020}. The equation \eqref{gHartree} is a generalization of the standard Hartree equation with $p=2$, which shows in a number of physical models, for further details see the introduction in \cite{2019gHAnudRoud} or review \cite{ARY2020}. A generalized version of the Hartree equation allows more flexibility in approaches, thus, we are interested in the powers of nonlinearities when $1<p<2$. The typical methods for well-posedness do not work in these low nonlinearities (see discussion on that in \cite{CazNaum2016}), however, using the weighted Sobolev spaces, we are able to obtain the local well-posedness in a certain range of $p<2$. The range that we obtain is not optimal, but rather technical, since this is the first such study in the context of the Hartree-type equation, where we use well-known in harmonic analysis weighted estimates for the Riesz transform. In fact, this approach can be helpful in establishing well-posedness in equations with a potential that can be expressed as a Calderon-Zygmund operator.  
Inspired by the results in \cite{CazNaum2016} (see also \cite{CazHauNaum2020,CazNaum2018,LinaresMiyaGus,LinaresI,LinaresII,Miyazaki2020}), we 
introduce a class of initial data, which guarantees existence of local solutions of \eqref{gHartree} for nonlinearity with power $p<2$. Main difficulties arise in our analysis due to the presence of the Riesz potential operator and the lack of regularity of the term $\big(\frac{1}{|x|^{N-\gamma}}\ast |u|^{p}\big)|u|^{p-2}u$ when $p<2$.      
Before stating our results, we recall a few invariances 
and introduce some notation.

The equation \eqref{gHartree} formally conserves several quantities, in particular, if $\mu \in \mathbb{R}\setminus\{0\}$, solutions of \eqref{gHartree} satisfy the \emph{mass conservation} 
\begin{equation}\label{masscon}
M[u(t)]=\int_{\mathbb{R}^N} |u(x,t)|^2\, dx=M[u_0],
\end{equation}
the \emph{energy conservation}
\begin{equation}\label{energycon}
E[u(t)]=\frac{1}{2}\int_{\mathbb{R}^N} |\nabla u(x,t)|^2\, dx-\frac{\mu}{2p}\int_{\mathbb{R}^N}\Big(\frac{1}{|\cdot|^{N-\gamma}}\ast |u(\cdot,t)|^{p}\Big)(x,t)|u(x,t)|^p\, dx=E[u_0],
\end{equation}
and the \emph{momentum conservation}
\begin{equation}\label{momem}
P[u(t)]=\operatorname{Im} \Big( \int_{\mathbb{R}^N} \overline{u(x,t)}\nabla u(x,t) \, dx \Big)=P[u_0].
\end{equation}

Next, we fix $\frac{4}{3}<p<2$ and take $0<\gamma<\frac{N(3p-4)}{2p}$. We choose $m\in \mathbb{R}^{+}$ such that 
\begin{equation}\label{condmainT1}  
\max\Big\{ \frac{2\gamma+N}{4(p-1)},\frac{N}{2}\Big\}<m<\frac{N-2\gamma}{2(2-p)}.
\end{equation}
Additionally, we consider two positive integers $M_0$ and $M$ satisfying
\begin{equation}\label{condmainT2}
\begin{aligned}
M_0> \max\Big\{\frac{(N-\gamma)(2mp-N)}{4m(p-1)-N},N+m \Big\}
\end{aligned}
\end{equation}
and 
\begin{equation}\label{condmainT3}   
M > \max\Big\{M_0-N+2\lfloor \frac{N}{2}\rfloor+2,4\lfloor \frac{N}{2}\rfloor+5+m\Big\}.
\end{equation}
(Here, $\lfloor x \rfloor$ denotes the floor function, i.e., the greatest integer that is less than or equal to $x$).

We define the space $\mathfrak{X}=\mathfrak{X}(m,M,M_0)$ as follows
\begin{equation}\label{spacedef} 
\begin{aligned}
\mathfrak{X}=\big\{f \in & H^{M+M_0-N}(\mathbb{R}^N) : \\
&\langle x \rangle^{m}\partial^{\alpha}f \in L^{\infty}(\mathbb{R}^N), \, \, \text{ for each multi-index } \, \, |\alpha|\leq \lfloor \frac{N}{2}\rfloor, \\
& \langle x \rangle^{m}\partial^{\alpha}f \in L^{2}(\mathbb{R}^N), \, \,\text{ for each multi-index } \, \, \lfloor \frac{N}{2}\rfloor  <|\alpha|\leq M \big\},
 \end{aligned}
\end{equation}  
equipped with the norm 
\begin{equation}\label{normspace}   
\|f\|_{\mathfrak{X}}=\sum_{|\alpha|\leq \lfloor \frac{N}{2}\rfloor }\|\langle x \rangle^{m}\partial^{\alpha}f\|_{L^{\infty}_x}+\sum_{\lfloor \frac{N}{2}\rfloor< |\alpha|\leq M}\|\langle x \rangle^{m}\partial^{\alpha}f\|_{L^{2}_x}+\sum_{M<|\alpha|\leq M+M_0-N}\|\partial^{\alpha}f\|_{L^{2}},
\end{equation}
where $\langle x \rangle=(1+|x|^2)^{1/2}$. An example of such a space in 1D is given in \eqref{functspa1D}. 
\medskip

Our 
main result is the following local well-posedness.

\begin{theorem}\label{mainTHM}
Let $\frac{4}{3}<p<2$, $0<\gamma<\frac{N(3p-4)}{2p}$ and $\mu \in \mathbb{C}\setminus\{0\}$. Assume \eqref{condmainT1}-\eqref{condmainT3} and let $\mathfrak{X}$ be defined by \eqref{spacedef} and \eqref{normspace}. 
If $u_0 \in \mathfrak{X}$ with
\begin{equation}\label{condi1}
\|u_0\|_{\mathfrak{X}}=\eta,
\end{equation}
and 
\begin{equation}\label{condi2}
\inf_{x\in \mathbb{R}^N}|\langle x \rangle^m u_0(x)|\geq \lambda>0,
\end{equation}
then there exist $T=T(m,M,M_0,\eta, \lambda)>0$ and a unique solution of $u\in C([0,T];\mathfrak{X})$ of the initial value problem \eqref{gHartree}. Moreover, the map data-solution
\begin{equation}
u_0\mapsto u(\cdot,t)
\end{equation}
from a neighborhood of the datum $u_0\in \mathfrak{X}$ satisfying \eqref{condi1}-\eqref{condi2} into the class $ C([0,T];\mathfrak{X})$ is locally continuous.
\end{theorem}

\begin{rem}
\normalfont It is worth to emphasize that the results for the equations studied in \cite{CazHauNaum2020,CazNaum2016,CazNaum2018,LinaresMiyaGus,LinaresI,LinaresII,Miyazaki2020} deal with weights of arbitrary size only limited by lower bounds. Thus, in these references, it is more natural to consider weights with integer powers. In our case, we obtain our results with fractional weights, where 
we also need an upper bound in
the condition \eqref{condmainT1}. 
The upper constrain on the weight power $m$ stated in \eqref{condmainT1} is required to obtain weighted estimates for the term $\frac{1}{|x|^{N-\gamma}}\ast |u|^{p}$ provided by the nonlinearity in \eqref{gHartree}. This restriction comes from the boundedness of the Riesz potential in weighted $L^{p}(\mathbb{R}^N)$ spaces (see Proposition \ref{propweightriesz} below).
\end{rem}

The two 
consequences of this local well-posedness result are the global existence and scatering in the spirit of Cazenave and Naumkin \cite{CazNaum2016} in Theorem \ref{scatres} and blow-up in finite time in Theorem \ref{blowupcriteria} and Corollary \ref{blowupcriteriacor}.

\begin{theorem}\label{scatres}
Let $\frac{4}{3}<p<2$, $0<\gamma<\min\{\frac{N(3p-4)}{2p},N(p-1)-1\}$ and $\mu \in \mathbb{C}\setminus\{0\}$. Consider $m\in \mathbb{R}^{+}$, $M_0, M \in \mathbb{Z}^{+}$ satisfying \eqref{condmainT1}-\eqref{condmainT3}. Suppose that $u_0=e^{\frac{ib|x|^2}{4}}v_0$, $b>0$, and $v_0\in \mathfrak{X}$ satisfies \eqref{condi2}. If $b$ is sufficiently large, then for any $0\leq s<\frac{2m-N}{2}$ there exists a global solution $u$ of \eqref{gHartree} in  the class
$$C\big([0,\infty);H^{s}(\mathbb{R}^N)\big)\cap L^{\infty}\big(\mathbb{R}^N;\langle x \rangle^{\frac{N}{2}}\, dx\, dt) \big).$$ 
Moreover, $u$ scatters, i.e., there exists $u_{+}\in H^{s}(\mathbb{R}^N)$ such that
\begin{equation*}
\lim_{\substack{t \to \infty \\ t>0}} \|e^{-it\Delta}u(t)-u_{+}\|_{H^s}=0.
\end{equation*}
In addition,
\begin{equation*}
\sup_{t>0} \, (1+t)^{\frac{N}{2}}\|u(t)\|_{L^{\infty}}<\infty.
\end{equation*}
\end{theorem}

The result in \cite[Theorem 1.3]{AndySvetlan2} establishes a blow-up criterion for solutions of the focusing generalized Hartree equation \eqref{gHartree}, $p\geq 2$. Following these ideas (see also \cite{DucRoud,Holmer_2010,PavelI,PavelII}), we can extend this conclusion to solutions of \eqref{gHartree}, $p<2$, in the $L^2$-supercritical case ($s_c >0$) determined by Theorem \ref{mainTHM}. Here, the critical scaling index $s_c$, determined from the scaling invariance of this equation, is defined as
\begin{equation}\label{critiindex}
s_c=\frac{N}{2}-\frac{\gamma+2}{2(p-1)}.
\end{equation}
We also define the variance
\begin{equation}\label{variance}
V(t)=\int_{\mathbb{R}^N} |x|^2|u(x,t)|^2 \, dx.
\end{equation}
We note that the existence of blow-up in finite time for negative energy and finite variance in the $L^2$-critical ($s_c=0$) and $L^2$-supercritical ($s_c>0$) cases extends automatically once the local well-posedness is available for solutions in $H^1$. Here, we consider initial data with positive energy.

\begin{theorem}\label{blowupcriteria}
Let $\max\{\frac{N+2}{N},\frac{4}{3}\}<p<2$, $0<\gamma<\min\{N(p-1)-2,\frac{(N+2)(p-1)-2}{2},\frac{N(3p-4)}{2p}\}$ and $\mu > 0$. 
We take
\begin{equation}\label{condiblow}
\max\Big\{\frac{N+2}{2},\frac{2\gamma+N}{4(p-1)}\Big\}< m<\frac{N-2\gamma}{2(2-p)},
\end{equation} 
and $M_0,M \in \mathbb{Z}^{+}$ verifying \eqref{condmainT2} and \eqref{condmainT3}. Let $u_0\in \mathfrak{X}$ satisfying  \eqref{condi2}.  Assume $E[u]>0$. The following is a sufficient condition for the blow-up in finite time for the solutions to the gHartree equation \eqref{gHartree} with initial data $u_0$ in the mass-supercritical case ($s_c>0$ as in \eqref{critiindex}):
\begin{equation}\label{blowupcondi}
\begin{aligned}
\frac{\partial_{t}V(0)}{\omega_c M[u_0]}<4\sqrt{2}F\Big(\frac{E[u_0]V(0)}{(\omega M[u_0])^2}\Big),
\end{aligned}
\end{equation}
where $\omega_c^2=\frac{N^2(N(p-2)+N-\gamma-2)}{8(N(p-2)+N-\gamma)}$ and the function $F$ is defined as (here, $k_{c}=s_c(p-1)$)
\begin{equation}\label{eqblowupcondi}
F(x)=\left\{\begin{aligned}
&\sqrt{\frac{1}{k_c x^{k_{c}}}+x-\frac{1+k_{c}}{k_{c}}}, \, \, \text{ if } 0<x<1\\
&-\sqrt{\frac{1}{k_{c}x^{k_{c}}}+x-\frac{1+k_{c}}{k_{c}}}, \, \, \text{ if } x\geq 1.
\end{aligned} \right.
\end{equation}
\end{theorem}

In contrast with Theorem \ref{scatres}, one may look for initial data $u_0=e^{\frac{ib|x|^2}{2}}v_0$, $b < 0$, 
such that the corresponding solution of \eqref{gHartree} blows up in finite time. As a consequence of Theorem \ref{blowupcriteria}, we provide several initial conditions exhibiting this property.

\begin{corollary}\label{blowupcriteriacor}
Let $\max\{\frac{N+2}{N},\frac{4}{3}\}<p<2$, $0<\gamma<\min\{N(p-1)-2,\frac{(N+2)(p-1)-2}{2},\frac{N(3p-4)}{2p}\}$ and $\mu > 0$. 
Take  
\begin{equation*}
\max\Big\{\frac{N+2}{2},\frac{2\gamma+N}{4(p-1)}\Big\}< m<\frac{N-2\gamma}{2(2-p)},
\end{equation*} 
and $M_0,M \in \mathbb{Z}^{+}$ verifying \eqref{condmainT2} and \eqref{condmainT3}. Furthermore, take $k_{c}$ and $\omega_{c}$ as in the statement of Theorem \ref{blowupcriteria}. Let $v_0\in \mathfrak{X}$ be real-valued and satisfy \eqref{condi2}. 
\begin{itemize}
\item[(i)] Assume
\begin{equation}
\frac{E[v_0]\|xv_0\|_{L^2}^2}{(M[v_0])^2} < \omega_c^2
\end{equation}
and 
\begin{equation}
\begin{aligned}
E[v_0]\|xv_0\|_{L^2}^2<\bigg(\frac{(\omega_c M[v_0])^{2k_{c}+2}}{\big((1+k_{c})(\omega_c M[v_0])^2-k_{c}E[v_0]\|xv_0\|_{L^2}^2\big)}\bigg)^{1/k_{c}}.
\end{aligned}
\end{equation}
Then there exists $b_1>b_0\geq 0$ such that for all $b_0<b< b_1$ the solution $u$ of \eqref{gHartree} associated to $u_0=e^{\frac{ib|x|^2}{2}}v_0$ blows up in finite time. In particular, if $E[v_0]>0$, one can take $b_0=0$.
\item[(ii)] Assume
\begin{equation}
\frac{E[v_0]\|xv_0\|_{L^2}^2}{(M[v_0])^2}<\frac{\omega_c^2(1+k_{c})}{k_{c}}.
\end{equation} 
Then there exist $b_1\leq 0$ such that for all $b \leq b_1$ the solution $u$ of \eqref{gHartree} with initial condition $u_0=e^{\frac{ib|x|^2}{2}}v_0$ blows up in finite time.
\end{itemize}
\end{corollary}

\begin{rem}
\normalfont \begin{itemize} 
\item[(i)] The function $\varphi(x)=\frac{a}{\langle x \rangle^m}$, $a \in \mathbb{C}\setminus\{0\}$, satisfies the hypothesis of Theorem \ref{mainTHM}.
\item[(ii)] Setting $\frac{N-2}{N+\gamma}<\frac{1}{p}<\frac{N}{N+\gamma}$, the results in \cite{MOROZ} establish that there exists a non-negative ground state solution of
\begin{equation}\label{groundequ}
\Delta\varphi+\varphi=c_{\gamma}(|x|^{-(N-\gamma)}\ast |\varphi|^{p})|\varphi|^{p-2}\varphi, \quad c_{\gamma}=\frac{\Gamma\big(\frac{N-\gamma}{2}\big)}{\Gamma(\frac{\gamma}{2})\pi^{N/2}2^{\gamma}|x|^{N-\gamma}},
\end{equation}
such that $\varphi\in H^1(\mathbb{R}^N)$ and for $p<2$,
\begin{equation}
\lim_{|x|\to \infty} (\varphi(x))^{2-p}|x|^{N-\gamma}=c_{\gamma}\|\varphi\|_{L^p}^{p}.
\end{equation}
One naturally looks for conditions such that $\langle x \rangle^{-\frac{(N-\gamma)}{2-p}}$ belongs to the class in Theorem \ref{mainTHM}. However, Theorem \ref{mainTHM} does not include the ground state solutions of the equation in \eqref{groundequ}.
\item[(iii)] The conclusion of Corollary \ref{blowupcriteriacor} is still true assuming $ \operatorname{Im} \int_{\mathbb{R}^N} \overline{v_0}(x \cdot \nabla v_0)\, dx =0$ instead of $v_0$ real-valued. See also Remark \ref{remblowup} for further conclusions of Theorem \ref{blowupcriteria}.
\end{itemize}
\end{rem}
\medskip

The paper is organized as follows: in the next section we recall some results on weighted spaces and fractional derivatives. The proof of the main Theorem \ref{mainTHM} is in Section \ref{S:main}. The global well-posedness and scattering result is discussed in Section \ref{S:scat}, and the blow-up conditions are proved in Section \ref{S:blowup}.  
\medskip

{\bf Notation.}
Given two positive quantities $a$ and $b$,  $a\lesssim b$ means that there exists a positive constant $c>0$ such that $a\leq c b$. We write $a\sim b$ to symbolize that $a\lesssim b$ and $b\lesssim a$. We use the standard multi-index notation, $\alpha=(\alpha_1,\dots,\alpha_N)\in \mathbb{N}^{N}$, $\partial^{\alpha}=\partial_{x_1}^{\alpha_1}\cdots \partial_{x_N}^{\alpha_N}$, $|\alpha|=\sum_{j=1}^N |\alpha_j|$, $\alpha \leq \beta$, whenever $\alpha_j \leq \beta_j$ for all $1\leq j \leq N$. 
\\ \\
Let $1\leq r \leq \infty$, $L^r(\mathbb{R}^N;d\omega(x))$ denote the weighted Lebesgue space defined by the norm
\begin{equation}\label{eqnot1}
\|f\|_{L^{r}(\mathbb{R}^N;d\omega(x))}^r=\int_{\mathbb{R}^N}|f(x)|^r \, d\omega(x),
\end{equation} 
with the respective modifications for the case $r=\infty$. It is typical to denote weighted spaces as $L^{p}(\omega)$ or $L^{p}(\omega^r)$, but for clarity, we use notation as in \eqref{eqnot1}. The Fourier transform and the inverse Fourier transform of a function $f$ are denoted by $\widehat{f}$ and $f^{\vee}$, respectively.  For $s\in \mathbb{R}$, the Bessel potential of order $-s$ is denoted by $J^s=(1-\Delta)^{s/2}$, equivalently, $J^s$ is defined by the Fourier multiplier with symbol $\langle \xi \rangle^{s/2}=(1+|\xi|^2)^{s/2}$.  The Riesz potential of order $-s$ is denoted $D^s=(-\Delta)^{s/2}$, that is, $D^s$ is the Fourier multiplier operator determined by the function $|\xi|^s$. Given $b\in (0,1)$, we will also be using one of the Stein's derivatives $\mathcal{D}^s$ (see \eqref{Steinsquare} below). The Sobolev spaces $H^s(\mathbb{R}^N)$ consist of all tempered distributions such that $\left\|f\right\|_{H^s}=\left\|J^{s}f \right\|_{L^2}<\infty$. Given an integer $l\geq 0$, the space $H^{l}(\mathbb{R}^N;d\omega(x))$ denotes the weighted Sobolev space defined by the norm
\begin{equation*}
\|f\|_{H^{l}(\mathbb{R}^N;d\omega(x))}^2=\sum_{|\beta|\leq l}\int_{\mathbb{R}^{N}} |\partial^{\beta}f(x)|^2\, d\omega(x).
\end{equation*}
Let $1\leq r \leq \infty$ and $T>0$, if $A$ denotes a function space (as those introduced above), we define the spaces $L^r_TA$ and $L^r_t A$ as follows
\begin{equation*}
\begin{aligned}
\|f\|_{L^{r}_T A_x}&=\Big(\int_0^T \|f(\cdot,t)\|_{A}^p\, dt \Big)^{1/r},\\
\|f\|_{L^{r}_t A_x}&=\Big(\int_{\mathbb{R}} \|f(\cdot,t)\|_{A}^r\, dt \Big)^{1/r}.
\end{aligned}
\end{equation*}



\textbf{Acknowledgements.} O.R. and S.R. were partially supported by the NSF grant DMS-1927258 (PI: Roudenko).


 
\section{Preliminary estimates}
 
This section aims to provide some results relating weighted spaces and fractional derivatives. Additionally, we  present some estimates in weighted spaces for solutions of the linear equation associated to \eqref{gHartree}.

\subsection{Preliminaries on weighted spaces and fractional derivatives}

Since in this subsection we will not provide estimates for the nonlinear problem \eqref{gHartree}, a real number $p\geq 1$ will be used for some preliminaries presented in this part. Thus, we say that a non-negative function $f \in L^{1}_{loc}(\mathbb{R}^N)$ satisfies the classical Muckenhoupt $A_p$ condition with $1<p<\infty$ if 
\begin{equation}\label{Apcond}
\begin{aligned}
\left[f\right]_{A_p}=\sup_{Q} \bigg( \dashint_Q f(x) \, dx \bigg) \bigg( \dashint_Q f(x)^{1-p'} \, dx \bigg)^{p-1}<\infty,
\end{aligned}
\end{equation}
where the supremum runs over cubes in $\mathbb{R}^d$ and $\frac{1}{p}+\frac{1}{p'}=1$ (see \cite{MuckApcond}). In particular, we have that (see \cite{JavierHarmo,stein1993harmonic}) 
\begin{equation}\label{weighAp}
|x|^{l} \in A_{p} \, \text{ if and only if } \, l\in(-N,N(p-1)).
\end{equation}

Motivated by the nonlinear term in \eqref{gHartree}, we will need continuity properties of Riesz potentials in weighted spaces. 
\begin{theorem}\label{weightriesz} (\cite{Rieszweigh2,LACEY,Rieszweigh})
Given $0<\gamma<N$ and $1<p<\frac{N}{\gamma}$, let
\begin{equation}\label{eqweightriesz0.1}
\frac{1}{q}=\frac{1}{p}-\frac{\gamma}{N}. 
\end{equation}
Then the following are equivalent:
\begin{itemize}
\item[(i)] $\omega \in A_{p,q}$:
\begin{equation}\label{Apqcond}
\left[\omega\right]_{A_{p,q}}=\sup_{Q} \bigg( \dashint_Q \omega(x)^q \, dx \bigg)^{1/q} \bigg( \dashint_Q \omega(x)^{-p'} \, dx \bigg)^{1/p'}<\infty,
\end{equation}
where the supremum is taken over any $N$-dimensional cube $Q$ and $\frac{1}{p}+\frac{1}{p'}=1$. 
\item[(ii)] The following inequality holds true
\begin{equation*}
\begin{aligned}
\|\big(|\cdot|^{-(N-\gamma)}\ast f\big)\omega\|_{L^q}\lesssim \|f \omega\|_{L^p},
\end{aligned}
\end{equation*}
for any $f\in L^{p}(\mathbb{R}^{N};\omega^{p}\, dx)$, where the implicit constant is independent of $f$. 
\end{itemize}
\end{theorem}
Since we are interested in applying Theorem \ref{weightriesz} for the case $\omega=\langle x \rangle^{l}$, we have to verify the relations between the indices $l,p,q$ assuring that $\langle x \rangle^{l} \in A_{p,q}$. We can rephrase this last property in terms of the $A_p$ condition. But first, taking $p$ and $q$ as in Theorem \ref{weightriesz}, we set 
\begin{equation*}
p^{\ast}=1+\frac{q}{p'}=q\Big(1-\frac{\gamma}{N}\Big)=p\Big(\frac{N-\gamma}{N-\gamma p}\Big).
\end{equation*}
Then, by \eqref{Apcond}, \eqref{weighAp} and \eqref{Apqcond}, we deduce
\begin{equation}\label{eqweightriesz2}
\begin{aligned}
|x|^{l}\in A_{p,q} \, \, &\text{ if and only if } \, \, |x|^{lq}\in A_{p^{\ast}} \\
&\text{ if and only if } \, \, -\frac{N-p\gamma}{p}=-\frac{N}{q}<l<\frac{N(p-1)}{p}. 
\end{aligned}
\end{equation}
Summarizing the previous discussion, we have
\begin{prop}\label{propweightriesz}
Let $0<\gamma<N$, $1<p<\frac{N}{\gamma}$,
\begin{equation}\label{eqweightriesz0}
\frac{1}{q}=\frac{1}{p}-\frac{\gamma}{N}, \, \text{ and } \, -\frac{N-p\gamma}{p}<l<\frac{N(p-1)}{p}.
\end{equation}
Then 
\begin{equation*}
\|\big(|\cdot|^{-(N-\gamma)}\ast f\big)\langle x \rangle^{l}\|_{L^q}\lesssim \|f \langle x \rangle^{l}\|_{L^p},
\end{equation*}
for any $f\in L^{p}(\mathbb{R}^N;\omega^{p}\, dx)$, where the implicit constant is independent of $f$. 
\end{prop}
We remark that Proposition \ref{propweightriesz} can also be deduced as a consequence of the results due to Stein and Weiss \cite{SGU}. Next, let $b\in (0,1)$, we define one of the Stein's fractional derivatives
\begin{equation}\label{Steinsquare}
\mathcal{D}^b f(x)=\left(\int_{\mathbb{R}^N}\frac{|f(x)-f(y)|^2}{|x-y|^{N+2b}}\, dy\right)^{1/2}, \, \,  x\in \mathbb{R}^N,
\end{equation}
for a sufficiently regular function $f$ (for instance, a Schwartz function). To deal with fractional weights, we recall the following characterization of the spaces $L^p_s(\mathbb{R}^N)=J^{-s}L^p(\mathbb{R}^N)$.
\begin{theorem}( \cite{SteinThe})\label{TheoSteDer} 
Let $b\in (0,1)$ and $\frac{2N}{N+2b}<p<\infty$. Then $f\in L_b^p(\mathbb{R}^N)$ if and only if
\begin{itemize}
\item[(i)]  $f\in L^p(\mathbb{R}^N)$ and
\item[(ii)]$\mathcal{D}^bf \in L^{p}(\mathbb{R}^N)$
\end{itemize}
with 
\begin{equation*}
\|J^b f\|_{L^p}=\|(1-\Delta)^{b/2} f \|_{L^p} \sim \|f\|_{L^p}+\|\mathcal{D}^b f\|_{L^p} \sim \|f\|_{L^p}+\|D^b f\|_{L^p}
\end{equation*} 
(recalling that $D^b=(-\Delta)^{b/2}$).
\end{theorem}
Let us recall some useful consequences of Theorem \ref{TheoSteDer}. When $p=2$ and $b\in (0,1)$, one can deduce 
\begin{equation} \label{prelimneq} 
\left\|\mathcal{D}^b(fg)\right\|_{L^2} \lesssim \left\|f\mathcal{D}^b g\right\|_{L^2}+\left\|g\mathcal{D}^bf \right\|_{L^2},
\end{equation}
and also
\begin{equation} \label{prelimneq1}
\left\|\mathcal{D}^{b} h\right\|_{L^{\infty}} \lesssim \left\|h\right\|_{L^{\infty}}+\left\|\nabla h\right\|_{L^{\infty}}.
\end{equation} 
We will apply the following interpolation result deduced in \cite[Lemma 4]{NahasPo}.
\begin{prop}\label{propweight2} For any $a,b>0$, $\theta\in (0,1)$,
\begin{equation}\label{eq1propweight2}
\|\langle x \rangle^{\theta a}\big(J^{(1-\theta)b}f\big)\|_{L^2}\lesssim \|J^{b}f\|^{1-\theta}_{L^2}\|\langle x \rangle^{a} f\|_{L^2}^{\theta},
\end{equation}
\begin{equation}\label{eq2propweight2}
\|J^{\theta a}\big(\langle x \rangle^{(1-\theta)b}f\big)\|_{L^2}\lesssim \|\langle x \rangle^{b}f\|^{1-\theta}_{L^2}\|J^{a} f\|_{L^2}^{\theta}.
\end{equation}
\end{prop}

\begin{rem}
\normalfont The inequality \eqref{eq1propweight2} is also valid in $L^p(\mathbb{R}^N)$, $1<p<\infty$, see \cite[Lemma 2.7]{LINARES2020}.
\end{rem}

We also require the following lemma relating weighted estimates between homogeneous and non-homogeneous derivatives.
\begin{lemma}\label{lemmachangingDbyJ}
Let $0<b<s$. Then it follows 
\begin{equation*}
\begin{aligned}
\|\langle x \rangle^{b}D^{s}f\|_{L^2} \lesssim \|\langle x \rangle^{b}f\|_{L^2}+\|J^{s-b}f\|_{L^2}+\|\langle x \rangle^{b}J^{s}f\|_{L^2}.
\end{aligned}
\end{equation*} 
\end{lemma}
\begin{proof}
We write $b=b_1+b_2$, where $b_1 \in \mathbb{Z}^{+}\cup\{0\}$ and $b_2\in [0,1)$. Then by Plancherel's identity and the linearity of the operator $D^{b_2}$, we get
\begin{equation}\label{eqlemmaopD1}
\begin{aligned}
\|J^{b}_{\xi}\big(|\xi|^{s}\widehat{f} \, \big)\|_{L^2}\lesssim \sum_{0\leq |\beta|\leq b_1}\|\partial^{\beta}\big(|\xi|^{s}\widehat{f} \, \big)\|_{L^2}+\|D^{b_2}\big(\partial^{\beta}\big(|\xi|^{s}\widehat{f} \, \big)\big)\|_{L^2}.
\end{aligned}
\end{equation}
When $b_2=0$, we will assume that $D^{b_2}$ is the identity operator. Let us bound each term on the right-hand side of the above inequality. By Leibniz's rule, we have
\begin{equation*}
\begin{aligned}
\sum_{|\beta| \leq b_1}\|\partial^{\beta}\big(|\xi|^{s}\widehat{f} \,\big)\|_{L^2}=&\sum_{|\beta| \leq b_1}\|\partial^{\beta}\big(\frac{|\xi|^{s}}{\langle \xi \rangle^{s}}\langle \xi \rangle^{s}\widehat{f} \, \big)\|_{L^2}\\
\lesssim &\sum_{|\beta| \leq b_1}\sum_{\beta_1+\beta_2=\beta}\|\partial^{\beta_1}\big(\frac{|\xi|^{s}}{\langle \xi \rangle^{s}}\big)\|_{L^{\infty}}\|\partial^{\beta_2}\big(\langle \xi \rangle^{s}\widehat{f} \, \big)\|_{L^2}\\
\lesssim & \|J^{b_1}_{\xi}\big(\langle \xi \rangle^{s}\widehat{f} \, \big)\|_{L^2},
\end{aligned}
\end{equation*}
where given that $s-b>0$, we have used that 
\begin{equation*}
\|\partial^{\beta_1}\big(\frac{|\xi|^{s}}{\langle \xi \rangle^{s}}\big)\|_{L^{\infty}} \lesssim 1, 
\end{equation*}
for all $|\beta_1|\leq b_1$. To deal with the second term on the right-hand side of \eqref{eqlemmaopD1}, we consider a real-valued smooth function $\phi$, compactly supported in the set $|\xi|\leq 2$, such that $\phi(\xi)=1$ for $|\xi|\leq 1$. Then, we split that term as follows
\begin{equation*}
\begin{aligned}
\|D^{b_2}\big(\partial^{\beta}\big(|\xi|^{s}\widehat{f} \, \big)\big)\|_{L^2}=&\|D^{b_2}\big(\partial^{\beta}\big(|\xi|^{s}\phi\widehat{f} \, \big)\big)\|_{L^2}+\|D^{b_2}\big(\partial^{\beta}\big(|\xi|^{s}(1-\phi)\widehat{f} \, \big)\big)\|_{L^2}\\
=:&\mathcal{A}_{1}(\beta)+\mathcal{A}_2(\beta),
\end{aligned}
\end{equation*}
for each $|\beta|\leq b_1$. We apply property \eqref{prelimneq} and Theorem \ref{TheoSteDer} to get
\begin{equation*}
\begin{aligned}
\mathcal{A}_2(\beta) \lesssim &  \sum_{\beta_1+\beta_2=\beta} \|D^{b_2}\big(\partial^{\beta_1}\big(\frac{|\xi|^{s}(1-\phi)}{\langle \xi \rangle^{s}}\big)\partial^{\beta_2}\big(\langle \xi \rangle^{s}\widehat{f} \, \big)\big)\|_{L^2}\\
\lesssim & \sum_{\beta_1+\beta_2=\beta} \Big(\|\partial^{\beta_1}\big(\frac{|\xi|^{s}(1-\phi)}{\langle \xi \rangle^{s}}\big)\|_{L^{\infty}}+\|\mathcal{D}^{b_2}_{\xi}\big(\partial^{\beta_1}\big(\frac{|\xi|^{s}(1-\phi)}{\langle \xi \rangle^{s}}\big)\big)\|_{L^{\infty}}\Big)\\
&\hspace{5cm} \times \big(\|\partial^{\beta_2}\big(\langle \xi \rangle^{s}\widehat{f}\big)\|_{L^2}+\|\mathcal{D}^{b_2}_{\xi}\big(\partial^{\beta_2}\big(\langle \xi \rangle^{s}\widehat{f} \, \big)\big)\|_{L^2}\big)\\
\lesssim & \|J_{\xi}^{b_1+b_2}\big(\langle \xi \rangle^{s}\widehat{f} \, \big)\|_{L^2},
\end{aligned}
\end{equation*}
where we have used property \eqref{prelimneq1} to deduce
\begin{equation*}
\begin{aligned}
\|\mathcal{D}_{\xi}^{b_2}\big(\partial^{\beta_1}\big(\frac{|\xi|^{s}(1-\phi)}{\langle \xi \rangle^{s}}\big)\big)\|_{L^{\infty}}\lesssim \|\partial^{\beta_1}\big(\frac{|\xi|^{s}(1-\phi)}{\langle \xi \rangle^{s}}\big)\|_{L^{\infty}}+\|\nabla\big(\partial^{\beta_1}\big(\frac{|\xi|^{s}(1-\phi)}{\langle \xi \rangle^{s}}\big)\big)\|_{L^{\infty}}<\infty.
\end{aligned}
\end{equation*}
We remark that when $b_2=0$, we will assume that $\mathcal{D}^{b_2}_{\xi}$ denotes the identity operator. This completes the study of $\mathcal{A}_2(\beta)$, whenever $|\beta|\leq b_1$. Next, we turn to $\mathcal{A}_1(\beta)$. Assuming that $|\beta| \geq 1$, by Theorem \ref{TheoSteDer}, \eqref{prelimneq} and \eqref{prelimneq1}, we have
\begin{equation*}
\begin{aligned}
\mathcal{A}_1(\beta)\lesssim &\sum_{\substack{\beta_1+\beta_2+\beta_3=\beta \\ \beta_1 \neq \beta}}\|D^{b_2}\big(\partial^{\beta_1}(|\xi|^{s})\partial^{\beta_2}\phi\partial^{\beta_3}\widehat{f} \, \big)\|_{L^2}+\|D^{b_2}\big(\partial^{\beta}(|\xi|^{s})\phi\widehat{f} \, \big)\|_{L^2}\\
\lesssim & \sum_{\substack{\beta_1+\beta_2+\beta_3=\beta \\ \beta_1 \neq \beta}} \Big(\|\partial^{\beta_1}(|\xi|^{s})\partial^{\beta_2}\phi\|_{L^{\infty}}+\|\nabla \partial^{\beta_1}(|\xi|^{s})\partial^{\beta_2}\phi\|_{L^{\infty}}\Big)\big(\|\partial^{\beta_3}\widehat{f} \, \|_{L^2}+\|\mathcal{D}^{b_2}\partial^{\beta_3}\widehat{f} \, \|_{L^2}\big)\\
& \qquad \quad +\|D^{b_2}\big(\partial^{\beta}(|\xi|^{s})\phi\widehat{f}\,\big)\|_{L^2}\\
\lesssim &\|J_{\xi}^{b_1+b_2}\widehat{f}\, \|_{L^2}+\|D^{b_2}\big(\partial^{\beta}(|\xi|^{s})\phi\widehat{f}\,\big)\|_{L^2}.
\end{aligned}
\end{equation*}
Notice that the conclusion of the above inequality is also valid for $\beta=0$.  Now, if $|\beta|<b_1$ (assuming that $b_1>0$), it follows that $\nabla\big(\partial^{\beta}(|\xi|^{s})\phi\big)\in L^{\infty}(\mathbb{R}^N)$. Then the estimate for $\|D^{b_2}\big(\partial^{\beta}(|\xi|^{s})\phi\widehat{f}\, \big)\|_{L^2}$ follows from the same arguments as above. Otherwise, if $|\beta|=b_1\geq 0$, we require the following claim.
\begin{claim}\label{claimSderiv}
Let  $b_1\in \mathbb{Z}^{+}\cup\{0\}$,  $b_2\in [0,1)$, and $s>b_1+b_2$. Then for any multi-index $\beta$ with $|\beta|=b_1$, it follows that 
\begin{equation*}
\begin{aligned}
\mathcal{D}_{\xi}^{b_2}\big(\partial^{\beta}(|\xi|^{s})\phi(\xi) \big)\lesssim \langle \xi \rangle^{s-b_1-b_2}, \, \, \xi \neq 0,
\end{aligned}
\end{equation*}
when $b_2=0$, we assume that $\mathcal{D}_{\xi}^{b_2}$ is the identity operator.
\end{claim}
Before proving Claim \ref{claimSderiv}, let us complete the estimate for the case $|\beta|=b_1$ in $\mathcal{A}_1(\beta)$. Indeed, by Theorem \ref{TheoSteDer}, property \eqref{prelimneq}, and Claim \ref{claimSderiv}, we get
\begin{equation*}
\begin{aligned}
\|D^{b_2}\big(\partial^{\beta}(|\xi|^{s})\phi\widehat{f} \, \big)\|_{L^2} \lesssim &\|\partial^{\beta}(|\xi|^{s})\phi\widehat{f} \, \|_{L^2}+\|\mathcal{D}^{b_2}_{\xi}\big(\partial^{\beta}(|\xi|^{s})\phi\big)\widehat{f} \, \|_{L^2}+\|\partial^{\beta}(|\xi|^{s})\phi\mathcal{D}_{\xi}^{b_2}(\widehat{f} \, ) \|_{L^2} \\
\lesssim & \|J^{b_2}_{\xi}\widehat{f} \, \|_{L^2}+\|\langle \xi \rangle^{s-b_1-b_2} \widehat{f} \, \|_{L^2}.
\end{aligned}
\end{equation*}
Collecting the estimates for $\mathcal{A}_1(\beta)$ and $\mathcal{A}_2(\beta)$, $|\beta|\leq b_1$, and reversing the Fourier variables, we complete the proof of the lemma.
\end{proof}

\begin{proof}[Proof of Claim \ref{claimSderiv}]
At once we deduce the case $b_2=0$, hence, we assume $b_2>0$. To simplify the exposition of our arguments, we denote by $F(\xi):=\partial^{\beta}(|\xi|^{s})\phi(\xi)$. By the definition of the fractional derivative \eqref{Steinsquare}, we divide our considerations in the following manner
\begin{equation*}
\begin{aligned}
\mathcal{D}_{\xi}^{b_2}\big(F(\xi) \big)^2=\int_{\mathbb{R}^N} \frac{|F(\xi)-F(y)|^2}{|\xi-y|^{N+2b_2}}\, dy=&\int_{|y|\leq 2|\xi|}(\cdots)\, dy+\int_{|y|\geq 2|\xi|}(\cdots)\, dy\\
=:&\mathcal{B}_1+\mathcal{B}_2.
\end{aligned}
\end{equation*}
In the support of the integral $\mathcal{B}_2$, we have $|y|\gg |\xi|$ and so $|\xi-y|\sim |y|$. Then, since $\phi(y)$ is supported in the set $|y|\leq 2$, we get
\begin{equation*}
\begin{aligned}
\mathcal{B}_2 \lesssim & \int_{|y|\geq 2|\xi|} \frac{|F(\xi)|^2}{|y|^{N+2b_2}} \, dy+\int_{|y|\geq 2|\xi|} \frac{|F(y)|^2}{|y|^{N+2b_2}} \, dy \\
\lesssim & \int_{|y|\geq 2|\xi|} \frac{|\xi|^{2s-2b_1}}{|y|^{N+2b_2}} \, dy+\int_{ 2|\xi|\leq |y|\leq 2} \frac{|y|^{2s-2b_1}}{|y|^{N+2b_2}} \, dy\\
\lesssim & \langle \xi \rangle^{2s-2b_1-2b_2},
\end{aligned}
\end{equation*}
where we have also used that $|\partial^{\beta}(|\xi|^{s})|\lesssim |\xi|^{s-b_1}$, $|\beta|=b_1$. Next, in virtue of the identity
\begin{equation*}
\begin{aligned}
\partial^{\beta}(|\xi|^{s})=|\xi|^{s-2|\beta|}P_{|\beta|}(\xi), \qquad  \xi \in \mathbb{R}^{N}\setminus\{0\},
\end{aligned}
\end{equation*}
where $P_{|\beta|}(\xi)$ denotes a homogeneous polynomial of degree $|\beta|$, we write
\begin{equation*}
\begin{aligned}
F(\xi)-F(y)=&\big(|\xi|^{s-2|\beta|}P_{|\beta|}(\xi)-|y|^{s-2|\beta|}P_{|\beta|}(y) \big)\phi(\xi)+|y|^{s-2|\beta|}P_{|\beta|}(y)\big(\phi(\xi)-\phi(y)\big)\\
=&|\xi|^{s-2|\beta|}(P_{|\beta|}(\xi)-P_{|\beta|}(y))\phi(\xi)+|\xi|^{s-2|\beta|}|y|^{-|\beta|}P_{|\beta|}(y)(|y|^{|\beta|}-|\xi|^{|\beta|})\phi(\xi)\\
&+|y|^{-|\beta|}P_{|\beta|}(y)(|\xi|^{s-|\beta|}-|y|^{s-|\beta|})\phi(\xi)+|y|^{s-2|\beta|}P_{|\beta|}(y)\big(\phi(\xi)-\phi(y)\big).
\end{aligned}
\end{equation*}
Then, since $|y|\leq 2|\xi|$ in the support of the integral defining $\mathcal{B}_1$, we have
\begin{equation*}
\begin{aligned}
|F(\xi)&-F(y)|\\
\lesssim &\big(|\xi|^{s-|\beta|-1}|\xi-y|+||\xi|^{s-|\beta|}-|y|^{s-|\beta|}|\big)\|\phi\|_{L^{\infty}}+|y|^{s-|\beta|}|\phi(\xi)-\phi(y)| \\
\lesssim & \left\{ \begin{array}{ll}
\big|\xi|^{s-|\beta|-1}|\xi-y|+ |\xi-y|^{s-|\beta|}+|\xi|^{s-|\beta|}|\phi(\xi)-\phi(y)|, & \text{ if } 0<s-|\beta| \leq 1, \\
|\xi|^{s-|\beta|-1}|\xi-y|+|\xi|^{s-|\beta|}|\phi(\xi)-\phi(y)|, & \text{ if } s-|\beta|>1.
\end{array} \right.
\end{aligned}
\end{equation*}
Hence, we arrive at
\begin{equation*}
\begin{aligned}
\mathcal{B}_2 \lesssim &|\xi|^{2s-2|\beta|-2} \int_{|y|\lesssim 2|\xi|}|\xi-y|^{2-N-2b_2}\, dy+\int_{|y|\lesssim 2|\xi|}|\xi-y|^{-N+2s-2|\beta|-2b_2}\, dy \\
&+|\xi|^{2s-2|\beta|} \big(\mathcal{D}_{\xi}^{b_2}(\phi)\big)^2(\xi)\\
\lesssim & \langle \xi\rangle^{2s-2|\beta|-2b_2}.
\end{aligned}
\end{equation*}
The proof is complete.
\end{proof}


\subsection{Some weighted estimates for the linear equation}

This part is intended to provide weighted $L^2$ estimates for solutions of the homogeneous Schr\"odinger equation. We begin by recalling the following result. 
\begin{prop}\label{derivexp} 
Let $b\in(0,1)$. For any $|t|>0$,
\begin{equation*}
\mathcal{D}^b\big(e^{it|x|^2}\big)\lesssim |t|^{b/2}+|t|^b|x|^b.
\end{equation*}
\end{prop}
For a proof of Proposition \ref{derivexp}, see \cite[Proposition 2]{NahasPo}. We also require an extension to weights of arbitrary size of the result obtained by Nahas and Ponce in \cite[Lemma 2]{NahasPo}. For the sake of completeness, we provide the proof of this extension. 
\begin{lemma}\label{derivexp2}
Let $b \in \mathbb{R}^{+}$. Then for any $t\in \mathbb{R}$
\begin{equation*}
\|\langle x \rangle^{b}e^{it\Delta}f\|_{L^2} \lesssim \langle t \rangle^b\big( \|J^{b}f\|_{L^2}+\|\langle x\rangle^b f\|_{L^2} \big),
\end{equation*}
where the implicit constant above is independent of $t$.
\end{lemma}

\begin{proof}
The case $b\in (0,1)$  is proved in \cite[Lemma 2]{NahasPo}. Hence, we assume that $b\geq 1$. We write $b=b_1+b_2$, where $b_1\in \mathbb{Z}^{+}$, $b_2\in [0,1)$. Then by Plancherel's identity and Leibniz's rule, we get 
\begin{equation*}
\begin{aligned}
\|\langle x\rangle^{b}e^{it\Delta}f\|_{L^2}\lesssim & \|e^{it\Delta}f\|_{L^2}+\||x|^{b_2}|x|^{b_1}e^{it\Delta}f\|_{L^2} \\
\lesssim & \|f\|_{L^2}+\sum_{j=1}^N\|D^{b_2}\big(\partial_{\xi_j}^{b_1}(e^{-it|\xi|^2}\widehat{f} \, )\big)\|_{L^2}\\
\lesssim & \|f\|_{L^2}+\sum_{j=1}^N\sum_{k=0}^{b_1}\|D^{b_2}\big(\partial_{\xi_j}^{k}(e^{-it|\xi|^2})\partial_{\xi_j}^{b_1-k}\widehat{f} \, )\big)\|_{L^2}.
\end{aligned}
\end{equation*}
Thus, we are reduced to control the second term on the right-hand side of the above inequality. For a given multi-index $\beta$, we use the following straightforward identity
\begin{equation*}
\begin{aligned}
\partial^{\beta}(e^{-it|\xi|^2})=e^{-it|\xi|^2}\sum_{l=0}^{\lfloor|\beta|/2\rfloor}t^{|\beta|-l}P_{|\beta|-2l}(\xi),
\end{aligned}
\end{equation*}
where $P_{l}(\xi)$ denotes a homogeneous polynomial of order $l \geq 0$ in the variables $\xi \in \mathbb{R}^N$.  Thus, by Theorem \ref{TheoSteDer}, property \eqref{prelimneq} and Proposition \ref{derivexp}, we have
\begin{equation}\label{dimNeq1}
\begin{aligned}
\sum_{j=1}^N\sum_{k=0}^{b_1} &\|D^{b_2}\big(\partial_{\xi_j}^{k}(e^{-it|\xi|^2})\partial_{\xi_j}^{b_1-k}\widehat{f} \, ) \big)\|_{L^2}\\
\lesssim & \sum_{j=1}^N\sum_{k=0}^{b_1}\sum_{l=0}^{\lfloor k/2 \rfloor}\langle t \rangle^{b_1}\|D^{b_2}\big(e^{-it|\xi|^2}P^j_{k-2l}(\xi)\partial_{\xi_j}^{b_1-k}\widehat{f} \, \big)\|_{L^2} \\
\lesssim & \sum_{j=1}^N\sum_{k=0}^{b_1}\sum_{l=0}^{\lfloor k/2 \rfloor}\langle t \rangle^{b_1+b_2}\Big(\|P^j_{k-2l}(\xi)\partial_{\xi_j}^{b_1-k}\widehat{f} \, \|_{L^2}+\||\xi|^{b_2}P^j_{k-2l}(\xi)\partial_{\xi_j}^{b_1-k}\widehat{f} \, \|_{L^2}\\
& \hspace{7cm}+\|\mathcal{D}^{b_2}_{\xi}\big(P^j_{k-2l}(\xi)\partial_{\xi_j}^{b_1-k}\widehat{f}\big)\|_{L^2} \Big)\\
\lesssim & \sum_{j=1}^N\sum_{k=0}^{b_1}\sum_{l=0}^{\lfloor k/2 \rfloor}\langle t \rangle^{b_1+b_2}\Big(\|\langle \xi\rangle^{b_2+k-2l}\partial_{\xi_j}^{b_1-k}\widehat{f} \, \|_{L^2}+\|D^{b_2}\big(\langle\xi\rangle^{k-2l}\partial_{\xi_j}^{b_1-k}\widehat{f}\, \big)\|_{L^2} \Big),
\end{aligned}
\end{equation}
where we have used properties \eqref{prelimneq} and \eqref{prelimneq1} to deduce 
\begin{equation*}
\begin{aligned}
\|\mathcal{D}^{b_2}_{\xi}\big(P^j_{k-2l}(\xi)&\partial_{\xi_j}^{b_1-k}\widehat{f} \, \big)\|_{L^2}\\
&=\|\mathcal{D}^{b_2}_{\xi}\big(\langle\xi \rangle^{-k+2l}P^j_{k-2l}(\xi)\langle\xi \rangle^{k-2l}\partial_{\xi_j}^{b_1-k}\widehat{f} \, \big)\|_{L^2}\\
&\lesssim \Big(\|\langle\xi \rangle^{-k+2l}P^j_{k-2l}(\xi)\|_{L^{\infty}}+\|\mathcal{D}^{b_2}_{\xi}\big(\langle\xi \rangle^{-k+2l} P^j_{k-2l}(\xi)\big)\|_{L^{\infty}}  \Big)\|\mathcal{D}^{b_2}_{\xi}\big(\langle\xi \rangle^{k-2l}\partial_{\xi_j}^{b_1-k}\widehat{f}\, \big)\|_{L^2}.
\end{aligned}
\end{equation*}
At this point the very last term could be written via the term of the fractional operator $D^{b_2}$. As before, when $b_2=0$, we will assume that $\mathcal{D}^{b_2}_{\xi}$ and $D^{b_2}$ are the identity operator. Before estimating \eqref{dimNeq1}, let us deduce a more general inequality involving weights and derivatives. Let $b' \in [0,1)$, $m' \in \mathbb{R}^{+}\cup \{0\}$ and $k' \in \mathbb{Z}^{+}\cup \{0\}$, we claim  
\begin{equation}\label{dimNeq2}
\begin{aligned}
\|D^{b'} \big(\langle \xi \rangle^{m'}\partial_{\xi_j}^{k'}\widehat{f} \, \big)\|_{L^2} \lesssim  \sum_{\substack{0\leq l\leq k' \\ l<m'}} \|J^{b'+k'-l}_{\xi}\big(\langle \xi \rangle^{m'-l}\widehat{f} \, \big)\|_{L^2}+\|J^{b'+k'-m'}_{\xi}\widehat{f} \, \|_{L^2}.
\end{aligned}
\end{equation} 
Indeed, let us deduce \eqref{dimNeq2}.  By applying the identity
\begin{equation*}
\begin{aligned}
\langle \xi \rangle^{m'}\partial_{\xi_j}^{k'}\widehat{f}(\xi)=\sum_{l=0}^{k'} c_l \partial_{\xi_j}^{k'-l}\big(\partial_{\xi_j}^{l}\langle \xi \rangle^{m'} \widehat{f}(\xi) \big),
\end{aligned}
\end{equation*}
we split the left-hand side of \eqref{dimNeq2} as follows
\begin{equation*}
\begin{aligned}
\|D^{b'} \big(\langle \xi \rangle^{m'}\partial_{\xi_j}^{k'}\widehat{f} \, \big)\|_{L^2}\lesssim & \sum_{\substack{ 0\leq l \leq k' \\ l < m'}}\|D^{b'}\partial_{\xi_j}^{k'-l}\big(\frac{\partial_{\xi_j}^{l}\langle \xi \rangle^{m'}}{\langle \xi \rangle^{m'-l}}\langle \xi \rangle^{m'-l}\widehat{f} \, \big)\|_{L^2}+\sum_{\substack{m'\leq l \leq k' }}\|D^{b'}\partial_{\xi_j}^{k'-l}\big(\partial_{\xi_j}^{l}\langle \xi \rangle^{m'}\widehat{f} \, \big)\|_{L^2}\\
=& \,  \mathcal{I}+\mathcal{II}.
\end{aligned}
\end{equation*}
We proceed to estimate each term of the above expression.  Theorem \ref{TheoSteDer} and \eqref{prelimneq1} yield
\begin{equation*}
\begin{aligned}
\mathcal{I}\lesssim & \sum_{\substack{0\leq l \leq k' \\ l<m'}} \sum_{l'=0}^{k'-l}\|D^{b'}\Big(\partial_{\xi_j}^{l'}\big(\frac{\partial_{\xi_j}^{l}\langle \xi \rangle^{m'}}{\langle \xi \rangle^{m'-l}} \big)\partial_{\xi_j}^{k'-l-l'}\big(\langle \xi \rangle^{m'-l}\widehat{f} \, \big)\Big)\|_{L^2} \\
\lesssim & \sum_{\substack{0\leq l \leq k' \\ l<m'}} \sum_{l'=0}^{k'-l}\Big(\|\partial_{\xi_j}^{l'}\big(\frac{\partial_{\xi_j}^{l}\langle \xi \rangle^{m'}}{\langle \xi \rangle^{m'-l}} \big)\|_{L^{\infty}}+\|\mathcal{D}^{b'}_{\xi}\Big(\partial_{\xi_j}^{l'}\big(\frac{\partial_{\xi_j}^{l}\langle \xi \rangle^{m'}}{\langle \xi \rangle^{m'-l}} \big)\Big)\|_{L^{\infty}}\Big)\|J^{b'+k'-l-l'}_{\xi}\big(\langle \xi \rangle^{m'-l}\widehat{f} \, \big)\|_{L^2} \\
\lesssim & \sum_{\substack{0\leq l \leq k' \\ l<m'}} \|J^{b'+k'-l}_{\xi}\big(\langle \xi \rangle^{m'-l}\widehat{f} \, \big)\|_{L^2}.
\end{aligned}
\end{equation*}
By a similar argument, we get
\begin{equation*}
\begin{aligned}
\mathcal{II}\lesssim  \sum_{\substack{m'\leq l \leq k' }}\sum_{l'=0}^{k'-l}\|D^{b'}\big(\partial_{\xi_j}^{l+l'}\langle \xi \rangle^{m'} \partial_{\xi_j}^{k'-l-l'}\widehat{f} \, \big)\|_{L^2}\lesssim \|J^{b'+k'-m'}_{\xi}\widehat{f} \, \|_{L^2}.
\end{aligned}
\end{equation*}
Collecting the above estimates, we verify \eqref{dimNeq2}.

Now, let us estimate the first sum on the right-hand side of \eqref{dimNeq1}. Since the cases $k=b_1$ and $b_2+k-2l=0$ are easily verified, we assume $0\leq k \leq b_1-1$, $0\leq l \leq \lfloor k/2\rfloor$ and $b_2+k-2l>0$. Thus, these assumptions, \eqref{dimNeq2} and Proposition \ref{propweight2} yield
\begin{equation}\label{dimNeq3}
\begin{aligned}
\|\langle \xi\rangle^{b_2+k-2l}\partial_{\xi_j}^{b_1-k}\widehat{f} \, \|_{L^2} \lesssim & \sum_{\substack{0\leq l' < b_1-k\\ l' < b_2+k-2l}}\|J^{b_1-k-l'}_{\xi}\big(\langle \xi\rangle^{b_2+k-2l-l'}\widehat{f} \big)\|_{L^2}+\|\langle \xi \rangle^{b_1+b_2}\widehat{f} \, \|_{L^2}+\|J^{b_1+b_2}_{\xi}\widehat{f} \, \|_{L^2}\\
\lesssim & \sum_{\substack{0\leq l' < b_1-k\\ l' < b_2+k-2l}} \|\langle \xi \rangle^{b_1+b_2}\widehat{f} \, \|^{\frac{b_2+k-2l-l'}{b_1+b_2}}_{L^2}\|J^{\frac{(b_1-k-l')(b_1+b_2)}{b_1-k+2l+l'}}_{\xi}\widehat{f} \, \|^{\frac{b_1-k+2l+l'}{b_1+b_2}}_{L^2}\\
&+\|\langle \xi \rangle^{b_1+b_2}\widehat{f} \, \|_{L^2}+\|J^{b_1+b_2}_{\xi}\widehat{f}\, \|_{L^2}\\
\lesssim & \|\langle \xi \rangle^{b_1+b_2}\widehat{f} \, \|_{L^2}+\|J^{b_1+b_2}_{\xi}\widehat{f}\|_{L^2}+\|J^{\frac{(b_1-k-l')(b_1+b_2)}{b_1-k+2l+l'}}_{\xi}\widehat{f} \, \|_{L^2}.
\end{aligned}
\end{equation}
Since $\frac{b_1-k-l'}{b_1-k+2l+l'} \leq 1$, we have $\|J^{\frac{(b_1-k-l')(b_1+b_2)}{b_1-k+2l+l'}}_{\xi}\widehat{f} \, \|_{L^2} \leq \|J^{b_1+b_2}_{\xi}\widehat{f} \, \|_{L^2}$, and thus, the above estimate yields the desired result.

On the other hand, we assume $b_2>0$, $0<k\leq b_1$ and $0\leq l \leq \lfloor k/2 \rfloor$. Since the case $k-2l=0$ follows trivially, we assume $k-2l>0$. By the estimate \eqref{dimNeq2} and the interpolation inequality in Proposition \ref{propweight2}, we get
\begin{equation}\label{dimNeq4}
\begin{aligned}
\|D^{b_2}\big(\langle\xi\rangle^{k-2l}\partial_{\xi_j}^{b_1-k}\widehat{f} \, \big)\|_{L^2} \lesssim & \sum_{\substack{0\leq l' \leq b_1-k \\ l'<k-2l}}\|J^{b_1+b_2-k-l'}_{\xi}\big(\langle \xi \rangle^{k-2l-l'} \widehat{f} \, \big)\|_{L^2}+\|\langle \xi \rangle^{b_1+b_2}\widehat{f}\|_{L^2}+\|J^{b_1+b_2}_{\xi}\widehat{f} \, \|_{L^2} \\
\lesssim & \sum_{\substack{0\leq l' \leq b_1-k \\ l'<k-2l}}\|\langle \xi \rangle^{b_1+b_2}\widehat{f} \, \|_{L^2}^{\frac{k-2l-l'}{b_1+b_2}}\|J^{\frac{(b_1+b_2-k-l')(b_1+b_2)}{(b_1+b_2-k+2l+l')}}_{\xi}\widehat{f} \, \|_{L^2}^{\frac{b_1+b_2-k-2l+l'}{b_1+b_2}} \\
&+\|\langle \xi \rangle^{b_1+b_2}\widehat{f} \, \|_{L^2}+\|J^{b_1+b_2}_{\xi}\widehat{f}\, \|_{L^2}\\
\lesssim & \|\langle \xi \rangle^{b_1+b_2}\widehat{f}\, \|_{L^2}+\|J^{b_1+b_2}_{\xi}\widehat{f} \, \|_{L^2}.
\end{aligned}
\end{equation}
Plugging \eqref{dimNeq3} and \eqref{dimNeq4} into \eqref{dimNeq1}, and applying Plancherel's identity, we complete the proof. 
\end{proof}

We close this section presenting some $L^{\infty}(\mathbb{R}^N)$ estimates for solutions of the homogeneous equation associated to \eqref{gHartree}.

\begin{prop}\label{propinfnorm}
Let $b \in \mathbb{R}^{+}$, $k, K\in \mathbb{Z}^{+}\cup\{0\}$ and $s\geq K +2k+3+\lfloor\frac{N}{2} \rfloor+b$. Then for any $t\in \mathbb{R}$ and any multi-index $\beta$ of order $|\beta|\leq K$,
\begin{equation}\label{eqlinfy}
\begin{aligned}
\|\langle x \rangle^{b}&\partial^{\beta}e^{it\Delta}f\|_{L^{\infty}}\\
\lesssim & \langle t \rangle^{k} \sum_{|\alpha|\leq k}\|\langle x \rangle^{b}\partial^{\alpha}f\|_{L^{\infty}}+ \langle t \rangle^{k+1+b}\Big( \|J^{s}f\|_{L^2}+ \sum_{k< |\alpha| \leq K +2k+3+\lfloor\frac{N}{2} \rfloor}\|\langle x \rangle^{b}\partial^{\alpha}f\|_{L^{2}}\Big).
\end{aligned}
\end{equation}
\end{prop}

\begin{proof}
From the relation $\frac{d^k}{dt^{k}}(e^{it\Delta}f)=(i\Delta)^{k}(e^{it\Delta}f)$, we apply Taylor's formula to find
\begin{equation}
e^{it\Delta}f=\sum_{j=0}^{k} \frac{(it)^{j}}{j!}\Delta^j f+\frac{i^{k+1}}{k!}\int_0^t (t-t')^{k}\Delta^{k+1}(e^{it'\Delta}f)\, dt',
\end{equation}
for all $t>0$. Then for a given multi-index $\beta$ of order $|\beta|\leq K$,
\begin{equation}
\partial^{\beta}(e^{it\Delta}f)=\sum_{j=0}^{k} \frac{(it)^{j}}{j!}\partial_x^{\beta}\Delta^j f+\frac{i^{k+1}}{k!}\int_0^t (t-t')^{k}\partial^{\beta}\Delta^{k+1}(e^{it'\Delta}f)\, dt',
\end{equation}
so that by Sobolev embedding $H^{\lfloor \frac{N}{2} \rfloor +1}(\mathbb{R}^N)\hookrightarrow L^{\infty}(\mathbb{R}^N)$ and an application of Lemma \ref{derivexp2}, we get 
\begin{equation}\label{eq(III)4} 
\begin{aligned}
\|\langle x \rangle^{b}\partial^{\beta}&(e^{it \Delta}f)\|_{L^{\infty}}\\
 \lesssim & \langle t \rangle^{k} \sum_{|\alpha|\leq K+2k}\|\langle x \rangle^{b}\partial^{\alpha}f\|_{L^{\infty}}+ |t|\langle t \rangle^{k} \sum_{2k+2 \leq |\alpha| \leq K +2k+2}\|\langle x \rangle^{b}e^{it\Delta}\partial^{\alpha}f\|_{L^{\infty}}\\
\lesssim & \langle t \rangle^{k} \sum_{|\alpha|\leq K+ 2k}\|\langle x \rangle^{b}\partial^{\alpha}f\|_{L^{\infty}}+ \langle t \rangle^{k+1} \sum_{2k+2< |\alpha| \leq K +2k+3+\lfloor\frac{N}{2} \rfloor}\|\langle x \rangle^{b}e^{it\Delta}\partial^{\alpha}f\|_{L^{2}} \\
\lesssim & \langle t \rangle^{k} \sum_{|\alpha|\leq k}\|\langle x \rangle^{b}\partial^{\alpha}f\|_{L^{\infty}}+ \langle t \rangle^{k+1+b}\Big( \|J^{s}f\|_{L^2}+ \sum_{k< |\alpha| \leq K +2k+3+\lfloor\frac{N}{2} \rfloor}\|\langle x \rangle^{b}\partial^{\alpha}f\|_{L^{2}}\Big),
\end{aligned}
\end{equation}
for all $s\geq K +2k+3+\lfloor\frac{N}{2} \rfloor+b$.
\end{proof}


\section{Proof of Theorem \ref{mainTHM}} \label{S:main}

Let $\frac{4}{3}<p<2$ and $0<\gamma<\frac{N(3p-4)}{2p}$. We consider $m\in \mathbb{R}^{+}$, $M,M_0\in \mathbb{Z}^{+}$ satisfying the conditions \eqref{condmainT1}-\eqref{condmainT3}. We further set
\begin{equation}\label{restric0.2}
N(u)=\big(|x|^{-(N-\gamma)}\ast |u|^{p}\big)|u|^{p-2}u.
\end{equation}
Recalling the space $\mathfrak{X}$ determined by \eqref{spacedef} and \eqref{normspace}, we will establish well-posedness conclusions by applying the contraction principle to the integral operator associated to \eqref{gHartree}, namely,
\begin{equation*}
\Phi(u(t))=e^{it \Delta }u_0+i\mu\int_0^t e^{i(t-t')\Delta} N(u(t'))\, dt',
\end{equation*}
acting on the space
\begin{equation}\label{functspa}
\begin{aligned}
\mathfrak{X}_T(R,\lambda)=\Big\{ u\in &C([0,T];\mathfrak{X}): \\
&\sup_{t\in [0,T]} \Big( \sum_{|\alpha|\leq \lfloor \frac{N}{2} \rfloor}\|\langle x \rangle^m \partial^{\alpha}u(t)\|_{L^{\infty}_x}+\sum_{\lfloor \frac{N}{2} \rfloor<|\alpha|\leq M}\|\langle x \rangle^m \partial^{\alpha} u(t)\|_{L^{2}_x}\\
&\hspace{5cm}+\sum_{M<|\alpha|\leq M+M_0-N}\|\partial^{\alpha} u(t)\|_{L^2_x}\Big)=\|u\|_{L^{\infty}_T\mathfrak{X}} \leq R, \\
& \inf_{(x,t)\in \mathbb{R}^N\times [0,T]} |\langle x\rangle^m u(x,t)|\geq \frac{\lambda}{2}\Big\}.
\end{aligned}
\end{equation}
Notice that in the one-dimensional case the space \eqref{functspa} reduces to
\begin{equation}\label{functspa1D}
\begin{aligned}
\mathfrak{X}_T^{1D}(R,\lambda)=\Big\{ u\in &C([0,T];\mathfrak{X}): \\
&\sup_{t\in [0,T]} \Big(\|\langle x \rangle^m u(t)\|_{L^{\infty}_x}+\sum_{j=1}^{M}\|\langle x \rangle^m \partial_x^j u(t)\|_{L^{2}_x}+\sum_{j=M+1}^{M+M_0-1}\|\partial_x^{j} u(t)\|_{L^2_x}\Big)=\|u\|_{L^{\infty}_T\mathfrak{X}} \leq R, \\
& \inf_{(x,t)\in \mathbb{R}\times [0,T]} |\langle x\rangle^m u(x,t)|\geq \frac{\lambda}{2}\Big\}.
\end{aligned}
\end{equation}
For example, setting $\frac{7}{4}<p<2$, $0<\gamma<\frac{2p-3}{2}$ and $m=\Big(\frac{1}{2}\Big)^{+}$ (i.e., $m=\frac{1}{2}+\epsilon$ with $\epsilon \ll 1$), we can take $M=6$ and $M_0=4$ in the space \eqref{functspa1D}. 

In what follows we find an appropriate $R>0$ and $0<T\leq 1$ such that $\Phi$ maps $\mathfrak{X}_T(R,\lambda)$ into itself and defines a contraction with the $\|\cdot\|_{L^{\infty}_T\mathfrak{X}}$-norm. 

We begin by deducing the following proposition, which is convenient to estimate the nonlinear part of the equation in \eqref{gHartree}.
\begin{prop}\label{propmain1} 
Let $b\in \mathbb{R}$, $u \in \mathfrak{X}_T(R,\lambda)$ and $1\leq r \leq \infty$. For any multi-index $\alpha$ of order $|\alpha|\leq M+M_0-N$, it follows that
\begin{equation}\label{eqspaceX1}
\begin{aligned}
\|\langle x \rangle^{b}\partial^{\alpha}(|u|^p)\|_{L^{\infty}_T L^r_x}\lesssim  R^{p}\|\langle x\rangle^{b-mp}\|_{L^r_x}&+ \sum_{k=1}^{|\alpha|} \lambda^{-(2k-p)}R^{2k-1}\Big(R\|\langle x \rangle^{b-mp}\|_{L^r_x}\\
&+\sum_{\lfloor \frac{N}{2} \rfloor<|\beta|\leq |\alpha|}\|\langle x \rangle^{b+m(1-p)}\partial^{\beta}u\|_{L^{\infty}_T L^r_x}\Big).
\end{aligned}
\end{equation}
Moreover,  
\begin{equation}\label{eqspaceX2}  
\begin{aligned}
\|\langle x \rangle^{b}\partial^{\alpha}(|u|^{p-2}u)\|_{L^{\infty}_TL^r_x}\lesssim 
\sum_{k=0}^{|\alpha|} \lambda^{-(2(k+1)-p)}R^{2k}\Big(&R\|\langle x \rangle^{b-m(p-1)}\|_{L^r_x}\\
&+\sum_{\lfloor \frac{N}{2} \rfloor<|\beta|\leq |\alpha|} \|\langle x\rangle^{b+m(2-p)} \partial^{\beta}u\|_{L^{\infty}_{T}L^r_x}\Big).
\end{aligned}
\end{equation} 
Above, the convention for the empty summation (such as $\sum_{1 \leq k \leq 0}$) is defined as zero.
\end{prop}

\begin{proof}
We first deal with \eqref{eqspaceX1}. We prove the case $|\alpha|\geq 1$, since the case with $\alpha=0$ follows directly from the definition of the space $\mathfrak{X}_T(R,\lambda)$. In virtue of the identity 
\begin{equation}\label{identi1}
\begin{aligned}
\partial^{\alpha}(|u|^p)=&\sum_{k=1}^{|\alpha|}|u|^{p-2k}\Big(\sum_{\substack{\beta_1+\dots+\beta_k =\alpha \\ |\beta_j|\geq 1}} c_{\beta_1,\dots,\beta_k} \partial^{\beta_1}(|u|^2)\cdots \partial^{\beta_k}(|u|^2) \Big),
\end{aligned}
\end{equation}
and the fact $|u(x,t)|^{-1}\lesssim \lambda^{-1}\langle x \rangle^m$, it is seen that
\begin{equation}\label{eq1propmain1}
\begin{aligned}
\|\langle x \rangle^{b}\partial^{\alpha}(|u|^p)\|_{L^{\infty}_T L^r_x}\lesssim & \sum_{k=1}^{|\alpha|}\lambda^{-(2k-p)}\sum_{\substack{\beta_1+\dots+\beta_k =\alpha \\ |\beta_j|\geq 1}}\|\langle x \rangle^{b+m(2k-p)}\partial^{\beta_1}(|u|^2)\cdots \partial^{\beta_k}(|u|^2) \|_{L^{\infty}_T L^r_x} \\
=& \sum_{k=1}^{|\alpha|}\lambda^{-(2k-p)}\Big(\sum_{A_k^1(\alpha)}(\cdots)+\sum_{A_k^2(\alpha)}(\cdots)\Big),
\end{aligned}
\end{equation}
where we have defined
\begin{equation}\label{eq2propmain1}
\begin{aligned}
A_k^1(\alpha):=&\big\{(\beta_1,\dots, \beta_k)\, : \, \beta_1+\dots+\beta_k=\alpha, \, \, 1\leq |\beta_j|\leq M-\lfloor\frac{N}{2} \rfloor-1, \, \, j=1,\dots,k\big\}, \\
A_k^2(\alpha):=&\big\{(\beta_1,\dots, \beta_k)\, : \, \beta_1+\dots+\beta_k=\alpha, \,\,  1\leq |\beta_j|, \, \, j=1,\dots,k\big\}\setminus A_{k}^1(\alpha).
\end{aligned}
\end{equation}
To estimate the sum over $A_k^1(\alpha)$ for a fixed $k=1,\dots, |\alpha|$, we use the identity
\begin{equation*}
\partial^{\beta_j}(|u|^2)=\sum_{\beta_{j,1}+\beta_{j,2}=\beta_j} c_{\beta_{j,1},\beta_{j,2}} (\partial^{\beta_{j,1}}u)(\partial^{\beta_{j,1}}\overline{u})
\end{equation*}
and the fact that all the derivatives in the sum over $A_k^1(\alpha)$ are of the order at most $M-\lfloor \frac{N}{2} \rfloor-1$ to get
\begin{equation}\label{eq3propmain1} 
\begin{aligned}
\sum_{(\beta_1,\dots,\beta_k)\in A_k^1(\alpha)}&\|\langle x \rangle^{b+m(2k-p)}\partial^{\beta_1}(|u|^2)\cdots \partial^{\beta_k}(|u|^2) \|_{L^{\infty}_T L^r_x} \\
\lesssim & \sum_{(\beta_1,\dots,\beta_k)\in A_k^1(\alpha)}\sum_{\substack{\beta_{j,1}+\beta_{j,2}=\beta_j \\ 1\leq j \leq k}}\|\langle x \rangle^{b+m(2k-p)}\partial^{\beta_{1,1}}u\partial^{\beta_{1,2}}\overline{u}\cdots \partial^{\beta_{k,1}}u \partial^{\beta_{k,2}}\overline{u} \|_{L^{\infty}_T L^r_x}\\
\lesssim & \Big( \sup_{|\beta|\leq M-\lfloor \frac{N}{2} \rfloor-1}\|\langle x\rangle^m \partial^{\beta}u\|_{L^{\infty}_{T,x}}\Big)^{2k}\|\langle x \rangle^{b-mp}\|_{L^{\infty}_T L^r_x}.
\end{aligned}
\end{equation}
Now, by Sobolev embedding $H^{\lfloor \frac{N}{2}\rfloor +1}(\mathbb{R}^N)\hookrightarrow L^{\infty}(\mathbb{R}^N)$, we obtain
\begin{equation}\label{eq3.1propmain1}
\begin{aligned}
\sup_{|\beta|\leq M-\lfloor \frac{N}{2} \rfloor-1}\|\langle x\rangle^m  \partial^{\beta}u\|_{L^{\infty}_{T,x}}\lesssim & \sup_{|\beta|\leq \lfloor \frac{N}{2} \rfloor}\|\langle x\rangle^m \partial^{\beta}u\|_{L^{\infty}_{T,x}}+\sum_{\lfloor \frac{N}{2} \rfloor <|\beta|\leq M}\|\langle x\rangle^m \partial^{\beta}u\|_{L^{\infty}_{T}L^2_x}\\
\lesssim & R.
\end{aligned}
\end{equation}
Plugging \eqref{eq3.1propmain1} into \eqref{eq3propmain1} completes the estimate for the sum over $A_k^1(\alpha)$.

On the other hand, since $|\alpha|\leq M+M_0-N$, the restrictions on $M$ and $M_0$ in \eqref{condmainT3} imply that for each $(\beta_1,\dots, \beta_k)\in A_k^2(\alpha)$ there exists a unique $k'\in \{1,\dots,k\}$ such that $|\beta_{k'}| >M-\lfloor\frac{N}{2} \rfloor-1$, in other words, $|\beta_{j}|\leq M-\lfloor\frac{N}{2} \rfloor-1$ for all $j\neq k'$. Therefore, by taking the $L^{\infty}_TL^r_x$-norm of the term with the derivative of higher order, and computing the $L^{\infty}_T L^{\infty}_x$-norm with the weight $\langle x \rangle^m$ of the remaining factors, we infer that the sum over $A_k^2(\alpha)$ is estimated as follows
\begin{equation}\label{eq4propmain1}
\begin{aligned}
\sum_{(\beta_1,\dots,\beta_k)\in A_k^2(\alpha)}&\|\langle x \rangle^{b+m(2k-p)}\partial^{\beta_1}(|u|^2)\cdots \partial^{\beta_k}(|u|^2) \|_{L^{\infty}_T L^r_x} \\
\lesssim & \Big( \sup_{|\beta|\leq M-\lfloor \frac{N}{2} \rfloor-1}\|\langle x\rangle^m \partial^{\beta}u\|_{L^{\infty}_{T,x}}\Big)^{2k-1}\Big(\sum_{M-\lfloor\frac{N}{2} \rfloor-1<|\beta|\leq |\alpha|}\|\langle x \rangle^{b+m(1-p)}\partial^{\beta}u\|_{L^{\infty}_T L^r_x}\Big)\\
\lesssim & R^{2k-1}\Big(\sum_{\lfloor \frac{N}{2} \rfloor<|\beta|\leq |\alpha|}\|\langle x \rangle^{b+m(1-p)}\partial^{\beta}u\|_{L^{\infty}_T L^r_x}\Big).
\end{aligned}
\end{equation}
Inserting \eqref{eq3propmain1}-\eqref{eq4propmain1} into \eqref{eq1propmain1}, we complete the deduction of \eqref{eqspaceX1}.

Now, we turn to \eqref{eqspaceX2}. This estimate follows by a similar reasoning leading to \eqref{eqspaceX1}. Indeed, for a given multi-index $\alpha$, we recall the identity
\begin{equation}\label{identi2}
\begin{aligned}
\partial^{\alpha}\big(|u|^{p-2}u\big)=&\sum_{k=1}^{|\alpha|} |u|^{p-2(k+1)}\sum_{\substack{\beta_0+\beta_1+\dots+\beta_k=\alpha, \\ |\beta_j|\geq 1, \, \, 1\leq j \leq k}}c_{\beta_0,\dots,\beta_k}\partial^{\beta_0}u\partial^{\beta_1}(|u|^2)\dots \partial^{\beta_k}(|u|^2)\\
&+ |u|^{p-2}\partial^{\alpha}u.
\end{aligned}
\end{equation}
Consequently, since $|u(x,t)|^{-1}\lesssim \lambda^{-1}\langle x \rangle^m$, we find
\begin{equation}\label{eq6propmain1}
\begin{aligned}
\|\langle x \rangle^{b}\partial^{\alpha}&\big(|u|^{p-2}u\big)\|_{L^{\infty}_T L^r_x}\\
\lesssim & \sum_{k=1}^{|\alpha|} \lambda^{-(2(k+1)-p)}\sum_{\substack{\beta_0+\beta_1+\dots+\beta_k=\alpha \\ |\beta_j|\geq 1,\, \, 1\leq j \leq k}}\|\langle x \rangle^{b+m(2(k+1)-p)}\partial^{\beta_0}u\partial^{\beta_1}(|u|^2)\dots \partial^{\beta_k}(|u|^2)\|_{L^{\infty}_T L^r_x} \\
&+\lambda^{-(2-p)}\|\langle x \rangle^{b+m(2-p)}\partial^{\alpha}u\|_{L^{\infty}_T L^r_x}\\ 
=& \sum_{k=1}^{|\alpha|}\Big(\sum_{B_{k}^1(\alpha)} (\cdots)+\sum_{B_{k}^2(\alpha)} (\cdots)\Big)+\lambda^{-(2-p)}\|\langle x \rangle^{b+m(2-p)}\partial^{\alpha}u\|_{L^{\infty}_T L^r_x},
\end{aligned}
\end{equation}
where we have set
\begin{equation}\label{eq7propmain1}
\begin{aligned}
B_k^1(\alpha):=&\big\{(\beta_0,\dots, \beta_k)\, : \, \beta_0+\dots+\beta_k=\alpha, \,\, |\beta_0|\leq M-1- \lfloor\frac{N}{2} \rfloor, \\
& \hspace{5.5cm} 1\leq |\beta_j|\leq M-1-\lfloor\frac{N}{2} \rfloor, \, \, j=1,\dots,k\big\}, \\
B_k^2(\alpha):=&\big\{(\beta_0,\dots, \beta_k)\, : \, \beta_0+\dots+\beta_k=\alpha, \,\,  1\leq |\beta_j|, \, \, j=1,\dots,k\big\}\setminus B_{k}^1(\alpha),
\end{aligned}
\end{equation}
for every $k=1,\dots,|\alpha|$. In virtue of \eqref{eq6propmain1} and \eqref{eq7propmain1}, we can   apply similar considerations as those used in \eqref{eq3propmain1} and \eqref{eq4propmain1} to obtain \eqref{eqspaceX2}. To avoid repetitions, we omit these computations. 
\end{proof}
We divide our considerations according to each component defining the norm of the space $\mathfrak{X}_T(R,\lambda)$.


\subsubsection*{ \bf Estimate in the space $L^{\infty}([0,T];W^{\lfloor\frac{N}{2} \rfloor,\infty}(\mathbb{R}^N;\langle x \rangle^m \, dx))$.}

Let $\beta$ be a multi-index of order $|\beta|\leq \lfloor \frac{N}{2} \rfloor$. By applying Proposition \ref{propinfnorm} and the restrictions \eqref{condmainT3}, we find
\begin{equation}\label{eqmainTH0}
\begin{aligned}
\sup_{t\in[0,T]}&\|\langle x \rangle^m \partial^{\beta}\Phi(u)\|_{L^{\infty}_x} \\
\lesssim & \langle T \rangle^{\lfloor \frac{N}{2} \rfloor+1+m}\Big( \sum_{|\alpha|\leq \lfloor \frac{N}{2} \rfloor}\|\langle x \rangle^m \partial^{\alpha}u_0\|_{L^{\infty}_{x}}+\sum_{\lfloor \frac{N}{2} \rfloor<|\alpha|\leq M}\|\langle x \rangle^m\partial^{\alpha}u_0\|_{L^2_x} \\
& \hspace{9cm}+\|J^{M+M_0-N}u_0\|_{L^2_x}\Big)\\ 
&+ |T|\langle T \rangle^{\lfloor \frac{N}{2} \rfloor+1+m} \Big(\sum_{|\alpha|\leq \lfloor \frac{N}{2} \rfloor}\|\langle x \rangle^m \partial^{\alpha}N(u)\|_{L^{\infty}_{T,x}}\\
&\hspace{3cm}+\sum_{\lfloor \frac{N}{2} \rfloor<|\alpha|\leq M}\|\langle x \rangle^m\partial^{\alpha}N(u)\|_{L^{\infty}_TL^2_x}+\|J^{M+M_0-N}N(u)\|_{L^{\infty}_TL^2_x}\Big).
\end{aligned}
\end{equation}
We further divide our analysis according to the previous norms involving the nonlinear term $N(u)$.
\\ \\
{\bf Estimate for $\|\langle x \rangle^m \partial^{\alpha} N(u)\|_{L^{\infty}_{T,x}}$, $|\alpha|\leq \lfloor \frac{N}{2} \rfloor$}. We set 
\begin{equation}\label{restric2}
\frac{1}{q}=\frac{1}{2}-\frac{\gamma}{N},
\end{equation}
which is well-defined provided that $0<\gamma<\frac{N}{2}$. Additionally, the condition \eqref{condmainT1} yields 
\begin{equation}\label{restric1}
\frac{N}{4(p-1)}<m< \frac{N}{2(2-p)}.
\end{equation}
Now, by Leibniz's rule we write
\begin{equation}\label{eqmainTH1}
\begin{aligned}
\|\langle x \rangle^m\partial^{\alpha}N(u)\|_{L^{\infty}_{T,x}}\lesssim & \sum_{\alpha_1+\alpha_2=\alpha}\|\langle x \rangle^{m(2-p)}\big( |x|^{-(N-\gamma)}\ast \partial^{\alpha_1}|u|^{p} \big)\|_{L^{\infty}_{T,x}}\|\langle x \rangle^{m(p-1)}\partial^{\alpha_2}(|u|^{p-2}u)\|_{L^{\infty}_{T,x}}.
\end{aligned}
\end{equation}
Bearing in mind that $|\alpha_1|\leq \lfloor \frac{N}{2} \rfloor$, by applying Sobolev inequality $W^{\lfloor \frac{N}{2} \rfloor +1,q}(\mathbb{R}^N)\hookrightarrow L^{\infty}(\mathbb{R}^N)$ (with $q$ given by \eqref{restric2}) and Proposition \ref{propweightriesz} with the right-hand side of \eqref{restric1}, we get 
\begin{equation}\label{eqmainTH1.0}
\begin{aligned}
\|\langle x \rangle^{m(2-p)}\big( |x|^{-(N-\gamma)}\ast \partial^{\alpha_1}|u|^{p} \big)\|_{L^{\infty}_{T,x}} \lesssim & \sum_{ |\beta|\leq 2\lfloor \frac{N}{2} \rfloor+1} \|\langle x \rangle^{m(2-p)}\big( |x|^{-(N-\gamma)}\ast \partial^{\beta}(|u|^{p})\big)\|_{L^{\infty}_{T}L^{q}_x}\\
\lesssim & \sum_{|\beta|\leq 2\lfloor \frac{N}{2} \rfloor+1} \|\langle x \rangle^{m(2-p)}\partial^{\beta}(|u|^{p})\|_{L^{\infty}_{T}L^{2}_x}.
\end{aligned}
\end{equation}
To complete the estimate of the above inequality, we apply \eqref{eqspaceX1} with $b=m(2-p)$ in Proposition \ref{propmain1} to get
\begin{equation}\label{eqmainTH1.1}
\begin{aligned} 
\sum_{ |\beta|\leq 2\lfloor \frac{N}{2} \rfloor+1}\|\langle x \rangle^{m(2-p)}&\partial^{\beta}(|u|^p)\|_{L^{\infty}_{T}L^2_x}\\
\lesssim &  R^{p} \, \|\langle x\rangle^{-2m(p-1)}\|_{L^2_x}+ \sum_{k=1}^{2\lfloor \frac{N}{2} \rfloor+1	} \lambda^{-(2k-p)}R^{2k-1}\Big(R\|\langle x \rangle^{-2m(p-1)}\|_{L^2_x}\\
&+\sum_{\lfloor \frac{N}{2} \rfloor<|\beta|\leq 2\lfloor \frac{N}{2} \rfloor+1}\|\langle x \rangle^{m(3-2p)}\partial^{\beta}u\|_{L^{\infty}_T L^2_x}\Big)\\
\lesssim &  R^{p}+ \sum_{k=1}^{2\lfloor \frac{N}{2} \rfloor+1} \lambda^{-(2k-p)}R^{2k},
\end{aligned}
\end{equation}
where we have used that $3-2p<1$ implies $\langle x \rangle^{m(3-2p)} \leq \langle x\rangle^{m}$. Likewise, since $|\alpha_2|\leq \lfloor \frac{N}{2}\rfloor$,  \eqref{eqspaceX2} with $b=m(p-1)$ yields
\begin{equation}\label{eqmainTH1.2}
\begin{aligned}
\|\langle x \rangle^{m(p-1)}\partial^{\alpha_2}(|u|^{p-2}u)\|_{L^{\infty}_{T,x}}\lesssim \sum _{k=0}^{\lfloor \frac{N}{2} \rfloor}\lambda^{-(2(k+1)-p)}R^{2k+1}.
\end{aligned}
\end{equation}
Collecting the previous estimates, we arrive at
\begin{equation}\label{eqmainTH2}
\|\langle x \rangle^m \partial^{\alpha}N(u)\|_{L^{\infty}_{T,x}} \lesssim  \Big(R^{p}+ \sum_{k=1}^{2\lfloor \frac{N}{2} \rfloor+1} \lambda^{-(2k-p)}R^{2k}\Big)\Big(\sum _{k=0}^{\lfloor \frac{N}{2} \rfloor}\lambda^{-(2(k+1)-p)}R^{2k+1}\Big),
\end{equation}
for each multi-index $\alpha$ with order $|\alpha| \leq \lfloor \frac{N}{2} \rfloor$.
\\ \\
{\bf Estimate for $\|\langle x \rangle^m \partial^{\alpha}N(u)\|_{L^{\infty}_{T}L^2_x}$, $\lfloor \frac{N}{2} \rfloor<|\alpha|\leq M$}. By the Leibniz's rule
\begin{equation}\label{eqmainTH3}
\begin{aligned}
\|\langle x \rangle^m\partial^{\alpha}N(u)\|_{L^{\infty}_TL^{2}_x}\lesssim \sum_{\alpha_1+\alpha_2=\alpha}\|\langle x \rangle^m \partial^{\alpha_1}\big(|x|^{-(N-\gamma)}\ast |u|^{p}\big)\partial^{\alpha_2}\big(|u|^{p-2}u\big)\|_{L^{\infty}_T L^2_x}.
\end{aligned}
\end{equation}
Let us divide our analysis according to the magnitude of the multi-index $\alpha_1$ in \eqref{eqmainTH3}.

\underline{Case $|\alpha_1| <N$}. We write
\begin{equation*}
\begin{aligned}
\|\langle x \rangle^m \big(|x|^{-(N-\gamma)}\ast \partial^{\alpha_1}|u|^{p}\big)&\partial^{\alpha_2}\big(|u|^{p-2}u\big)\|_{L^{\infty}_T L^2_x}\\
 \lesssim  &\|\langle x \rangle^{m\kappa} \big(|x|^{-(N-\gamma)}\ast \partial^{\alpha_1}(|u|^{p})\big)\|_{L^{\infty}_{T}L^{r_1}_x}\|\langle x \rangle^{m(1-\kappa)}\partial^{\alpha_2}\big(|u|^{p-2}u\big)\|_{L^{\infty}_{T}L^{r_2}_x},
\end{aligned}
\end{equation*}
where
\begin{equation}\label{eqcondi1}
\frac{1}{r_1}+\frac{1}{r_2}=\frac{1}{2}.
\end{equation}
We will choose $\kappa$, $r_1$ and $r_2$ according to the value of the weight $m$. But first, to apply Proposition \ref{propmain1}, we write
\begin{equation}\label{eqcondi2}
\frac{1}{r_1}=\frac{1}{q_2}-\frac{\gamma}{N}.
\end{equation}
Now, if $\Big\{\frac{2\gamma+N}{4(p-1)},\frac{N}{2}\Big\}<m\leq \frac{N}{p}$, we consider $1<q_2<\infty$ fixed, satisfying
\begin{equation}\label{eqcondi2.1}
\frac{\gamma}{N}<\frac{1}{q_2}<\frac{4m(p-1)-N}{2N},
\end{equation}
with this, \eqref{eqcondi1} and \eqref{eqcondi2}, we define $r_1$ and $r_2$, and therefore, we consider $\kappa$ such that 
\begin{equation}\label{eqcondi3}
\begin{aligned}
2-p+\frac{N}{mr_2}<\kappa < p-\frac{N}{mq_2}.
\end{aligned}
\end{equation}
Assuming that $\frac{N}{p}<m<\frac{N-2\gamma}{2(2-p)}$, we take $q_2$ according to
\begin{equation}\label{eqcondi3.1}
0<\frac{\gamma}{N}<\frac{1}{q_2}<\frac{N-2m(2-p)}{2N}.
\end{equation}
This condition determines $r_1$ and $r_2$, and thus, for this case, we consider
\begin{equation}\label{eqcondi4}  
2-p+\frac{N}{mr_2}<\kappa<\frac{N}{m}-\frac{N}{mq_2}.
\end{equation}
By \eqref{eqspaceX2} in Proposition \ref{propmain1}, together with the conditions \eqref{eqcondi3} and \eqref{eqcondi4}, we find
\begin{equation}\label{eqmainTH4}
\begin{aligned}
\|\langle x \rangle^{m(1-\kappa)}\partial^{\alpha_2}\big(|u|^{p-2}u\big)\|_{L^{\infty}_{T}L^{r_2}_x}\lesssim &
\sum_{k=0}^{|\alpha_2|} \lambda^{-(2(k+1)-p)}R^{2k}\Big(R\|\langle x \rangle^{m(2-p-\kappa)}\|_{L^{r_2}_x}\\
&\hspace{3cm}+\sum_{\lfloor \frac{N}{2} \rfloor<|\beta|\leq |\alpha_2|} \|\langle x\rangle^{m(3-p-\kappa)} \partial^{\beta}u\|_{L^{\infty}_{T}L^{r_2}_x}\Big)\\
\lesssim & \sum_{k=0}^{|\alpha_2|} \lambda^{-(2(k+1)-p)}R^{2k+1}.
\end{aligned}
\end{equation}
By applying Proposition \ref{propweightriesz} with our choice of $q_2$ and $\kappa_2$ determined by \eqref{eqcondi2.1} and \eqref{eqcondi3.1}, we get
\begin{equation}\label{eqmainTH4.1}
\begin{aligned}
\|\langle x \rangle^{m\kappa} \big(|x|^{-(N-\gamma)}\ast \partial^{\alpha_1}(|u|^{p})\big)\|_{L^{\infty}_{T}L^{r_1}_x} \lesssim & \|\langle x \rangle^{m\kappa}\partial^{\alpha_1}(|u|^{p})\|_{L^{\infty}_{T}L^{q_2}_x}.
\end{aligned}
\end{equation}
Now, for each multi-index $|\alpha_1|\leq N$, Proposition \ref{propmain1} yields
\begin{equation}\label{eqmainTH5}
\begin{aligned}
\|\langle x \rangle^{m\kappa}\partial^{\alpha_1}(|u|^p)&\|_{L^{\infty}_T L^{q_2}_x}\\
\lesssim &  R^{p}\|\langle x\rangle^{-m(p-\kappa)}\|_{L^{q_2}_x}\\
&+ \sum_{k=1}^{|\alpha_1|} \lambda^{-(2k-p)}R^{2k-1}\|\langle x \rangle^{-m(p-\kappa)}\|_{L^{q_2}_x}\Big(R+\sum_{\lfloor \frac{N}{2} \rfloor<|\beta|\leq |\alpha_1|}\|\langle x \rangle^{m}\partial^{\beta}u\|_{L^{\infty}_{T,x}}\Big) \\
\lesssim &  R^{p}+ \sum_{k=1}^{|\alpha_1|} \lambda^{-(2k-p)}R^{2k}.
\end{aligned}
\end{equation}
Gathering together \eqref{eqmainTH4} and \eqref{eqmainTH5}, we arrive at 
\begin{equation}\label{eqmainTH6}
\|\langle x \rangle^m \big(|x|^{-(N-\gamma)}\ast \partial^{\alpha_1}|u|^{p}\big)\partial^{\alpha_2}\big(|u|^{p-2}u\big)\|_{L^{\infty}_{T}L^2_x}\lesssim \Big(R^p+\sum_{k=1}^{M}\lambda^{-(2k-p)}R^{2k}\Big)\Big(\sum_{k=0}^{M} \lambda^{-(2(k+1)-p)}R^{2k+1}\Big),
\end{equation}
whenever $|\alpha_1|< N$.
\\ \\
\underline{Case $N \leq |\alpha_1|\leq M$}. We write
\begin{equation*}
\begin{aligned}
\|\langle x \rangle^m \partial^{\alpha_1}\big(|x|^{-(N-\gamma)}\ast |u|^{p}\big)&\partial^{\alpha_2}\big(|u|^{p-2}u\big)\|_{L^{\infty}_T L^2_x}\\
&\lesssim \|\langle x \rangle^{m(2-p)} \partial^{\alpha_1}\big(|x|^{-(N-\gamma)}\ast |u|^{p}\big)\|_{L^{\infty}_{T}L^2_x}\|\langle x \rangle^{m(p-1)}\partial^{\alpha_2}\big(|u|^{p-2}u\big)\|_{L^{\infty}_{T,x}}.
\end{aligned}
\end{equation*}
Since $\lfloor \frac{N}{2} \rfloor+1 \leq |\alpha_1|$, we have $|\alpha_2|\leq M-\lfloor\frac{N}{2} \rfloor-1$, then by \eqref{eqspaceX2} and Sobolev embedding,  
\begin{equation}\label{eqmainTH7}
\begin{aligned}
\|\langle x \rangle^{m(p-1)}\partial^{\alpha_2}\big(|u|^{p-2}u\big)\|_{L^{\infty}_{T,x}}\lesssim & 
\sum_{k=0}^{|\alpha|} \lambda^{-(2(k+1)-p)}R^{2k}\Big(R+\sum_{\lfloor \frac{N}{2} \rfloor<|\beta|\leq |\alpha_2|} \|\langle x\rangle^{m} \partial^{\beta}u\|_{L^{\infty}_{T,x}}\Big) \\
\lesssim & \sum_{k=0}^{|\alpha|} \lambda^{-(2(k+1)-p)}R^{2k}\Big(R+\sum_{\lfloor \frac{N}{2} \rfloor<|\beta|\leq |\alpha_2|+\lfloor \frac{N}{2}\rfloor+1} \|\langle x\rangle^{m} \partial^{\beta}u\|_{L^{\infty}_{T}L^2_x}\Big)\\
\lesssim & \sum_{k=0}^{M} \lambda^{-(2(k+1)-p)}R^{2k+1}.
\end{aligned}
\end{equation}
Now, given that $|\alpha_1|\geq N$, there exists a multi-index $\alpha_{1,1}=(\alpha^1_{1,1},\dots,\alpha^N_{1,1})$ such that $\alpha_{1,1} \leq \alpha_1$ and $|\alpha_{1,1}|=N$. Consequently, for some constant $c_{\gamma}\neq 0$, we write
\begin{equation}\label{eqmainTH8}
\begin{aligned}
|x|^{-(N-\gamma)}\ast \partial^{\alpha_1}\big(|u|^{p}\big)=c_{\gamma}\partial^{\alpha_1}D^{-\gamma}\big(|u|^p\big)=&c_{\gamma}\partial^{\alpha_{1,1}}D^{-\gamma}\partial^{\alpha_1-\alpha_{1,1}}\big(|u|^{p}\big)\\
=&(-1)^{|\alpha_{1,1}|}c_{\gamma}\mathcal{R}_1^{\alpha_{1,1}^1}\cdots \mathcal{R}_N^{\alpha_{1,1}^N}D^{N-\gamma}\partial^{\alpha_1-\alpha_{1,1}}\big(|u|^p\big),
\end{aligned}
\end{equation}
where $\mathcal{R}_j=-\partial_{x_j}D^{-1}$ denotes the Riesz transform in the $j$-variable for any dimension $N\geq 2$, and in the one-dimensional setting, we will use the same notation to refer to the Hilbert transform operator. In view of the fact that $0<m<\frac{N}{2(2-p)}$, we have $\langle x \rangle^{2m(2-p)}$ is in the $A_2$ class, and thus, the map 
\begin{equation*}
\mathcal{R}_j: L^{2}(\mathbb{R}^N;\langle x \rangle^{2m(2-p)} \, dx)\rightarrow L^2(\mathbb{R}^N;\langle x \rangle^{2m(2-p)}\, dx)
\end{equation*}
is bounded for each $j=1,\dots,N$. By applying the previous decomposition \eqref{eqmainTH8}, we get
\begin{equation}\label{eqmainTH8.1}
\begin{aligned}
\|\langle x \rangle^{m(2-p)} \partial^{\alpha_1}\big(|x|^{-(N-\gamma)}\ast |u|^{p}\big)\|_{L^{\infty}_{T}L^2_x}\lesssim & \|\langle x \rangle^{m(2-p)} \partial^{\alpha_1}D^{-\gamma}\big( |u|^{p}\big)\|_{L^{\infty}_{T}L^2_x}\\
\sim &   \|\langle x \rangle^{m(2-p)}\mathcal{R}_1^{\alpha_{1,1}^1}\cdots \mathcal{R}_N^{\alpha_{1,1}^N}D^{N-\gamma}\partial^{\alpha_1-\alpha_{1,1}}\big(|u|^p\big)\|_{L^{\infty}_{T}L^2_x} \\
 \lesssim & \sup_{ |\beta|\leq M-N} \|\langle x \rangle^{m(2-p)}D^{N-\gamma}\partial^{\beta}\big(|u|^p\big)\|_{L^{\infty}_{T}L^2_x}.
\end{aligned}
\end{equation}
Considering that $0<m(2-p)<N/2$ and $0<\gamma<N/2$, it follows that $N-\gamma-m(2-p)>0$. Then setting $b=m(2-p)$ and $s=N-\gamma$ in Lemma \ref{lemmachangingDbyJ}, we deduce 
\begin{equation}\label{eqmainTH9}
\begin{aligned}
\|\langle x \rangle^{m(2-p)}D^{N-\gamma}\partial^{\beta}\big(|u|^p\big)\|_{L^{\infty}_{T}L^2_x} \lesssim & \|\langle x \rangle^{m(2-p)}\partial^{\beta}\big(|u|^p\big)\|_{L^{\infty}_{T}L^2_x}+\|J^{N-\gamma-m(2-p)}\partial^{\beta}\big(|u|^p\big)\|_{L^{\infty}_{T}L^2_x}\\
&+\|\langle x \rangle^{m(2-p)}J^{N-\gamma}\partial^{\beta}\big(|u|^p\big)\|_{L^{\infty}_{T}L^2_x},
\end{aligned}
\end{equation}
for each $|\beta|\leq M-N$. Let us estimate each factor on the right-hand side of the above inequality. Indeed, by arguing as in \eqref{eqmainTH1.1}, applying Proposition \ref{propmain1}, we deduce
\begin{equation}\label{eqmainTH10}
\begin{aligned}
\sup_{ |\beta|\leq M-N}\|\langle x \rangle^{m(2-p)}\partial^{\beta}\big(|u|^p\big)\|_{L^{\infty}_{T}L^2_x} \lesssim R^p +\sum_{k=1}^{M-N}\lambda^{-(2k-p)}R^{2k}.
\end{aligned}
\end{equation}
Likewise, given that $|\beta|\leq M-N$, we have $N-\gamma-m(2-p)+|\beta|\leq M$, and hence,
\begin{equation}
\begin{aligned}
\sup_{|\beta|\leq M-N}\|J^{N-\gamma-m(2-p)}\partial^{\beta}\big(|u|^p\big)\|_{L^{\infty}_{T}L^2_x} \lesssim &\sup_{|\beta|\leq M}\|\langle x \rangle^{m(2-p)}\partial^{\beta}\big(|u|^p\big)\|_{L^{\infty}_{T}L^2_x}\\
\lesssim & R^p +\sum_{k=1}^{M}\lambda^{-(2k-p)}R^{2k}.
\end{aligned}
\end{equation}
We are left with controlling the last term on the right-hand side of \eqref{eqmainTH9}. Recalling the integer number $M_0$ satisfying \eqref{condmainT2}, there exists $0<\epsilon<\frac{4m(p-1)-N}{2m}$ such that
\begin{equation}\label{epsicond}
\begin{aligned}
\frac{(N-\gamma)(2mp-N)}{4m(p-1)-N}<\frac{(N-\gamma)((2-p)+\epsilon)}{\epsilon}\leq M_0.
\end{aligned}
\end{equation}
Consequently, we apply the interpolation inequality in Proposition \ref{propweight2} to get 
\begin{equation}\label{eqmainTH11} 
\begin{aligned}
\|\langle x \rangle^{m(2-p)}J^{N-\gamma}\partial^{\beta}\big(|u|^p\big)\|_{L^{\infty}_{T}L^2_x} \lesssim & \|J^{M_0}\partial^{\beta}\big(|u|^{p}\big)\|_{L^{\infty}_TL^2_x}^{\frac{\epsilon}{2-p+\epsilon}}\|\langle x \rangle^{m(2-p)+m\epsilon}\partial^{\beta}\big(|u|^{p}\big)\|_{L^{\infty}_T L^2_x}^{\frac{2-p}{2-p+\epsilon}}\\
\lesssim & \|J^{M+M_0-N}\big(|u|^{p}\big)\|_{L^{\infty}_TL^2_x}+\|\langle x \rangle^{m(2-p)+m\epsilon}\partial^{\beta}\big(|u|^{p}\big)\|_{L^{\infty}_T L^2_x},
\end{aligned}
\end{equation}
where we have $ |\beta|\leq M-N$. To estimate the second term on the right-hand side of \eqref{eqmainTH11}, we use the embedding $H^{N}(\mathbb{R}^N)\hookrightarrow L^{\infty}(\mathbb{R}^N)$ to get
\begin{equation*}
\begin{aligned}
\sup_{\lfloor \frac{N}{2} \rfloor<|\eta|\leq |\beta|}\|\langle x \rangle^{m}\partial^{\eta}u\|_{L^{\infty}_{T,x}}\lesssim \sum_{\lfloor \frac{N}{2} \rfloor<|\eta|\leq M}\|\langle x \rangle^{m}\partial^{\eta}u\|_{L^{\infty}_T L^2_x},
\end{aligned}
\end{equation*}
so that Proposition \ref{propmain1} and our choice of $\epsilon>0$ yield
\begin{equation*}
\begin{aligned}
\|\langle x \rangle^{m(2-p)+m\epsilon}&\partial^{\beta}\big(|u|^{p}\big)\|_{L^{\infty}_T L^2_x}\\
\lesssim & \|\langle x\rangle^{-2m(p-1)+m\epsilon}\|_{L^2_x}\Big( R^{p}+ \sum_{k=1}^{|\beta|} \lambda^{-(2k-p)}R^{2k-1}\big(R+\sum_{\lfloor \frac{N}{2} \rfloor<|\eta|\leq |\beta|}\|\langle x \rangle^{m}\partial^{\eta}u\|_{L^{\infty}_{T,x}}\big)\Big)\\
\lesssim & \Big( R^{p}+ \sum_{k=1}^{M-N} \lambda^{-(2k-p)}R^{2k}\Big).
\end{aligned}
\end{equation*}
On the other hand, we use Proposition \ref{propmain1} with $a=0$, and the fact that $1-p<0$ to deduce
\begin{equation}\label{eqmainTH11.1}
\begin{aligned}
\|J^{M+M_0-N}\big(|u|^{p}\big)\|_{L^{\infty}_TL^2_x}\lesssim & \sum_{|\beta|\leq M+M_0-N} \|\partial^{\beta}\big(|u|^p\big)\|_{L^{\infty}_T L^2_x} \\
\lesssim & \Big(R^{p}+\sum_{k=1}^{M+M_0-N}\lambda^{-(2k-p)}R^{2k-1}\Big)\|\langle x \rangle^{-mp}\|_{L^2_x}\\
&+\sum_{k=1}^{M+M_0-N}\lambda^{-(2k-p)}R^{2k-1}\sum_{\lfloor \frac{N}{2} \rfloor < |\beta|\leq M+M_0-N}\|\langle x \rangle^{m(1-p)}\partial^{\beta}u\|_{L^{\infty}_T L^2_x}\\
\lesssim & R^{p}+\sum_{k=1}^{M+M_0-N}\lambda^{-(2k-p)}R^{2k}.
\end{aligned}
\end{equation}
Summarizing, for any $N\leq|\alpha_1|\leq M$, we have
\begin{equation}\label{eqmainTH12}  
\|\langle x \rangle^{m(2-p)} \partial^{\alpha_1}\big(|x|^{-(N-\gamma)}\ast |u|^{p}\big)\|_{L^{\infty}_{T}L^2_x}\lesssim R^{p}+\sum_{k=1}^{M+M_0-N}\lambda^{-(2k-p)}R^{2k}.
\end{equation}
This estimate concludes considering 
the case $N \leq |\alpha_1|\leq M$.

Collecting the above results for the different values of $\alpha_1$, we deduce
\begin{equation}\label{eqmainTH13} 
\sum_{\lfloor \frac{N}{2}\rfloor <|\alpha|\leq M}\|\langle x \rangle^m\partial^{\alpha}N(u)\|_{L^{\infty}_TL^{2}_x}\lesssim \Big(R^p+\sum_{k=1}^{M+M_0-N}\lambda^{-(2k-p)}R^{2k}\Big)\Big(\sum_{k=0}^{M} \lambda^{-(2(k+1)-p)}R^{2k+1}\Big).
\end{equation}
\\ \\
{\bf Estimate for $\|J^{M+M_0-N} N(u)\|_{L^{\infty}_{T,x}}$}. Once again, we write
\begin{equation}\label{eqmainTH13.1}
\begin{aligned}
\|\partial^{\alpha}N(u)\|_{L^{\infty}_TL^{2}_x}\lesssim \sum_{\alpha_1+\alpha_2=\alpha}\| \partial^{\alpha_1}\big(|x|^{-(N-\gamma)}\ast |u|^{p}\big)\partial^{\alpha_2}\big(|u|^{p-2}u\big)\|_{L^{\infty}_T L^2_x}.
\end{aligned}
\end{equation}
If $|\alpha_1|<N$, we estimate as follows
\begin{equation*}
\begin{aligned}
\| \partial^{\alpha_1}\big(|x|^{-(N-\gamma)}&\ast |u|^{p}\big)\partial^{\alpha_2}\big(|u|^{p-2}u\big)\|_{L^{\infty}_T L^2_x}\\
&\lesssim \|\langle x \rangle^{m(2-p)} \partial^{\alpha_1}\big(|x|^{-(N-\gamma)}\ast |u|^{p}\big)\|_{L^{\infty}_{T,x}}\|\langle x \rangle^{-m(2-p)}\partial^{\alpha_2}\big(|u|^{p-2}u\big)\|_{L^{\infty}_T L^2_x}.
\end{aligned}
\end{equation*}
Following the same line of arguments in \eqref{eqmainTH1.0} and \eqref{eqmainTH1.1}, one can deduce 
\begin{equation*}
\begin{aligned}
\|&\langle x \rangle^{m(2-p)}  \partial^{\alpha_1}\big(|x|^{-(N-\gamma)}\ast |u|^{p}\big)\|_{L^{\infty}_{T,x}}\\
 &\lesssim  R^p+ \sum_{k=1}^{N+\lfloor \frac{N}{2} \rfloor+1} \lambda^{-(2k-p)}\big( R^{2k}+R^{2k-1}\sum_{\lfloor\frac{N}{2}\rfloor<|\beta|\leq  N+\lfloor \frac{N}{2} \rfloor+1}\|\langle x \rangle^{m(3-2p)}\partial^{\beta}u\|_{L^{\infty}_T L^2_x}\big)\\
&\lesssim  R^p+\sum_{k=1}^{N+\lfloor \frac{N}{2} \rfloor+1} \lambda^{-(2k-p)} R^{2k}.
\end{aligned}
\end{equation*}
On the other hand, by Proposition \ref{propmain1}, we find
\begin{equation} \label{eqmainTH13.2}
\begin{aligned}
\|\langle x \rangle^{-m(2-p)}&\partial^{\alpha_2}\big(|u|^{p-2}u\big)\|_{L^{\infty}_T L^2_x}\\
\lesssim &\sum_{k=0}^{M+M_0-N} \lambda^{-(2(k+1)-p)}R^{2k}\Big(R+\sum_{\lfloor \frac{N}{2} \rfloor<|\beta|\leq M+M_0-N} \|\partial^{\beta}u\|_{L^{\infty}_{T}L^2_x}\Big)\\
\lesssim &\sum_{k=0}^{M+M_0-N} \lambda^{-(2(k+1)-p)}R^{2k+1}.
\end{aligned}
\end{equation}
This completes the estimate for  \eqref{eqmainTH13.1} whenever $|\alpha_1|\leq N$. 

Assuming now that $N\leq |\alpha_1| \leq M$, in virtue of \eqref{eqmainTH12} and \eqref{eqmainTH13.2}, we get
\begin{equation*}
\begin{aligned}
\| \partial^{\alpha_1}\big(|x|^{-(N-\gamma)}&\ast |u|^{p}\big)\partial^{\alpha_2}\big(|u|^{p-2}u\big)\|_{L^{\infty}_T L^2_x}\\
\lesssim & \|\langle x \rangle^{m(2-p)} \partial^{\alpha_1}\big(|x|^{-(N-\gamma)}\ast |u|^{p}\big)\|_{L^{\infty}_{T}L^2_x}\|\langle x \rangle^{-m(2-p)}\partial^{\alpha_2}\big(|u|^{p-2}u\big)\|_{L^{\infty}_{T,x}}\\
\lesssim & \Big(R^p+\sum_{k=1}^{M+M_0-N}\lambda^{-(2k-p)}R^{2k}\Big)\Big(\sum_{k=0}^{M+M_0-N} \lambda^{-(2(k+1)-p)}R^{2k+1}\Big).
\end{aligned}
\end{equation*}
Finally, if $M<|\alpha_1| \leq M+M_0-N$, we have $|\alpha_2|\leq M_0-N\leq M-\lfloor\frac{N}{2} \rfloor-1$. Thus, we can apply \eqref{eqmainTH7} to obtain
\begin{equation*}
\begin{aligned}
\| \partial^{\alpha_1}\big(|x|^{-(N-\gamma)}&\ast |u|^{p}\big)\partial^{\alpha_2}\big(|u|^{p-2}u\big)\|_{L^{\infty}_T L^2_x}\\
\lesssim & \|\partial^{\alpha_1}\big(|x|^{-(N-\gamma)}\ast |u|^{p}\big)\|_{L^{\infty}_{T}L^2_x}\|\partial^{\alpha_2}\big(|u|^{p-2}u\big)\|_{L^{\infty}_{T,x}}\\
\lesssim & \|\partial^{\alpha_1}\big(|x|^{-(N-\gamma)}\ast |u|^{p}\big)\|_{L^{\infty}_{T}L^2_x}\Big(\sum_{k=0}^{M} \lambda^{-(2(k+1)-p)}R^{2k+1} \Big).
\end{aligned}
\end{equation*}
To complete the estimate of the above inequality, we write $\alpha_1=(\alpha_{1,1},\dots, \alpha_{1,N})$, so that there exists a constant $c_{\gamma}\neq 0$ such that
\begin{equation*}
\begin{aligned}
\partial^{\alpha_1}\big(|x|^{-(N-\gamma)} \ast |u|^{p}\big)=c_{\gamma} \mathcal{R}_1^{\alpha_{1,1}}\dots \mathcal{R}_N^{\alpha_{1,N}}D^{|\alpha_1|-\gamma}\big(|u|^p \big),
\end{aligned}
\end{equation*}
and thus, by the boundedness of the Hilbert and Riesz transforms and \eqref{eqmainTH11.1}, we obtain 
\begin{equation*}
\begin{aligned}
\|\partial^{\alpha_1}\big(|x|^{-(N-\gamma)} \ast |u|^{p}\big)\|_{L^{\infty}_T L^2_x} \lesssim \|D^{|\alpha_1|-\gamma}\big(|u|^p \big)\|_{L^{\infty}_T L^2_x}&\lesssim \sum_{|\beta|\leq M+M_0-N}\|\partial^{\beta}\big(|u|^p \big)\|_{L^{\infty}_T L^2_x} \\
&\lesssim  R^{p}+\sum_{k=1}^{M+M_0-N}\lambda^{-(2k-p)}R^{2k}.
\end{aligned}
\end{equation*} 
Consequently, the above estimates show that 
\begin{equation}\label{eqmainTH14} 
\begin{aligned}
\|J^{M+M_0-N} N(u)\|_{L^{\infty}_{T}L^2_x} \lesssim \Big(R^p+\sum_{k=1}^{M+M_0-N}\lambda^{-(2k-p)}R^{2k}\Big)\Big( \sum_{k=0}^{M+M_0-N} \lambda^{-(2(k+1)-p)}R^{2k+1}\Big).
\end{aligned}
\end{equation}
Finally, plugging \eqref{eqmainTH2}, \eqref{eqmainTH13} and \eqref{eqmainTH14}  into \eqref{eqmainTH0}, we find 
\begin{equation}\label{eqmainTH15} 
\begin{aligned}
\sup_{t\in[0,T]}&\|\langle x \rangle^m\Phi(u)\|_{L^{\infty}_x} \\
\lesssim & \langle T \rangle^{\lfloor \frac{N}{2} \rfloor+1+m}\Big( \sum_{|\alpha|\leq \lfloor \frac{N}{2} \rfloor}\|\langle x \rangle^m \partial^{\alpha}u_0\|_{L^{\infty}_x}+\sum_{\lfloor \frac{N}{2} \rfloor<|\alpha|\leq M}\|\langle x \rangle^m\partial^{\alpha}u_0\|_{L^2_x} \\
& \hspace{9cm}+\|J^{M+M_0-N}u_0\|_{L^2_x}\Big)\\
&+ |T|\langle T \rangle^{\lfloor \frac{N}{2} \rfloor+1+m} \Big(R^{p}+ \sum_{k=1}^{M+M_0-N} \lambda^{-(2k-p)}R^{2k}\Big)\Big(\sum _{k=0}^{M+M_0-N}\lambda^{-(2(k+1)-p)}R^{2k+1}\Big).
\end{aligned}
\end{equation}

\subsubsection*{\bf Estimate for $\partial^{\alpha}u \in L^{\infty}([0,T];L^2(\mathbb{R}^N;\langle x \rangle^{2m} \, dx))$, $\lfloor \frac{N}{2} \rfloor<|\alpha|\leq M$.} Let $\alpha$ be a multi-index of order $|\alpha|\leq M$. By Lemma \ref{derivexp2}, the restrictions \eqref{condmainT2}, \eqref{condmainT3}, \eqref{eqmainTH13} and \eqref{eqmainTH14}, we deduce
\begin{equation}\label{eqmainTH16} 
\begin{aligned}
\sup_{t\in[0,T]}\|\langle x \rangle^m\partial^{\alpha}\Phi(u)\|_{L^{2}_x} \lesssim & \langle T\rangle^m \big(\|J^{M+M_0-N} u_0\|_{L^{2}_x}+\|\langle x \rangle^m \partial^{\alpha}u_0\|_{L^{2}_x}\big)\\
& + |T|\langle T \rangle^m \big(\|J^{M+M_0-N} N(u)\|_{L^{\infty}_TL^{2}_x}+\|\langle x \rangle^m \partial^{\alpha}N(u)\|_{L^{\infty}_TL^{2}_x}\big)\\
\lesssim &\langle T\rangle^m \big(\|J^{M+M_0-N} u_0\|_{L^{2}_x}+\|\langle x \rangle^m \partial^{\alpha}u_0\|_{L^{2}_x}\big)\\
& +|T|\langle T \rangle^m \Big(R^p+\sum_{k=1}^{M+M_0-N}\lambda^{-(2k-p)}R^{2k}\Big)\Big( \sum_{k=0}^{M+M_0-N} \lambda^{-(2(k+1)-p)}R^{2k+1}\Big).
\end{aligned}
\end{equation}

\subsubsection*{\bf Estimate in the space $L^{\infty}([0,T];H^{M+M_0-N}(\mathbb{R}^N))$.}

We consider a multi-index $\alpha$ with $|\alpha|\leq M+M_0-N$. By \eqref{eqmainTH14}, we deduce
\begin{equation}\label{eqmainTH17}
\begin{aligned}
\sup_{t\in[0,T]}\|\partial^{\alpha}\Phi(u)\|_{L^{2}_x}\lesssim & \|u_0\|_{H^{M+M_0-N}}+|T|\|\partial^{\alpha}N(u)\|_{L^{\infty}_T L^2_x} \\
\lesssim &\|u_0\|_{H^{M+M_0-N}}+|T| \Big(R^p+\sum_{k=1}^{M+M_0-N}\lambda^{-(2k-p)}R^{2k}\Big)\Big( \sum_{k=0}^{M+M_0-N} \lambda^{-(2(k+1)-p)}R^{2k+1}\Big).
\end{aligned}
\end{equation}
This completes the study of the $\|\cdot\|_{\mathfrak{X}}$-norm of $\Phi$. 

Let us now deduce some consequences of the previous results. But first, for $\lambda, R>0$, we define
\begin{equation}\label{eqmainTH18} 
\mathcal{G}_1(\lambda,R)=R^p+\sum_{k=1}^{M+M_0-N}\lambda^{-(2k-p)}R^{2k} \, \, \text{ and } \, \, \mathcal{G}_2(\lambda,R)=\sum_{k=0}^{M+M_0-N} \lambda^{-(2(k+1)-p)}R^{2k+1}.
\end{equation} 
Thus, recalling the definition of the $\mathfrak{X}$-norm \eqref{normspace}, and gathering \eqref{eqmainTH15}, \eqref{eqmainTH16} and \eqref{eqmainTH17}, there exists a constant $c>0$ such that 
\begin{equation*}
\begin{aligned}
\|\Phi(u)\|_{\mathfrak{X}}\leq & c\langle T \rangle^{\lfloor \frac{N}{2} \rfloor+1+m}\big(\sum_{|\alpha|\leq \lfloor \frac{N}{2} \rfloor}\|\langle x \rangle^{m}\partial^{\alpha}u_0\|_{L^{\infty}_x}+\sum_{\lfloor \frac{N}{2} \rfloor<|\alpha|\leq M}\|\langle x \rangle^m \partial^{\alpha} u_0\|_{L^2_x}\\
& \hspace{4cm}+\|J^{M+M_0-N}u_0\|_{L^2_x}\big)\\
&+c|T|\langle T \rangle^{\lfloor \frac{N}{2} \rfloor+1+m}\mathcal{G}_1(\lambda,R)\mathcal{G}_2(\lambda,R).
\end{aligned}
\end{equation*}  
Fixing 
\begin{equation*}
R=2c\big(\sum_{|\alpha|\leq \lfloor \frac{N}{2} \rfloor}\|\langle x \rangle^{m}\partial^{\alpha}u_0\|_{L^{\infty}_x}+\sum_{\lfloor \frac{N}{2} \rfloor<|\alpha|\leq M}\|\langle x \rangle^m \partial^{\alpha} u_0\|_{L^2_x}+\|J^{M+M_0-N}u_0\|_{L^2_x}\big)
\end{equation*}
and taking $T>0$ such that
\begin{equation}\label{eqmainTH19}
\begin{aligned}
\frac{1}{2}\langle T \rangle^{\lfloor \frac{N}{2} \rfloor+1+m}+c|T|\langle T \rangle^{\lfloor \frac{N}{2} \rfloor+1+m}R^{-1}\mathcal{G}_1(\lambda,R)\mathcal{G}_2(\lambda,R)\leq 1,
\end{aligned}
\end{equation}
we get $\|\Phi(u)\|_{L^{\infty}_T\mathfrak{X}}\leq R$. 
\\ \\
To prove that $\Phi$ is well-defined in $\mathfrak{X}_T(R,\lambda)$, it only remains to establish the condition $\inf_{(x,t)\in \mathbb{R} \times [0,T]}|\langle x \rangle^m \Phi(x,t)|\geq \frac{\lambda}{2}$. By the arguments around \eqref{eq(III)4} in the proof of Proposition \ref{propinfnorm}, we have that
\begin{equation*}
\begin{aligned}
e^{it\Delta}u_0-u_0=\sum_{j=1}^{k} \frac{(it)^{j}}{j!}\Delta^j u_0+\frac{i^{k+1}}{k!}\int_0^t (t-t')^{k}\Delta^{k+1}(e^{it'\Delta}u_0)\, dt',
\end{aligned}
\end{equation*}
where $k=\lfloor \frac{N}{2} \rfloor$, and therefore,
\begin{equation}\label{eqmainTH20}
\begin{aligned}
\big\|\langle x \rangle^{m} \big(e^{it\Delta}u_0-u_0 \big)\big\|_{L^{\infty}_x}\lesssim &|t|\langle t\rangle^{\lfloor \frac{N}{2} \rfloor+m}\big(\sum_{|\alpha|\leq \lfloor \frac{N}{2} \rfloor}\|\langle x \rangle^{m}\partial^{\alpha}u_0\|_{L^{\infty}_x}\\
&\hspace{3cm}+\sum_{\lfloor \frac{N}{2} \rfloor<|\alpha|\leq M}\|\langle x \rangle^m\partial^{\alpha}u_0\|_{L^2_x}+\|J^{M+M_0-N}u_0\|_{L^2_x} \big).
\end{aligned}
\end{equation}
Additionally,  we apply Proposition \ref{propinfnorm} (see, for instance, \eqref{eqmainTH0}) together with \eqref{eqmainTH2}, \eqref{eqmainTH13} and \eqref{eqmainTH14} to obtain
\begin{equation}\label{eqmainTH21}
\begin{aligned}
\big\|\langle x \rangle^{m}\int_0^t e^{i(t-t')\Delta}N(u) \, dt'\big\|_{L^{\infty}_{T,x}}\lesssim  |T|\langle T \rangle^{\lfloor \frac{N}{2} \rfloor+1+m}\mathcal{G}_1(\lambda,R)\mathcal{G}_2(\lambda,R).
\end{aligned}
\end{equation}
By \eqref{eqmainTH20} and \eqref{eqmainTH21} there exists a constant $c_1=c_1(\mu)>0$ such that
\begin{equation}\label{eqmainTH22}
\begin{aligned}
|\langle x \rangle^m \Phi(u)(x,t)|\geq & |\langle x \rangle^{m}e^{it \Delta}u_0(x)|-|\mu|\big\|\langle x \rangle^{m}\int_0^t e^{i(t-t')\Delta}N(u) \, dt'\big\|_{L^{\infty}_{T,x}} \\
\geq & |\langle x \rangle^m u_0(x)|-\big\|\langle x \rangle^{m}\big( e^{it\Delta}u_0-u_0 \big)\big\|_{L^{\infty}_x}-|\mu|\big\|\langle x \rangle^{m}\int_0^t e^{i(t-t')\Delta}N(u) \, dt'\big\|_{L^{\infty}_{T,x}}\\
\geq & \lambda-c_1|T|\langle T\rangle^{\lfloor \frac{N}{2} \rfloor+1+m}\big(\sum_{|\alpha|\leq \lfloor \frac{N}{2}\rfloor}\|\langle x \rangle^{m}\partial^{\alpha}u_0\|_{L^{\infty}_x}\\
&\hspace{3cm}+\sum_{\lfloor \frac{N}{2} \rfloor< |\alpha|\leq M}\|\langle x \rangle^m\partial^{\alpha}u_0\|_{L^2_x}+\|J^{M+M_0-N}u_0\|_{L^2_x} \big)\\
&-c_1|T|\langle T \rangle^{\lfloor \frac{N}{2} \rfloor+1+m}\mathcal{G}_1(\lambda,R)\mathcal{G}_2(\lambda,R),
\end{aligned}  
\end{equation}
where $\mathcal{G}_1$ and $\mathcal{G}_2$ are defined in \eqref{eqmainTH18}. If we take $T>0$  small such that \eqref{eqmainTH19} holds true and
\begin{equation}\label{eqmainTH22.0}  
\begin{aligned}
c_1|T|\langle T\rangle^{\lfloor \frac{N}{2} \rfloor+1+m}&\big(\sum_{|\alpha|\leq \lfloor \frac{N}{2} \rfloor}\|\langle x \rangle^{m}\partial^{\alpha}u_0\|_{L^{\infty}_x}\\
&+\sum_{\lfloor \frac{N}{2} \rfloor< |\alpha|\leq M}\|\langle x \rangle^m\partial^{\alpha}u_0\|_{L^2_x}+\|J^{M+M_0-N}u_0\|_{L^2_x}+\mathcal{G}_1(\lambda,R)\mathcal{G}_2(\lambda,R)\big)\leq \frac{\lambda}{2}.
\end{aligned}
\end{equation}
Then the inequality \eqref{eqmainTH22} yields
\begin{equation*}
\inf_{(x,t)\in \mathbb{R}\times [0,T]}|\langle x \rangle^m \Phi(u)(x,t)|\geq \frac{\lambda}{2}.
\end{equation*}
We conclude that $\Phi$ maps $\mathfrak{X}_T(R,\lambda)$ into itself. 
\\ \\
Next, we show that for some small $T>0$, $\Phi$ defines a contraction on $\mathfrak{X}_T(R,\lambda)$. But first, to compute the difference $N(u)-N(v)$ in the space $\mathfrak{X}_T(R,\lambda)$, we require the following result.

\begin{prop}\label{propdiffesolu}
Let $b\in \mathbb{R}$, $1\leq r \leq \infty$, and $u, v \in \mathfrak{X}_T(R,\lambda)$. Let $\alpha$ be a multi-index of order $|\alpha|\leq M+M_0-N$, then we have
\begin{equation}\label{eqprop0.1}
\begin{aligned}
\|\langle x \rangle^{b}&\partial^{\alpha}(|u|^p-|v|^{p})\|_{L^{\infty}_{T}L^{r}_x}\\
\lesssim & \big(R^{p-1}+ \lambda^{-(6-2p)}R^{5-p}\big)\|\langle x \rangle^{b-mp}\|_{L^r_x}\|u-v\|_{L^{\infty}_T\mathfrak{X}} \\
&+\sum_{k=1}^{|\alpha|} \lambda^{-2(2k-p)}R^{4k-p-2}\Big(R\|\langle x \rangle^{b-mp}\|_{L^r_x}+\sum_{\lfloor \frac{N}{2}\rfloor<|\beta| \leq |\alpha|}\|\langle x \rangle^{b+m(1-p)}\partial^{\beta}u\|_{L^{\infty}_T L^r_x}\Big)\|u-v\|_{L^{\infty}_T\mathfrak{X}}\\
&+\sum_{k=1}^{|\alpha|} \lambda^{-(2k-p)} \Big(R^{2k-1}\big(\|u-v\|_{L^{\infty}_T\mathfrak{X}}\|\langle x \rangle^{b-mp}\|_{L^r_x}+\sum_{\lfloor \frac{N}{2}\rfloor<|\beta| \leq |\alpha|}\|\langle x \rangle^{b+m(1-p)}\partial^{\beta}(u-v)\|_{L^{\infty}_T L^r_x}\big) \\
& \hspace{1cm} + R^{2k-2}\|u-v\|_{L^{\infty}_T\mathfrak{X}} \big(\sum_{\lfloor \frac{N}{2}\rfloor<|\beta|\leq |\alpha|}\|\langle x \rangle^{b+m(1-p)}\partial^{\beta}u\|_{L^{\infty}_{T}L^r_x}+\|\langle x \rangle^{b+m(1-p)}\partial^{\beta}v\|_{L^{\infty}_{T}L^r_x}\big)\Big).
\end{aligned}
\end{equation}
Additionally,
\begin{equation}\label{eqprop0.2}
\begin{aligned}
\|\langle x \rangle^{b}&\partial^{\alpha}(|u|^{p-2}u-|v|^{p-2}v)\|_{L^{\infty}_TL^{2}_x}\\
\lesssim & \lambda^{-(6-2p)}R^{4-p}\|\langle x \rangle^{b-m(p-1)}\|_{L^{\infty}_T L^r_x}\|u-v\|_{L^{\infty}_T\mathfrak{X}} +\sum_{k=1}^{|\alpha|} \lambda^{-2(2(k+1)-p)}R^{4k-p+1}\Big(R\|\langle x \rangle^{b-m(p-1)}\|_{L^r_x}\\
&\hspace{4cm}+\sum_{\lfloor \frac{N}{2}\rfloor<|\beta| \leq |\alpha|}\|\langle x \rangle^{b+m(2-p)}\partial^{\beta}u\|_{L^{\infty}_T L^r_x}\Big)\|u-v\|_{L^{\infty}_T\mathfrak{X}}\\
&+\sum_{k=1}^{|\alpha|} \lambda^{-(2(k+1)-p)} \Big(R^{2k}\big(\|u-v\|_{L^{\infty}_T\mathfrak{X}}\|\langle x \rangle^{b-m(p-1)}\|_{L^r_x}+\sum_{\lfloor \frac{N}{2}\rfloor<|\beta| \leq |\alpha|}\|\langle x \rangle^{b+m(2-p)}\partial^{\beta}(u-v)\|_{L^{\infty}_T L^r_x}\big) \\
& \hspace{1cm} + R^{2k-1}\|u-v\|_{L^{\infty}_T\mathfrak{X}} \big(\sum_{\lfloor \frac{N}{2}\rfloor<|\beta|\leq |\alpha|}\|\langle x \rangle^{b+m(2-p)}\partial^{\beta}u\|_{L^{\infty}_{T}L^r_x}+\|\langle x \rangle^{b+m(2-p)}\partial^{\beta}v\|_{L^{\infty}_{T}L^r_x}\big)\Big).
\end{aligned}
\end{equation}
\end{prop}
\begin{proof}
Let us first consider $|\alpha|=0$. Using the identity
\begin{equation}\label{eqprop1}
\begin{aligned}
||u|^{p}-|v|^p|\lesssim \max\{|u|^{p-1},|v|^{p-1}\}|u-v|,
\end{aligned}
\end{equation}
we have
\begin{equation*}
\begin{aligned}
\|\langle x \rangle^{d}(|u|^{p}-|v|^p)\|_{L^{\infty}_{T} L^r_x}\lesssim \big(\|\langle x \rangle^m u\|_{L^{\infty}_{T,x}}^{p-1}+\|\langle x \rangle^m v\|_{L^{\infty}_{T,x}}^{p-1}\big)\|\langle x \rangle^{d-m(p-1)}\big(u-v\big)\|_{L^{\infty}_T L^r_x}.
\end{aligned}
\end{equation*}
Now we deal with \eqref{eqprop0.1} for any multi-index $\alpha$ of order $|\alpha|\geq 1$. By identity \eqref{identi1},
\begin{equation}\label{eqprop2}
\begin{aligned}
\partial^{\alpha}(|u|^{p}-|v|^p)=&\sum_{k=1}^{|\alpha|}\big(|u|^{p-2k}-|v|^{p-2k}\big)\Big(\sum_{\substack{\beta_1 + \dots + \beta_k=\alpha \\ |\beta_j|\geq 1, \, \, 1\leq j\leq k}} c_{\beta_1,\dots,\beta_k} \partial^{\beta_1}(|u|^2)\dots \partial^{\beta_k}(|u|^2) \Big)\\
&+ \sum_{k=1}^{|\alpha|}|v|^{p-2k}\sum_{l=1}^{k}\Big(\sum_{\substack{\beta_1 + \dots + \beta_k=\alpha \\ |\beta_j|\geq 1, \, \, 1\leq j\leq k}} c_{\beta_1,\dots,\beta_k} \partial^{\beta_1}(|v|^2)\dots \partial^{\beta_{l-1}}(|v|^2) \\
&\qquad \qquad \qquad \qquad \big(\partial^{\beta_l}(|u|^2)-\partial^{\beta_l}(|v|^2)\big)\partial^{\beta_{l+1}}(|u|^2)\dots \partial^{\beta_k}(|u|^2) \Big) \\
=:& \, \mathcal{C}_1+\mathcal{C}_2.
\end{aligned}
\end{equation}
To bound the term $\mathcal{C}_1$, since $|u(x,t)|^{-1}, |v(x,t)|^{-1}\lesssim \lambda^{-1}\langle x\rangle^m$, it follows that
\begin{equation}\label{eqprop3}
\begin{aligned}
||u|^{p-2}-|v|^{p-2}|=&||u|^{p-4}|u|^2-|v|^{p-4}|v|^2|\\
\lesssim & ||u|^{p-4}(|u|^2-|v|^2)|+|v|^2||u|^{p-4}-|v|^{p-4}| \\
\lesssim & \lambda^{-(4-p)}\langle x \rangle^{m(4-p)}(|u|+|v|)(|u-v|)+\frac{1}{|u|^{4-p}|v|^{2-p}}||v|^{4-p}-|u|^{4-p}| \\
\lesssim & \lambda^{-(4-p)}\langle x \rangle^{m(4-p)}(|u|+|v|)(|u-v|)+\lambda^{-(6-2p)}\langle x \rangle^{m(6-2p)}\max\{|u|^{3-p},|v|^{3-p}\}|u-v|.
\end{aligned}
\end{equation}
Similarly, for $k\geq 2$, we deduce
\begin{equation}\label{eqprop4}
\begin{aligned}
||u|^{p-2k}-|v|^{p-2k}|&=\frac{||v|^{2k-p}-|u|^{2k-p}|}{|u|^{2k-p}|v|^{2k-p}}\\
&\lesssim \lambda^{-2(2k-p)}\langle x\rangle^{2m(2k-p)}\max\{|u|^{2k-p-1},|v|^{2k-p-1}\}|u-v|.
\end{aligned}
\end{equation}
Hence, by \eqref{eqprop3} and \eqref{eqprop4}, we have
\begin{equation*}
\begin{aligned}
\|\langle x \rangle^{b}\mathcal{C}_1\|_{L^{\infty}_T L^r_x} \lesssim & \lambda^{-(4-p)}\big(\|\langle x\rangle^m u\|_{L^{\infty}_{T,x}}+\|\langle x\rangle^m v\|_{L^{\infty}_{T,x}}\big)\big(\sup_{|\gamma|\leq 1}\|\langle x \rangle^m\partial^{\gamma}u\|_{L^{\infty}_{T,x}}\big)^2\\
& \qquad \times \|\langle x \rangle^{m}(u-v)\|_{L^{\infty}_{T,x}}\|\langle x \rangle^{b-mp}\|_{L^r_x}\\
 & +\lambda^{-(6-2p)}\big(\|\langle x\rangle^m u\|_{L^{\infty}_{T,x}}^{3-p}+\|\langle x\rangle^m v\|_{L^{\infty}_{T,x}}^{3-p}\big)\big(\sup_{|\gamma|\leq 1}\|\langle x \rangle^m\partial^{\gamma}u\|_{L^{\infty}_{T,x}}\big)^2\\
& \qquad \times \|\langle x \rangle^{m}(u-v)\|_{L^{\infty}_{T,x}}\|\langle x \rangle^{b-mp}\|_{L^r_x}\\
 & +\sum_{k=2}^{|\alpha|}\lambda^{-2(2k-p)}\big(\|\langle x\rangle^m u\|_{L^{\infty}_{T,x}}^{2k-p-1}+\|\langle x\rangle^m v\|_{L^{\infty}_{T,x}}^{2k-p-1} \big)\|\langle x \rangle^{m}(u-v)\|_{L^{\infty}_{T,x}}\\
& \qquad \times \sum_{\substack{\beta_1 + \dots + \beta_k=\alpha\\ |\beta_j|\geq 1 \, \, 1\leq j \leq k}}\|\langle x \rangle^{b+m(2k-p)} \partial^{\beta_1}(|u|^2)\dots \partial^{\beta_k}(|u|^2)\|_{L^{\infty}_T L^r_x}\\
=:& \, \mathcal{C}_{1,1}+\mathcal{C}_{1,2}+\mathcal{C}_{1,3}.
\end{aligned}
\end{equation*}
Hence, to finish the estimate of $\mathcal{C}_1$, it only remains to control $\mathcal{C}_{1,3}$.  Applying the same partition as in \eqref{eq2propmain1} and the arguments in \eqref{eq3propmain1}-\eqref{eq4propmain1}, we deduce
\begin{equation*}
\begin{aligned}
\mathcal{C}_{1,3}\lesssim &\sum_{k=1}^{|\alpha|}\lambda^{-2(2k-p)}\big(\|\langle x\rangle^m u\|_{L^{\infty}_{T,x}}^{2k-p-1}+\|\langle x\rangle^m v\|_{L^{\infty}_{T,x}}^{2k-p-1} \big)\|\langle x \rangle^{m}(u-v)\|_{L^{\infty}_{T,x}}\\
&\hspace{1cm}\times R^{2k-1}\Big(R\|\langle x \rangle^{b-mp}\|_{L^r_x}+\sum_{\lfloor m\rfloor+\lfloor \frac{N}{2} \rfloor+2 <|\gamma|\leq |\alpha|}\|\langle x \rangle^{b+m(1-p)}\partial^{\gamma} u\|_{L^{\infty}_{T}L^r_x}\Big).
\end{aligned}
\end{equation*}
Next, we deal with $\mathcal{C}_2$. We write
\begin{equation*}
\begin{aligned}
\big(\partial^{\beta}(|u|^2)-\partial^{\beta}(|v|^2)\big)=\sum_{\beta_1+\beta_2=\beta}c_{\beta_1,\beta_2}(\partial^{\beta_1}u-\partial^{\beta_1}v)\partial^{\beta_2}\overline{u}+\partial^{\beta_1}v(\partial^{\beta_2}\overline{u}-\partial^{\beta_2}\overline{v}),
\end{aligned}
\end{equation*}
for any multi-index $|\beta|\geq 1$. Then, by using the above identity and dividing the corresponding sum as in \eqref{eq2propmain1}, we obtain
\begin{equation*}
\begin{aligned}
\|\langle x\rangle^{b}\mathcal{C}_2&\|_{L^{\infty}_T L^r_x}\\
\lesssim & \sum_{k=1}^{|\alpha|}\sum_{l=1}^k \sum_{\substack{\beta_1 + \dots + \beta_k=\alpha\\ |\beta_j|\geq 1 \, \, 1\leq j \leq k}}\lambda^{-(2k-p)}\|\langle x \rangle^{b+m(2k-p)} \partial^{\beta_1}(|v|^2)\dots \partial^{\beta_{l-1}}(|v|^2)\\
& \hspace{4cm}\times\big(\partial^{\beta_l}(|u|^2)-\partial^{\beta_1}(|v|^2)\big)\partial^{\beta_{l+1}}(|u|^2)\dots \partial^{\beta_k}(|u|^2)\|_{L^{\infty}_T L^2_x}\\
\lesssim & \sum_{k=1}^{|\alpha|}\lambda^{-(2k-p)}\Big( R^{2k-1}\big(\sum_{|\beta|\leq \lfloor \frac{N}{2} \rfloor}\|\langle x \rangle^{m}\partial^{\beta}(u-v)\|_{L^{\infty}_{T,x}}\|\langle x \rangle^{b-mp}\|_{L^{r}_x}\\
&\hspace{5cm}+\sum_{\lfloor \frac{N}{2} \rfloor<|\beta|\leq |\alpha|}\|\langle x \rangle^{b+m(1-p)}\partial^{\beta}(u-v)\|_{L^{\infty}_{T}L^r_x}\big)\\
 & +R^{2k-2}\|u-v\|_{L^{\infty}_T\mathfrak{X}}\big(\sum_{\lfloor \frac{N}{2} \rfloor<|\beta|\leq |\alpha|}\|\langle x \rangle^{b+m(1-p)}\partial^{\beta}u\|_{L^{\infty}_{T}L^r_x}+\|\langle x \rangle^{b+m(1-p)}\partial^{\beta}v\|_{L^{\infty}_{T}L^r_x}\big)\Big).
\end{aligned}
\end{equation*}
Hence, the proof of \eqref{eqprop0.1} is complete. 
\\ \\
Next, we turn to \eqref{eqprop0.2}. By identity \eqref{identi2}, we get
\begin{equation*}
\begin{aligned}
\partial^{\alpha}\big(|u|^{p-2}u&-|v|^{p-2}v\big)\\
=&\sum_{k=1}^{|\alpha|} \big(|u|^{p-2(k+1)}-|v|^{p-2(k+1)}\big)\sum_{\substack{\beta_0+\beta_1+\dots+\beta_k=\alpha, \\ |\beta_j|\geq 1, \, \, 1\leq j \leq k}}c_{\beta_0,\dots,\beta_k}\partial^{\beta_0}u\partial^{\beta_1}(|u|^2)\dots \partial^{\beta_k}(|u|^2)\\
&+\sum_{k=1}^{|\alpha|} |v|^{p-2(k+1)}\sum_{\substack{\beta_0+\beta_1+\dots+\beta_k=\alpha, \\ |\beta_j|\geq 1, \, \, 1\leq j \leq k}}\Big(c_{\beta_0,\dots,\beta_k}\big(\partial^{\beta_0}u-\partial^{\beta_0}v\big)\partial^{\beta_1}(|u|^2)\dots \partial^{\beta_k}(|u|^2)\\
&\qquad \quad+\sum_{l=1}^{k}c_{\beta_0,\dots,\beta_k}\partial^{\beta_0}v\partial^{\beta_1}(|v|^2)\dots\partial^{\beta_{l-1}}(|v|^2)\\
&\hspace{4cm}\times\big(\partial^{\beta_l}(|u|^2)-\partial^{\beta_l}(|v|^2)\big)\partial^{\beta_{l+1}}(|u|^2)\dots \partial^{\beta_{k}}(|u|^2)\Big) \\
&+ \big(|u|^{p-2}-|v|^{p-2}\big)\partial^{\alpha}u+|v|^{p-2}(\partial^{\alpha}u-\partial^{\alpha}v)\\
=:& \, \widetilde{\mathcal{C}}_1+\widetilde{\mathcal{C}}_2+\widetilde{\mathcal{C}}_3.
\end{aligned}
\end{equation*}
We proceed to analyze each term $\widetilde{\mathcal{C}}_j$, $1\leq j \leq 3$. Applying \eqref{eqprop4}, and using a similar partition to \eqref{eq7propmain1}, we get 
\begin{equation*}
\begin{aligned}
\|\langle x \rangle^{b}\widetilde{\mathcal{C}}_1\|_{L^{\infty}_T L^r_x}\lesssim & \sum_{k=1}^{|\alpha|} \lambda^{-2(2(k+1)-p)}\big(\|\langle x \rangle^{m}u\|_{L^{\infty}_{T,x}}^{2k-p+1}+\|\langle x \rangle^{m}v\|_{L^{\infty}_{T,x}}^{2k-p+1}\big)\|\langle x \rangle^m(u-v)\|_{L^{\infty}_{T,x}}\\
& \qquad \qquad \times\sum_{\substack{\beta_0+\beta_1+\dots+\beta_k=\alpha, \\ |\beta_j|\geq 1, \, \, 1\leq j \leq k}}\|\langle x\rangle^{b+m(2k+2-p)}\partial^{\beta_0}u\partial^{\beta_1}(|u|^2)\dots \partial^{\beta_k}(|u|^2)\|_{L^{\infty}_{T}L^{r}_x} \\
\lesssim & \sum_{k=1}^{|\alpha|} \lambda^{-2(2(k+1)-p)}\big(\|\langle x \rangle^{m}u\|_{L^{\infty}_{T,x}}^{2k-p+1}+\|\langle x \rangle^{m}v\|_{L^{\infty}_{T,x}}^{2k-p+1}\big)\|\langle x \rangle^m(u-v)\|_{L^{\infty}_{T,x}}\\
& \qquad \qquad \times R^{2k} \Big(R\|\langle x\rangle^{b-m(p-1)}\|_{L^r_x}+\sum_{\lfloor \frac{N}{2}\rfloor<|\beta|\leq |\alpha|}\|\langle x\rangle^{b+m(2-p)}\partial^{\beta}u\|_{L^{\infty}_TL^{r}_x}\Big).
\end{aligned}
\end{equation*}
In a similar manner, we deduce
\begin{equation*}
\begin{aligned}
\|\langle x \rangle^{b} &\widetilde{\mathcal{C}}_2\|_{L^{\infty}_T L^r_x} \\
\lesssim & \sum_{k=1}^{|\alpha|} \lambda^{-(2(k+1)-p)}\sum_{\substack{\beta_0+\beta_1+\dots+\beta_k=\alpha, \\ |\beta_j|\geq 1, \, \, 1\leq j \leq k}}\Big(\|\langle x\rangle^{b+m(2k+2-p)}(\partial^{\beta_0}u-\partial^{\beta_0}u)\partial^{\beta_1}(|u|^2)\dots \partial^{\beta_k}(|u|^2)\|_{L^{\infty}_TL^{r}_x} \\
& \qquad +\sum_{l=1}^{k}\|\langle x\rangle^{b+m(2k+2-p)} \partial^{\beta_0}v\partial^{\beta_1}(|v|^2)\dots\partial^{\beta_{l-1}}(|v|^2)\\
&\hspace{4cm}\times\big(\partial^{\beta_l}(|u|^2)-\partial^{\beta_l}(|v|^2)\big)\partial^{\beta_{l+1}}(|u|^2)\dots \partial^{\beta_{k}}(|u|^2)\|_{L^{\infty}_TL^{r}_x}\Big)\\
\lesssim & \sum_{k=1}^{|\alpha|} \lambda^{-(2(k+1)-p)}\Big(R^{2k}\big(\|u-v\|_{L^{\infty}_T\mathfrak{X}}\|\langle x \rangle^{b-m(p-1)}\|_{L^r_x}\\
& \hspace{3cm}+ \sum_{\lfloor\frac{N}{2} \rfloor<|\beta|\leq |\alpha|}\|\langle x \rangle^{b+m(2-p)}\partial^{\beta}(u-v)\|_{L^{\infty}_TL^r_x}\big)\\
& \hspace{1cm}+R^{2k-1}\|u-v\|_{L^{\infty}_T\mathfrak{X}}\sum_{\lfloor\frac{N}{2} \rfloor<|\beta|\leq |\alpha|}\big(\|\langle x \rangle^{b+m(2-p)}\partial^{\beta}u\|_{L^{\infty}_TL^r_x}+\|\langle x \rangle^{b+m(2-p)}\partial^{\beta}v\|_{L^{\infty}_TL^r_x}\big)\Big).
\end{aligned}
\end{equation*}
Finally, by \eqref{eqprop3}, we find 
\begin{equation*}
\begin{aligned}
\|\langle x \rangle^b\widetilde{\mathcal{C}}_3\|_{L^{\infty}_T L^r_x}\lesssim & \lambda^{-(4-p)}\big(\|\langle x \rangle^m u\|_{L^{\infty}_{T,x}}+\|\langle x \rangle^m v\|_{L^{\infty}_{T,x}}\big)\|\langle x \rangle^{b+m(2-p)}\partial^{\alpha}u\|_{L^{\infty}_T L^r_x}\|\langle x \rangle^m (u-v)\|_{L^{\infty}_{T,x}}\\
&+\lambda^{-(6-2p)}\big(\|\langle x \rangle^m u\|_{L^{\infty}_{T,x}}^{3-p}+\|\langle x \rangle^m v\|_{L^{\infty}_{T,x}}^{3-p}\big)\|\langle x\rangle^{b+m(2-p)}\partial^{\alpha}u\|_{L^{\infty}_T L^r_x}\|\langle x \rangle^m(u-v)\|_{L^{\infty}_{T,x}} \\
&+\lambda^{-(2-p)}\|\langle x \rangle^{b+m(2-p)}\partial^{\alpha}(u-v)\|_{L^{\infty}_T L^r_x},
\end{aligned}
\end{equation*}
and since
\begin{equation*}
\begin{aligned}
\|&\langle x \rangle^{b+m(2-p)}\partial^{\alpha}u\|_{L^{\infty}_T L^r_x}\\
&\lesssim \sum_{|\beta|\leq \lfloor \frac{N}{2} \rfloor}\|\langle x\rangle^m\partial^{\beta}u\|_{L^{\infty}_{T,x}}\|\langle x \rangle^{b-m(p-1)}\|_{L^{\infty}_T L^r_x} +\sum_{\lfloor \frac{N}{2} \rfloor<|\beta|\leq |\alpha|}\|\langle x\rangle^{b+m(2-p)}\partial^{\beta}u\|_{L^{\infty}_{T}L^r_x},
\end{aligned}
\end{equation*}
the deduction of \eqref{eqprop0.2} is complete.
\end{proof}

Now, we compute the difference $\Phi(u)-\Phi(v)$ in the space $\mathfrak{X}_T(R,\lambda)$. We divide our consideration into several cases depending on each component in the norm \eqref{normspace} used in the space $\mathfrak{X}_T(R,\lambda)$.

\subsubsection*{\bf Estimate in the space $L^{\infty}([0,T];W^{\lfloor \frac{N}{2} \rfloor,\infty}(\mathbb{R}^N;\langle x \rangle^m \, dx))$.}

Let $\beta$ be a multi-index of order $|\beta|\leq \lfloor \frac{N}{2} \rfloor$ and $u,v \in \mathfrak{X}_T(R,\lambda)$. By the arguments leading to \eqref{eqmainTH0}, we find
\begin{equation}\label{eqmainTH22.1}
\begin{aligned}
\sup_{t\in[0,T]}&\|\langle x \rangle^m \partial^{\beta}\big(\Phi(u)-\Phi(v)\big)\|_{L^{\infty}_x} \\
\lesssim &|T|\langle T \rangle^{\lfloor \frac{N}{2} \rfloor+1+m} \Big(\sum_{|\alpha|\leq \lfloor \frac{N}{2} \rfloor}\|\langle x \rangle^m \partial^{\alpha}\big(N(u)-N(v)\big)\|_{L^{\infty}_{T,x}}\\
&\hspace{1cm}+\sum_{\lfloor \frac{N}{2} \rfloor<|\alpha|\leq M}\|\langle x \rangle^m\partial^{\alpha}\big(N(u)-N(v)\big)\|_{L^{\infty}_TL^2_x}+\|J^{M+M_0-N}\big(N(u)-N(v)\big)\|_{L^{\infty}_TL^2_x}\Big).
\end{aligned}
\end{equation}
We proceed to estimate each term on the right-hand side of the above inequality.
\\ \\
{\bf Estimate for $\|\langle x \rangle^m \partial^{\alpha} \big(N(u)-N(v)\big)\|_{L^{\infty}_{T,x}}$, $|\alpha|\leq \lfloor \frac{N}{2} \rfloor$}. We divide our analysis as follows
\begin{equation*}
\begin{aligned}
\|\langle x \rangle^{m}\partial^{\alpha}&(N(u)-N(v))\|_{L^{\infty}_{T,x}}\\
\lesssim & \|\langle x \rangle^{m}\partial^{\alpha}\big((|x|^{-(N-\gamma)}\ast (|u|^p-|v|^p))|u|^{p-2}u\big)\|_{L^{\infty}_{T,x}}\\
&+\|\langle x \rangle^m\partial^{\alpha}\big(\big(|x|^{-(N-\gamma)}\ast |v|^p\big)\big(|u|^{p-2}u-|v|^{p-2}v\big)\big)\|_{L^{\infty}_{T,x}}\\
\lesssim &\sum_{\alpha_1+\alpha_2=\alpha} \|\langle x \rangle^{m(2-p)}\partial^{\alpha_1}\big(|x|^{-(N-\gamma)}\ast (|u|^p-|v|^p)\big)\|_{L^{\infty}_{T,x}}\|\langle x \rangle^{m(p-1)}\partial^{\alpha_2}(|u|^{p-2}u)\|_{L^{\infty}_{T,x}}\\
& \qquad+\|\langle x \rangle^{m(2-p)}\partial^{\alpha_1}\big(|x|^{-(N-\gamma)}\ast |v|^p\big)\|_{L^{\infty}_{T,x}}\|\langle x \rangle^{m(p-1)}\partial^{\alpha_2}(|u|^{p-2}u-|v|^{p-2}v)\|_{L^{\infty}_{T,x}}.
\end{aligned}
\end{equation*}
By applying the same arguments in \eqref{eqmainTH1.0} and Proposition \ref{propdiffesolu}, yields
\begin{equation}\label{eqmainTH23}
\begin{aligned}
\|\langle x \rangle^{m(2-p)}\partial^{\alpha_1}&\big(|x|^{-(N-\gamma)}\ast (|u|^p-|v|^p)\big)\|_{L^{\infty}_{T,x}} \\
\lesssim & 
\sum_{ |\beta|\leq 2\lfloor \frac{N}{2} \rfloor+1}\|\langle x \rangle^{m(2-p)}\partial^{\beta}(|u|^p-|v|^p)\|_{L^{\infty}_{T}L^2_x}\\
\lesssim & \Big(R^{p-1}+\lambda^{-(6-2p)}R^{5-p}+\sum_{k=1}^{2\lfloor \frac{N}{2} \rfloor+1}\lambda^{-2(2k-p)}R^{4k-p-1}+\lambda^{-(2k-p)}R^{2k-1} \Big)\|u-v\|_{L^{\infty}_T\mathfrak{X}}.
\end{aligned}
\end{equation}
Similarly, an application of Proposition \ref{propdiffesolu} allows us to conclude
\begin{equation}\label{eqmainTH24}
\begin{aligned}
\|\langle x \rangle^{m(p-1)}\partial^{\alpha_2}&(|u|^{p-2}u-|v|^{p-2}v)\|_{L^{\infty}_{T,x}}\\
\lesssim & \Big( \lambda^{-(6-2p)}R^{4-p}+\sum_{k=1}^{\lfloor \frac{N}{2} \rfloor} \lambda^{-2(2(k+1)-p)}R^{4k-p+2}+\lambda^{-(2(k+1)-p)}R^{2k}\Big)\|u-v\|_{L^{\infty}_T\mathfrak{X}}.
\end{aligned}
\end{equation}
We combine \eqref{eqmainTH1.1}, \eqref{eqmainTH1.2}, \eqref{eqmainTH23} and \eqref{eqmainTH24} to deduce
\begin{equation}\label{eqmainTH24.1}
\begin{aligned}
\|\langle x \rangle^{m}\partial^{\alpha}&(N(u)-N(v))\|_{L^{\infty}_{T,x}}\lesssim \Big(\mathcal{J}_1(\lambda,R)\mathcal{G}_2(\lambda,R)+\mathcal{G}_1(\lambda,R)\mathcal{J}_2(\lambda,R)\Big)\|u-v\|_{L^{\infty}_T\mathfrak{X}},
\end{aligned}
\end{equation}
where $\mathcal{G}_1(\lambda,R)$ and $\mathcal{G}_2(\lambda,R)$ are defined in \eqref{eqmainTH18} and we have set
\begin{equation}\label{Jdefi}
\begin{aligned}
&\mathcal{J}_1(\lambda,R)=R^{p-1}+\lambda^{-(6-2p)}R^{5-p}+\sum_{k=1}^{M+M_0-N}\lambda^{-2(2k-p)}R^{4k-p-1}+\lambda^{-(2k-p)}R^{2k-1}, \\
&\mathcal{J}_2(\lambda,R)= \lambda^{-(6-2p)}R^{4-p}+\sum_{k=1}^{M+M_0-N} \lambda^{-2(2(k+1)-p)}R^{4k-p+2}+\lambda^{-(2(k+1)-p)}R^{2k}.
\end{aligned}
\end{equation}
\\ \\
{\bf Estimate for $\|\langle x \rangle^m \partial^{\alpha}\big(N(u)-N(v)\big)\|_{L^{\infty}_{T}L^2_x}$, $ \lfloor \frac{N}{2} \rfloor<|\alpha|\leq M$}. We use Leibniz's rule to get
\begin{equation}\label{eqmainTH25}
\begin{aligned}
\|\langle x \rangle^m\partial^{\alpha}\big(N(u)-N(v)\big)\|_{L^{\infty}_TL^{2}_x}\lesssim \sum_{\alpha_1+\alpha_2=\alpha}&\|\langle x \rangle^m \partial^{\alpha_1}\big(|x|^{-(N-\gamma)}\ast (|u|^{p}-|v|^{p})\big)\partial^{\alpha_2}\big(|u|^{p-2}u\big)\|_{L^{\infty}_T L^2_x}\\
&+\|\langle x \rangle^m \partial^{\alpha_1}\big(|x|^{-(N-\gamma)}\ast |v|^{p}\big)\partial^{\alpha_2}\big(|u|^{p-2}u-|v|^{p-2}v\big)\|_{L^{\infty}_T L^2_x}.
\end{aligned}
\end{equation}
Let us further divide our analysis according to the magnitude of the multi-index $\alpha_1$ in \eqref{eqmainTH3}.

\underline{Case $|\alpha_1| <N$}. Recalling the  conditions \eqref{eqcondi1}-\eqref{eqcondi4}, and the arguments leading to \eqref{eqmainTH4.1}, we apply Proposition \ref{propdiffesolu} to deduce
\begin{equation*}
\begin{aligned}
\|\langle x \rangle^{m\kappa}& \big(|x|^{-(N-\gamma)}\ast \partial^{\alpha_1}(|u|^{p}-|v|^{p})\big)\|_{L^{\infty}_{T}L^{r_1}_x} \\
\lesssim &  \|\langle x \rangle^{m\kappa}\partial^{\alpha_1}(|u|^{p}-|v|^{p})\|_{L^{\infty}_{T}L^{q_2}_x}\\
\lesssim & \Big( R^{p-1}+\lambda^{-(6-2p)}R^{5-p}\|\langle x \rangle^{-m(p-\kappa)}\|_{L^{q_2}_x}\\
& \qquad +\|\langle x \rangle^{-m(p-\kappa)}\|_{L^{q_2}_x}\sum_{k=1}^{M+M_0-N}\lambda^{-2(2k-p)}R^{4k-p-1}+\lambda^{-(2k-p)}R^{2k-1}\Big)\|u-v\|_{L^{\infty}_T\mathfrak{X}}\\
\lesssim & \mathcal{J}_1(\lambda,R)\|u-v\|_{L^{\infty}_T\mathfrak{X}}.
\end{aligned}
\end{equation*}
Likewise, by similar arguments as in \eqref{eqmainTH4} and using Proposition \ref{propdiffesolu}, we deduce 
\begin{equation*}
\begin{aligned}
\|\langle x \rangle^{m(1-\kappa)}\partial^{\alpha_2}\big(|u|^{p-2}u-|v|^{p-2}v\big)\|_{L^{\infty}_T L^{r_2}_x}\lesssim \mathcal{J}_2(\lambda,R)\|u-v\|_{L^{\infty}_T\mathfrak{X}}.
\end{aligned}
\end{equation*}
Then, we collect the previous results and similar considerations as in \eqref{eqmainTH6} to find
\begin{equation}\label{eqmainTH25.1}
\begin{aligned}
\sum_{\substack{\alpha_1+\alpha_2=\alpha}}&\|\langle x \rangle^m \partial^{\alpha_1}\big(|x|^{-(N-\gamma)}\ast (|u|^{p}-|v|^{p})\big)\partial^{\alpha_2}\big(|u|^{p-2}u\big)\|_{L^{\infty}_T L^2_x}\\
&+\|\langle x \rangle^m \partial^{\alpha_1}\big(|x|^{-(N-\gamma)}\ast |v|^{p}\big)\partial^{\alpha_2}\big(|u|^{p-2}u-|v|^{p-2}v\big)\|_{L^{\infty}_T L^2_x}\\
\lesssim & \Big(\mathcal{J}_1(\lambda,R)\mathcal{G}_{2}(\lambda,R)+\mathcal{G}_{1}(\lambda,R)\mathcal{J}_2(\lambda,R)\Big)\|u-v\|_{L^{\infty}_T\mathfrak{X}},
\end{aligned}
\end{equation}
for all $|\alpha_1|\leq N$.
\\ \\
\underline{Case $N\leq |\alpha_1| \leq M$}. By using \eqref{eqmainTH7} and \eqref{eqmainTH12}, we find
\begin{equation}\label{eqmainTH26}
\begin{aligned}
\|\langle x \rangle^m & \partial^{\alpha_1}\big(|x|^{-(N-\gamma)}\ast (|u|^{p}-|v|^{p})\big)\partial^{\alpha_2}\big(|u|^{p-2}u\big)\|_{L^{\infty}_T L^2_x}\\
&\hspace{0.5cm}+\|\langle x \rangle^m \partial^{\alpha_1}\big(|x|^{-(N-\gamma)}\ast |v|^{p}\big)\partial^{\alpha_2}\big(|u|^{p-2}u-|v|^{p-2}v\big)\|_{L^{\infty}_T L^2_x}\\
&\lesssim \|\langle x \rangle^{m(2-p)} \partial^{\alpha_1}\big(|x|^{-(N-\gamma)}\ast (|u|^{p}-|v|^{p})\big)\|_{L^{\infty}_T L^2_x}\mathcal{G}_2(\lambda,R)\\
&\hspace{1cm}+\mathcal{G}_1(\lambda,R)\|\langle x \rangle^{m(p-1)}\partial^{\alpha_2}\big(|u|^{p-2}u-|v|^{p-2}v\big)\|_{L^{\infty}_{T,x}}.
\end{aligned}
\end{equation}
Now, by the same line of arguments as in the estimates \eqref{eqmainTH8.1}-\eqref{eqmainTH11}, it follows that
\begin{equation*}
\begin{aligned}
\|\langle x \rangle^{m(2-p)} \partial^{\alpha_1} \big(|x|^{-(N-\gamma)}\ast (|u|^{p}-|v|^{p})\big)\|_{L^{\infty}_T L^2_x}\lesssim & \sup_{|\beta|\leq M} \|\langle x \rangle^{m(2-p)+m\epsilon}\partial^{\beta}(|u|^{p}-|v|^{p})\|_{L^{\infty}_T L^2_x}\\
 &+ \|J^{M+M_0-N}(|u|^{p}-|v|^{p})\|_{L^{\infty}_T L^2_x},
\end{aligned}
\end{equation*}
where $\epsilon>0$ is given by \eqref{epsicond}. Since $\langle x \rangle^{-2m(p-1)+m\epsilon}\in L^2(\mathbb{R}^N)$, Proposition \ref{propdiffesolu} yields
\begin{equation*}
\sup_{|\beta|\leq M} \|\langle x \rangle^{m(2-p)+m\epsilon}\partial^{\beta}(|u|^{p}-|v|^{p})\|_{L^{\infty}_T L^2_x}\lesssim \mathcal{J}_1(\lambda,R)\|u-v\|_{L^{\infty}_T\mathfrak{X}}.
\end{equation*}
A further application of Proposition \ref{propdiffesolu} gives
\begin{equation*}
\sup_{|\beta|\leq M} \|J^{M+M_0-N}(|u|^{p}-|v|^{p})\|_{L^{\infty}_T L^2_x}\lesssim \mathcal{J}_1(\lambda,R)\|u-v\|_{L^{\infty}_T\mathfrak{X}}.
\end{equation*}
Finally, we use \eqref{eqprop0.2} to deduce
\begin{equation*}
\|\langle x \rangle^{m(p-1)}\partial^{\alpha_2}\big(|u|^{p-2}u-|v|^{p-2}v\big)\|_{L^{\infty}_{T,x}}\lesssim \mathcal{J}_2(\lambda,R)\|u-v\|_{L^{\infty}_T\mathfrak{X}}.
\end{equation*}
Plugging the previous estimates into \eqref{eqmainTH26} allows us to conclude that \eqref{eqmainTH25.1} holds true for all $|\alpha|\leq M$. This in turn provides the estimate 
\begin{equation} \label{eqmainTH27}             
\sum_{\lfloor \frac{N}{2} \rfloor<|\alpha|\leq M}\|\langle x \rangle^m\partial^{\alpha}\big(N(u)-N(v)\big)\|_{L^{\infty}_TL^{2}_x}\lesssim \Big(\mathcal{J}_1(\lambda,R)\mathcal{G}_2(\lambda,R)+\mathcal{G}_1(\lambda,R)\mathcal{J}_2(\lambda,R)\Big)\|u-v\|_{L^{\infty}_T\mathfrak{X}}.
\end{equation}
\\ \\
{\bf Estimate for $\|J^{M+M_0-N}\big(N(u)-N(v)\big)\|_{L^{\infty}_{T}L^2_x}$}. Following the same reasoning leading to \eqref{eqmainTH14}, using Proposition \ref{propdiffesolu}, we deduce 
\begin{equation}\label{eqmainTH28}
\begin{aligned}
\|J^{M+M_0-N} \big(N(u)-N(v)\big)\|_{L^{\infty}_{T}L^2_x} \lesssim \Big(\mathcal{J}_1(\lambda,R)\mathcal{G}_2(\lambda,R)+\mathcal{G}_1(\lambda,R)\mathcal{J}_2(\lambda,R)\Big)\|u-v\|_{L^{\infty}_T\mathfrak{X}}.
\end{aligned}
\end{equation}
To avoid falling into repetition, we omit the deduction of the above estimate.

Combining \eqref{eqmainTH24.1}, \eqref{eqmainTH27}, \eqref{eqmainTH28} and \eqref{eqmainTH22.1}, we find
\begin{equation*}
\begin{aligned}
\sup_{t\in[0,T]}\|\langle x \rangle^m \partial^{\beta}&\big(\Phi(u)-\Phi(v)\big)\|_{L^{\infty}_x} \\
\lesssim &|T|\langle T \rangle^{\lfloor \frac{N}{2} \rfloor+1+m} \Big(\mathcal{J}_1(\lambda,R)\mathcal{G}_2(\lambda,R)+\mathcal{G}_1(\lambda,R)\mathcal{J}_2(\lambda,R)\Big)\|u-v\|_{L^{\infty}_T\mathfrak{X}}.
\end{aligned}
\end{equation*}
This completes the required analysis on the space $L^{\infty}([0,T];W^{\lfloor \frac{N}{2} \rfloor,\infty}(\mathbb{R}^N;\langle x \rangle^m \, dx))$.
\\ \\
Next, by the inequalities \eqref{eqmainTH24.1}, \eqref{eqmainTH27} and \eqref{eqmainTH28}, we can argue as in \eqref{eqmainTH16}  and  \eqref{eqmainTH17} to control the remaining terms of the $\|\cdot\|_{\mathfrak{X}}$-norm for the difference $\Phi(u)-\Phi(v)$, $u,v \in \mathfrak{X}_T(R,\lambda)$. Thus, we get
\begin{equation*}
\begin{aligned}
\sum_{\lfloor \frac{N}{2} \rfloor<|\alpha|\leq M}&\|\langle x \rangle^m\partial^{\alpha}\big(\Phi(u)-\Phi(v)\big)\|_{L^{\infty}_T L^{2}_x}+\|J^{M+M_0-N}\big(\Phi(u)-\Phi(v)\big)\|_{L^{\infty}_TL^{2}_x} \\
&\lesssim |T|\langle T \rangle^m \Big(\mathcal{J}_1(\lambda,R)\mathcal{G}_2(\lambda,R)+\mathcal{G}_1(\lambda,R)\mathcal{J}_2(\lambda,R)\Big)\|u-v\|_{L^{\infty}_T\mathfrak{X}}.
\end{aligned}
\end{equation*}
Consequently, there exists a constant $c_2>0$ such that the above estimates are summarized as follows
\begin{equation*}
\begin{aligned}
\|\Phi(u)-\Phi(v)\|_{L^{\infty}_T\mathfrak{X}}\leq c_3|T|\langle T \rangle^{\lfloor \frac{N}{2}\rfloor+1+m} \Big(\mathcal{J}_1(\lambda,R)\mathcal{G}_2(\lambda,R)+\mathcal{G}_1(\lambda,R)\mathcal{J}_2(\lambda,R)\Big)\|u-v\|_{L^{\infty}_T\mathfrak{X}},
\end{aligned}
\end{equation*}
where $\mathcal{G}_j$, $\mathcal{J}_j$, $j=1,2$ are defined in \eqref{eqmainTH18} and \eqref{Jdefi}, respectively. We take $T>0$ sufficiently small satisfying \eqref{eqmainTH19}, \eqref{eqmainTH22.0} and such that
\begin{equation}\label{eqmainTH29}
c_3|T|\langle T \rangle^{\lfloor \frac{N}{2}\rfloor+1+m} \Big(\mathcal{J}_1(\lambda,R)\mathcal{G}_2(\lambda,R)+\mathcal{G}_1(\lambda,R)\mathcal{J}_2(\lambda,R)\Big) <1.
\end{equation}
Consequently, $\Phi$ is a contraction on $\mathfrak{X}_{T}(R,\lambda)$. Thus, there exists a unique fixed point solving \eqref{gHartree}. The remaining properties stated in Theorem \ref{mainTHM} are deduced by standard arguments, thus, we omit their proofs.


\section{Theorem \ref{scatres}}\label{S:scat}

We apply pseudo-conformal transformation in the spirit of \cite{CAZENAVE2011}. Via this transformation a global solution of \eqref{gHartree} corresponds to a solution of the nonautonomous equation
\begin{equation}\label{agHartree}
\left\{\begin{aligned}
    &i v_t+\Delta v+\mu (1-bt)^{N(p-1)-2-\gamma}\Big(\frac{1}{|x|^{N-\gamma}}\ast |v|^{p}\Big)|v|^{p-2}v=0, \, \,  x \in \mathbb{R}^N, \, \, t\in \mathbb{R},\\
    & v(x,0)=v_0(x), 
\end{aligned}\right. 
\end{equation}
where $b,\mu \in \mathbb{C}\setminus\{0\}$. We solve the initial value problem \eqref{agHartree} in the class $\mathfrak{X}$ determined by \eqref{spacedef} and \eqref{normspace}.

\begin{prop}\label{propglobal}
Let $\frac{4}{3}<p<2$, $0<\gamma<\min\{\frac{N(3p-4)}{2p},N(p-1)-1\}$, and $m\in \mathbb{R}^{+}$, $M_0, M \in \mathbb{Z}^{+}$ satisfy \eqref{condmainT1}-\eqref{condmainT3}. Additionally, let $v_0\in \mathfrak{X}$ such that
\begin{equation}\label{lowebound}
\inf_{x \in \mathbb{R}} |v_0(x)|\geq \lambda.
\end{equation}
Then there exists a unique solution $v$ of \eqref{agHartree} with initial data $v_0$ such that
\begin{equation}
v\in C([0,|b|^{-1}];\mathfrak{X}),
\end{equation}
provided that $b>0$ is sufficiently large.
\end{prop}

\begin{proof}
The proof is similar to that of Theorem \ref{mainTHM}. We will apply the contraction mapping principle to the integral operator defined by \eqref{agHartree}, that is,
\begin{equation}\label{integralmap2}
\Phi_2(v(t))=e^{it \Delta }v_0+i\mu\int_0^t (1-b t')^{N(p-1)-2-\gamma} e^{i(t-t')\Delta} N(v(t'))\, dt',
\end{equation}
acting on the space $\mathfrak{X}_{|b|^{-1}}(R,\lambda)$ defined by \eqref{functspa}.
\\ \\
Following the same arguments leading to the inequalities \eqref{eqmainTH15}, \eqref{eqmainTH16} and \eqref{eqmainTH17}, there exists a constant $c>0$ such that 
\begin{equation}\label{scatteq1}
\begin{aligned}
\|\Phi_2(u)\|_{L^{\infty}_T\mathfrak{X}}\leq & c\langle |b|^{-1} \rangle^{\lfloor \frac{N}{2} \rfloor+1+m}\big(\sum_{|\alpha|\leq \lfloor \frac{N}{2} \rfloor}\|\langle x \rangle^{m}\partial^{\alpha}v_0\|_{L^{\infty}_x}+\sum_{\lfloor \frac{N}{2} \rfloor<|\alpha|\leq M}\|\langle x \rangle^m \partial^{\alpha} v_0\|_{L^2_x}\\
& \hspace{4cm}+\|J^{M+M_0-N}v_0\|_{L^2_x}\big)\\
&+c|b|^{-1}\langle |b|^{-1} \rangle^{\lfloor \frac{N}{2} \rfloor+1+m}\mathcal{G}_1(\lambda,R)\mathcal{G}_2(\lambda,R),
\end{aligned}
\end{equation}  
where $\mathcal{G}_1$ and $\mathcal{G}_2$ are defined by \eqref{eqmainTH18}, and we have used 
\begin{equation}
\int_0^{|b|^{-1}} |1-b t'|^{N(p-1)-2-\gamma} dt'=\frac{|b|^{-1}}{N(p-1)-1-\gamma}.
\end{equation}
Next, by \eqref{eqmainTH20} and arguing as in \eqref{eqmainTH22}, there exists a constant $c_1$ such that 
\begin{equation}
\begin{aligned}
|\langle x \rangle^m \Phi_2(v)(x,t)|\geq & \lambda-c_1|b|^{-1}\langle |b|^{-1} \rangle^{\lfloor \frac{N}{2} \rfloor+1+m}\big(\sum_{|\alpha|\leq \lfloor \frac{N}{2}\rfloor}\|\langle x \rangle^{m}\partial^{\alpha}v_0\|_{L^{\infty}_x}\\
&\hspace{3cm}+\sum_{\lfloor \frac{N}{2} \rfloor< |\alpha|\leq M}\|\langle x \rangle^m\partial^{\alpha}v_0\|_{L^2_x}+\|J^{M+M_0-N}v_0\|_{L^2_x} \big)\\
&-c_1|b|^{-1}\langle |b|^{-1} \rangle^{\lfloor \frac{N}{2} \rfloor+1+m}\mathcal{G}_1(\lambda,R)\mathcal{G}_2(\lambda,R).
\end{aligned}  
\end{equation}
Similarly, the same reasoning in the deduction of \eqref{eqmainTH29} assures the existence of a constant $c_2>0$, for which 
\begin{equation}
\begin{aligned}
\|\Phi_2(u)-\Phi_2(v)\|_{L^{\infty}_T \mathfrak{X}}\leq c_2|b|^{-1}\langle |b|^{-1} \rangle^{\lfloor \frac{N}{2}\rfloor+1+m} \Big(\mathcal{J}_1(\lambda,R)\mathcal{G}_2(\lambda,R)+\mathcal{G}_1(\lambda,R)\mathcal{J}_2(\lambda,R)\Big)\|u-v\|_{L^{\infty}_T\mathfrak{X}},
\end{aligned}
\end{equation}
where $\mathcal{J}_1$ and $\mathcal{J}_2$ are defined in \eqref{Jdefi}. Recalling the constant $c>0$ in \eqref{scatteq1}, we set
\begin{equation}\label{initialdatahy}
R=2c\big(\sum_{|\alpha|\leq \lfloor \frac{N}{2} \rfloor}\|\langle x \rangle^{m}\partial^{\alpha}v_0\|_{L^{\infty}_x}+\sum_{\lfloor \frac{N}{2} \rfloor<|\alpha|\leq M}\|\langle x \rangle^m \partial^{\alpha} v_0\|_{L^2_x}+\|J^{M+M_0-N}v_0\|_{L^2_x}\big).
\end{equation}
Thus, we take $b>0$ large such that
\begin{equation}
\begin{aligned}
&\frac{1}{2}\langle |b|^{-1} \rangle^{\lfloor \frac{N}{2} \rfloor+1+m}+c|b|^{-1}\langle |b|^{-1} \rangle^{\lfloor \frac{N}{2} \rfloor+1+m}R^{-1}\mathcal{G}_1(\lambda,R)\mathcal{G}_2(\lambda,R)<1, \\
&\lambda-\frac{c_1}{2c}|b|^{-1}\langle |b|^{-1}\rangle^{\lfloor \frac{N}{2} \rfloor+1+m}R-c_1|b|^{-1}\langle |b|^{-1} \rangle^{\lfloor \frac{N}{2} \rfloor+1+m}\mathcal{G}_1(\lambda,R)\mathcal{G}_2(\lambda,R)>\frac{\lambda}{2},\\
&c_2|b|^{-1}\langle |b|^{-1} \rangle^{\lfloor \frac{N}{2}\rfloor+1+m} \Big(\mathcal{J}_1(\lambda,R)\mathcal{G}_2(\lambda,R)+\mathcal{G}_1(\lambda,R)\mathcal{J}_2(\lambda,R)\Big)<1.
\end{aligned}
\end{equation}
These conditions establish that $\Phi: \mathfrak{X}_{|b|^{-1}}(R,\lambda) \rightarrow \mathfrak{X}_{|b|^{-1}}(R,\lambda)$ is a contraction. Therefore, it has a fixed point, which is a solution of \eqref{agHartree}.
\end{proof}

Let us now obtain a global solution of \eqref{gHartree} with initial data $e^{i\frac{b|x|^2}{4}}v_0$, $v_0 \in \mathfrak{X}$ satisfying \eqref{lowebound}. As in Proposition \ref{propglobal}, let $b>0$ sufficiently large such that there exists a unique solution $v\in C([0,b^{-1}];\mathfrak{X})$ of \eqref{agHartree}. We define
\begin{equation}\label{defiglobalsol}
\begin{aligned}
u(x,t)=(1+bt)^{-\frac{N}{2}}e^{i\frac{b|x|^2}{4(1+bt)}}v\Big(\frac{x}{1+bt},\frac{t}{1+bt}\Big)
\end{aligned}
\end{equation}
for any $0\leq t <\infty$ and $x \in \mathbb{R}^N$. Recalling that $0\leq s <\frac{2m-N}{2}$, by Proposition \ref{derivexp} and similar arguments as in proof of Lemma \ref{derivexp2}, we find that $u$ solves \eqref{gHartree} with
\begin{equation}\label{classofsolu}
u \in C([0,\infty); \, H^s(\mathbb{R}^N))\cap L^{\infty}\big(\mathbb{
R}^N\times(0,\infty);\langle x \rangle^{\frac{N}{2}}\, dx \, dt\big),
\end{equation}
and $u(x,0)=e^{i\frac{b|x|^2}{4}}v_0$. We emphasize that the condition on the regularity $0\leq s <\frac{2m-N}{2}$ assures that $v\in C([0,b^{-1});L^{2}(\mathbb{R}^N;\langle x \rangle^{2s} \, dx))$, which is required to deduce \eqref{classofsolu}. Next, we use the following identity ((3.8) in \cite{cazenave_weissler})
\begin{equation}
e^{-it\Delta}u(x,t)=e^{\frac{ib|x|^2}{4}}e^{-\frac{t}{1+bt}\Delta}v\Big(x,\frac{t}{1+bt}\Big).
\end{equation}
Thus, we define
\begin{equation}
u_{+}=e^{\frac{ib|x|^2}{4}}e^{-\frac{1}{b}\Delta}v\Big(\frac{1}{b}\Big).
\end{equation}
We claim
\begin{equation}\label{claimscat}
e^{-it\Delta}u(t)\xrightarrow[t \to \infty]{} u_{+} \, \, \text{ in } \, \,  H^s(\mathbb{R}^N).
\end{equation}
Indeed, in view of \eqref{integralmap2}, we find
\begin{equation}
\begin{aligned}
e^{-it\Delta}u(t)-u_{+}=-i\mu e^{\frac{ib|x|^2}{4}}\int_{\frac{t}{1+bt}}^{\frac{1}{b}}(1-bt')^{N(p-1)-2-\gamma}e^{-it'\Delta}N(v(t'))\, dt'.
\end{aligned}
\end{equation}
By Theorem \ref{TheoSteDer}, property \eqref{prelimneq} and Lemma \ref{derivexp2}, we find
\begin{equation}\label{eqscat1}
\begin{aligned}
\|e^{-it\Delta}u(t)-u_{+}\|_{H^s}\lesssim & \|e^{-it\Delta}u(t)-u_{+}\|_{L^2}+\|\mathcal{D}^{s}\big(e^{-it\Delta}u(t)-u_{+}\big)\|_{L^2} \\
\lesssim &\|\int_{\frac{t}{1+bt}}^{\frac{1}{b}}(1-bt')^{N(p-1)-2-\gamma}e^{-it'\Delta}N(v(t'))\, dt'\|_{H^s}\\
&+\|\mathcal{D}^s\big(e^{\frac{ib|x|^2}{4}}\big)\int_{\frac{t}{1+bt}}^{\frac{1}{b}}(1-bt')^{N(p-1)-2-\gamma}e^{-it'\Delta}N(v(t'))\, dt'\|_{L^2}\\
\lesssim & \Big(\int_{\frac{t}{1+bt}}^{\frac{1}{b}}|1-bt'|^{N(p-1)-2-\gamma} dt'\Big)\sup_{t\in [0,|b|^{-1}]}\|N(v(t))\|_{H^s}\\
&+\langle b\rangle^{s}\|\int_{\frac{t}{1+bt}}^{\frac{1}{b}}(1-bt')^{N(p-1)-2-\gamma}\big( \langle x \rangle^s e^{-it'\Delta}N(v(t'))\big)\, dt'\|_{L^2}.
\end{aligned}
\end{equation}
Since $s<m<M+M_0-N$, by \eqref{eqmainTH14}, we get 
\begin{equation}\label{eqscat2}
\begin{aligned}
\Big(\int_{\frac{t}{1+bt}}^{\frac{1}{b}}|1-bt'|^{N(p-1)-2-\gamma} dt'\Big)&\sup_{t\in [0,|b|^{-1}]}\|N(v(t))\|_{H^s} \\
&\lesssim \Big(\int_{\frac{t}{1+bt}}^{\frac{1}{b}}|1-bt'|^{N(p-1)-2-\gamma} dt'\Big)\mathcal{G}_1(\lambda,R)\mathcal{G}_2(\lambda,R)\xrightarrow[t \to \infty]{} 0, 
\end{aligned}
\end{equation}
where $R$ is given by \eqref{initialdatahy}, and $\mathcal{G}_1$ and $\mathcal{G}_2$ as in \eqref{eqmainTH18}. Next, we apply Proposition \ref{propinfnorm} to get
\begin{equation}\label{eqscat3}
\begin{aligned}
\|\int_{\frac{t}{1+bt}}^{\frac{1}{b}}(1-bt')^{N(p-1)-2-\gamma}&\big( \langle x \rangle^s  e^{-it'\Delta}N(v(t'))\big)\, dt'\|_{L^2}\\
\lesssim & \|\langle x \rangle^{s-m}\|_{L^{2}}\int_{\frac{t}{1+bt}}^{\frac{1}{b}}(1-bt')^{N(p-1)-2-\gamma}\|\langle x \rangle^m e^{-it'\Delta}N(v(t'))\|_{L^{\infty}_x}\, dt' \\
\lesssim & \int_{\frac{t}{1+bt}}^{\frac{1}{b}}(1-bt')^{N(p-1)-2-\gamma}\langle t' \rangle^{\lfloor \frac{N}{2} \rfloor+m+1}\big(\sum_{|\alpha|\leq \lfloor \frac{N}{2} \rfloor}\|\langle x \rangle^m N(v(t'))\|_{L^{\infty}_x} \\
&+ \sum_{\lfloor \frac{N}{2} \rfloor<|\alpha|\leq M}\|\langle x \rangle^m\partial^{\alpha}N(v(t'))\|_{L^{\infty}_TL^2_x}+\|J^{M+M_0-N}N(v(t'))\|_{L^{\infty}_TL^2_x} \, \big)dt'\\
\lesssim & \Big(\int_{\frac{t}{1+bt}}^{\frac{1}{b}}(1-bt')^{N(p-1)-2-\gamma}\langle t' \rangle^{\lfloor \frac{N}{2} \rfloor+m+1} dt'\Big) \mathcal{G}_1(\lambda,R)\mathcal{G}_2(\lambda,R)\xrightarrow[t \to \infty]{} 0,
\end{aligned}
\end{equation}
where we have also used that $s<m-\frac{N}{2}$. Gathering \eqref{eqscat2} and \eqref{eqscat3} into \eqref{eqscat1}, we deduce the desired limit \eqref{claimscat}. Finally, since $\sup_{0<t<|b|^{-1}} \|v(t)\|_{L^{\infty}}<\infty$, by \eqref{defiglobalsol}, it follows that $\sup_{t>0} \, (1+t)^{\frac{N}{2}}\|u(t)\|_{L^{\infty}}<\infty$. The proof of Theorem \ref{scatres} is complete.

\begin{rem}
\normalfont The arguments in \eqref{eqscat3} also establish the following limit
\begin{equation}
\lim_{\substack{t\to \infty \\ t>0}} \|\langle x \rangle^s \big(e^{-it\Delta}u(t)-u_{+}\big)\|_{L^2}=0.
\end{equation}
\end{rem}


\section{Blow-up criterion} \label{S:blowup}

This section aims to establish Theorem \ref{blowupcriteria} and Corollary \ref{blowupcriteriacor}. We assume that $\max\{\frac{N+2}{N},\frac{4}{3}\}<p<2$, $0<\gamma<\min\{N(p-1)-2,\frac{(N+2)(p-1)-2}{2},\frac{N(3p-4)}{2p}\}$ and $\mu >0$. We consider $m \in \mathbb{R}^{+}$ as in the statement of Theorem \ref{blowupcriteria}, and $M_0$ and $M$ satisfying \eqref{condmainT2} and \eqref{condmainT3}, respectively. Then, we set the space $\mathfrak{X}$ by \eqref{spacedef} and \eqref{normspace}.

Let $v_0\in \mathfrak{X}$ verify \eqref{condi2} and $u \in C([0,T];\mathfrak{X})$ be the corresponding solution of \eqref{gHartree} with initial data $u_0=e^{\frac{ib|x|^2}{4}}v_0$ provided by Theorem \ref{mainTHM}. We first observe that
\begin{equation}\label{redeq0}
u\in C([0,T]; L^2(\mathbb{R}^N; |x|^2\, dx)),
\end{equation}
which is a direct consequence of the following computation
\begin{equation*}
\|\langle x \rangle u\|_{L^{\infty}_T L^2_x}\leq \|\langle x \rangle^{m} u\|_{L^{\infty}_{T,x}}\|\langle x \rangle^{-(m-1)}\|_{L^2_x}\lesssim \|\langle x \rangle^{m} u\|_{L^{\infty}_{T,x}},
\end{equation*}
and the fact that $m>\frac{N}{2}+1$. We remark that the condition $0<\gamma< \frac{(N+2)(p-1)-2}{2}$ in Theorem \ref{blowupcriteria} assures that $\frac{N+2}{2}<\frac{N-2\gamma}{2(2-p)}$, which is implicitly required in \eqref{condiblow}. Additionally, it follows
\begin{equation}\label{redeq0.1}
\Big(\frac{1}{|\cdot|^{N-\gamma}}\ast |u|^p \Big)|u|^{p}\in L^{\infty}([0,T];L^{1}(\mathbb{R}^N)).
\end{equation}
Indeed, we apply H\"older's inequality and Hardy-Littlewood-Sobolev inequality to get
\begin{equation*}
\begin{aligned}
\|\big(|\cdot|^{-(N-\gamma)}\ast |u|^p \big)|u|^{p}\|_{L^{\infty}_TL^1_x}\lesssim &  \||\cdot|^{-(N-\gamma)}\ast |u|^p  \|_{L^{\infty}_T L^{\frac{2N}{N-2\gamma}}_x}\||u|^{p}\|_{L^{\infty}_TL^{\frac{2N}{N+2\gamma}}_x}\\
\lesssim &  \||u|^p  \|_{L^{\infty}_T L^2_x}\||u|^{p}\|_{L^{\infty}_TL^{\frac{2N}{N+2\gamma}}_x}\\
\lesssim &  \|\langle x \rangle^{-mp} \|_{L^2}\|\langle x \rangle^{-mp}\|_{L^{\frac{2N}{N+2\gamma}}}\|\langle x \rangle^m u\|_{L^{\infty}_{T,x}}^{2p},
\end{aligned}
\end{equation*}
where the restriction \eqref{condiblow} implies $\langle x \rangle^{-mp} \in L^2(\mathbb{R}^N)\cap L^{\frac{2N}{N+2\gamma}}(\mathbb{R}^N)$. Next, we recall the variance \eqref{variance}
\begin{equation*}
V(t)=\int_{\mathbb{R}^N} |x|^2|u(x,t)|^2 \, dx,
\end{equation*}
and the energy \eqref{energycon}
\begin{equation*}
E[u(t)]=\frac{1}{2}\int_{\mathbb{R}^N} |\nabla u(x,t)|^2\, dx-\frac{\mu}{2p}\int_{\mathbb{R}^N}\Big(\frac{1}{|\cdot|^{N-\gamma}}\ast |u(\cdot,t)|^{p}\Big)(x,t)|u(x,t)|^p\, dx.
\end{equation*}
Notice that \eqref{redeq0} and \eqref{redeq0.1} establish the validity of $V(t)$ and $E[u(t)]$. Then, we deduce that the above solution $u$ of \eqref{gHartree} satisfies the following virial identities
\begin{equation}\label{redeq0.2}
V_t(t)=4\operatorname{Im}\int_{\mathbb{R}^N} \overline{u} (x \cdot \nabla u)\, dx,
\end{equation}
and
\begin{equation}\label{redeq0.3}
\begin{aligned}
V_{tt}(t)&=16E[u(t)]-8s_c\frac{(p-1)}{p}\mu \int_{\mathbb{R}^{N}}\big(|\cdot|^{-(N-\gamma)}\ast |u|^p\big)(x,t)|u|^p(x,t)\, dx\\
&=16(s_c(p-1)+1)E[u(t)]-8s_c(p-1)\|\nabla u\|_{L^2}^2,
\end{aligned}
\end{equation}
also, we recall the critical index \eqref{critiindex},
\begin{equation*}
s_c=\frac{N}{2}-\frac{\gamma+2}{2(p-1)}.
\end{equation*}
In particular, $s_c>0$, if and only if $0<\gamma<N(p-1)-2$, which is given by our assumption $p>\frac{N+2}{N}$. 

Consequently, \eqref{redeq0.2} and \eqref{redeq0.3} allow us to follow the same arguments as in the proof of \cite[Theorem 1.3]{AndySvetlan2} for the gHartree equation \eqref{gHartree} with $p\geq 2$ to deduce Theorem \ref{blowupcriteria}.

Next, we infer some consequences of Theorem \ref{blowupcriteria}. Since we are considering solutions of \eqref{gHartree} with initial condition $u_0=e^{\frac{ib|x|^2}{4}}v_0$, we  deduce the following identities
\begin{equation}\label{eqblowupcondi1}
\begin{aligned}
V(0)&=\|xu_0\|_{L^2}^2=\|xv_0\|_{L^2}^2,\\
V_t(0)&=4\operatorname{Im} \int_{\mathbb{R}^N} \overline{v_0}(x\cdot \nabla v_0)\, dx +2b\int |x|^2|v_0|^2 \, dx, \\
E[u_0]&=E[v_0]+\frac{b}{2}\operatorname{Im} \int_{\mathbb{R}^N} \overline{v_0}(x \cdot \nabla v_0)\, dx+\frac{|b|^2}{8}\int_{\mathbb{R}^N}|x|^2|v_0|^2 \, dx.
\end{aligned}
\end{equation}
Notice that $E[u_0]\to \infty$ and $V_t(0) \to -\infty$ as $b \to -\infty$. Let us rewrite some conditions in Theorem \ref{blowupcriteria} in terms of the equations in \eqref{eqblowupcondi1}. We have $\frac{E[u_0]V(0)}{(\omega_c M[u_0])^2} \geq 1$ if and only if
\begin{equation}\label{eqlbowupcondi2}
\begin{aligned}
E[v_0]+\frac{b}{2}\operatorname{Im} \int_{\mathbb{R}^N} \overline{v_0}(x \cdot \nabla v_0)\, dx+\frac{|b|^2}{8}\int_{\mathbb{R}^N}|x|^2|v_0|^2 \, dx-\frac{(\omega_c M[v_0])^2}{\|xv_0\|_{L^2}^2}\geq 0.
\end{aligned}
\end{equation}
Assuming that $\frac{\partial_{t}V(0)}{\omega_c M[u_0]}>0$ and $\frac{E[u_0]V(0)}{(\omega_c M[u_0])^2} <1 $, the inequality \eqref{blowupcondi} is equivalent to 
\begin{equation}\label{eqblowupcondi3}
\begin{aligned}
\frac{k_{c}|\partial_{t}V(0)|^2-32k_{c}E[u_0]V(0)+32(1+k_{c})(\omega_c M[v_0])^2}{k_{c}(\omega_c M[v_0])^2}<\frac{32(\omega_c M[v_0])^{2k_{c}}}{k_{c}(E[u_0]V(0))^{k_{c}}}.
\end{aligned}
\end{equation}
Likewise,  if $\frac{\partial_{t}V(0)}{\omega_c M[u_0]}<0$ and $\frac{E[u_0]V(0)}{(\omega_c M[u_0])^2} \geq 1$, \eqref{blowupcondi} is determined by
\begin{equation}\label{eqblowupcondi3.1}
\begin{aligned}
\frac{k_{c}|\partial_{t}V(0)|^2-32k_{c}E[u_0]V(0)+32(1+k_{c})(\omega_c M[v_0])^2}{k_{c}(\omega_c M[v_0])^2}>\frac{32(\omega_c M[v_0])^{2k_{c}}}{k_{c}(E[u_0]V(0))^{k_{c}}}.
\end{aligned}
\end{equation}
We remark that
\begin{equation}\label{eqblowupcondi4}
\begin{aligned}
|\partial_{t}V(0)|^2-32E[u_0]V(0)=&16\Big( \operatorname{Im} \int_{\mathbb{R}^N} \overline{v_0}(x \cdot \nabla v_0)\, dx \Big)^2-32E[v_0]\|xv_0\|_{L^2}^2.
\end{aligned}
\end{equation}
Now we assume that $v_0$ is a real-valued function. Then,
\begin{equation}\label{redeq1} 
\frac{E[u_0]V(0)}{(\omega_c M[u_0])^2}=\frac{\big(8E[v_0]+|b|^2\|xv_0\|_{L^2}^2\big)\|xv_0\|_{L^2}^2}{8(\omega_c M[v_0])^2},
\end{equation}
\begin{equation}\label{redeq2}
\frac{\partial_{t}V(0)}{\omega_c M[u_0]}=\frac{2b\|xv_0\|_{L^2}^2}{\omega_c M[v_0]},
\end{equation}
and
\begin{equation}\label{condionv1}
\begin{aligned}
\frac{E[u_0]V(0)}{(\omega_c M[u_0])^2}\geq 1 \quad \text{ if and only if } \quad |b|^2
\geq \frac{8(\omega_c M[v_0])^2-8E[v_0]\|xv_0\|_{L^2}^2}{\|xv_0\|_{L^2}^4}.
\end{aligned}
\end{equation}
Notice that $E[u_0]>0$, if $|b|^2 >\frac{-8E[v_0]}{\|xv_0\|_{L^2}^2}$. Now, we divide our analysis according to the sign of $b\neq 0$.
\begin{itemize}
\item[I)] \underline{Assume $b>0$}. Here, $\frac{\partial_{t}V(0)}{\omega_c M[u_0]}>0$, and thus, by \eqref{condionv1}, we can only verify the hypothesis of Theorem \ref{blowupcriteria} when
\begin{equation}\label{redeq2.1}
|b|< \frac{2\sqrt{2}}{\|xv_0\|_{L^2}^2}\Big((\omega_c M[v_0])^2-E[v_0]\|xv_0\|_{L^2}^2\Big)^{1/2}.
\end{equation}
This forces us to assume
\begin{equation}
\frac{E[v_0]\|xv_0\|_{L^2}^2}{(M[v_0])^2} < \omega_c^2.
\end{equation}
Hence, in virtue of \eqref{eqblowupcondi3} and \eqref{eqblowupcondi4}, we find
\begin{equation}\label{redeq2.2}
\begin{aligned}
|b|^2< \frac{8}{\|xv_0\|_{L^2}^4}\bigg(\frac{(\omega_c M[v_0])^{2k_{c}+2}}{\big((1+k_{c})(\omega_c M[v_0])^2-k_{c}E[v_0]\|xv_0\|_{L^2}^2\big)}\bigg)^{1/k_{c}}-\frac{8E[v_0]}{\|xv_0\|_{L^2}^2}.
\end{aligned}
\end{equation}
If the right-hand side of \eqref{redeq2.2} is positive, Theorem \ref{blowupcriteria} assures the existence of two numbers $b_1>b_0\geq 0$ such that the solution $u$ of \eqref{gHartree} associated to $u_0=e^{\frac{ib|x|^2}{2}}v_0$ determined by Theorem \ref{mainTHM} blows up in finite time for all $b_0<b< b_1$. In particular, if $E[v_0]>0$, one can take $b_0=0$.
\item[II)] \underline{Assume $b<0$}. In this case, \eqref{eqblowupcondi3.1} and \eqref{eqblowupcondi4} impose the condition 
\begin{equation}\label{condionv1.1}
\frac{E[v_0]\|xv_0\|_{L^2}^2}{(\omega_c M[v_0])^2}<\frac{1+k_{c}}{k_{c}}.
\end{equation} 
Hence, Theorem \ref{blowupcriteria} and \eqref{condionv1.1} yield the existence of $b_1\leq 0$ such that for all $b \leq b_1$ the solution $u$ of \eqref{gHartree} with initial condition $u_0=e^{\frac{ib|x|^2}{2}}v_0$ blows up in finite time.
\end{itemize}
Collecting the conclusion in a) and b), we deduce Corollary \ref{blowupcriteriacor}.

\begin{rem}\label{remblowup}
\normalfont \begin{itemize}
\item[(i)] The same conclusion in parts I) and II) above are valid assuming the weaker hypothesis $ \operatorname{Im} \int_{\mathbb{R}^N} \overline{v_0}(x \cdot \nabla v_0)\, dx =0$.
\item[(ii)] Assuming
\begin{equation}
2(1+k_{c})(\omega_c M[v_0])^2+k_c\Big( \operatorname{Im} \int_{\mathbb{R}^N} \overline{v_0}(x \cdot \nabla v_0)\, dx \Big)^2-2k_cE[v_0]\|xv_0\|_{L^2}^2>0,
\end{equation}
the equations \eqref{eqblowupcondi1}-\eqref{eqblowupcondi4}, and Theorem \ref{blowupcriteria} assure that there exists $b_1\leq 0$ such that for any $b \leq b_1$, the solution $u(t)$ of \eqref{gHartree} associated to $u_0=e^{\frac{ib|x|^2}{2}}v_0$ determined by Theorem \ref{mainTHM} blows up in finite time.
\end{itemize}
\end{rem}

 
\bibliographystyle{abbrv}
\bibliography{bibli}

\begin{thebibliography}{10}

\bibitem{AndySvetlan2}
A.~K. Arora and S.~Roudenko.
\newblock Well-posedness and blow-up properties for the generalized {H}artree
  equation.
\newblock {\em J. Hyperbolic Diff. Eq.}, 17(4):727--763, 2020.

\bibitem{2019gHAnudRoud}
A.~K. {Arora} and S.~{Roudenko}.
\newblock {Global behavior of solutions to the focusing generalized Hartree
  equation}.
\newblock {\em Michigan Math J., forthcoming}, 2021.

\bibitem{ARY2020}
A.~K. Arora, S.~Roudenko, and K.~Yang.
\newblock On the focusing generalized {H}artree equations.
\newblock {\em Math. in Appl. Sci. and Eng.}, 1(4):381--400, December 2020.

\bibitem{CAZENAVE2011}
T.~Cazenave, D.~Fang, and Z.~Han.
\newblock {Continuous dependence for NLS in fractional order spaces}.
\newblock {\em Ann. Inst. H. Poincar\'{e} Anal. Non Lin\'{e}aire},
  28(1):135--147, 2011.

\bibitem{CazHauNaum2020}
T.~Cazenave, Z.~Han, and I.~Naumkin.
\newblock {Asymptotic behavior for a dissipative nonlinear Schr\"odinger
  equation}.
\newblock {\em Nonlinear Anal.}, 205:112243, 2021.

\bibitem{CazNaum2016}
T.~Cazenave and I.~Naumkin.
\newblock {Local existence, global existence, and scattering for the nonlinear
  Schrödinger equation}.
\newblock {\em Comm. Contemp. Math.}, 19(02):1650038, 20 pp, 2017.

\bibitem{CazNaum2018}
T.~Cazenave and I.~Naumkin.
\newblock {Modified scattering for the critical nonlinear Schrödinger
  equation}.
\newblock {\em J. Funct. Anal.}, 274(2):402--432, 2018.

\bibitem{cazenave_weissler}
T.~Cazenave and F.~B. Weissler.
\newblock {The structure of solutions to the pseudo-conformally invariant
  nonlinear Schrödinger equation}.
\newblock {\em Proc. Royal Soc. Edinburgh: Sect. A Math.},
  117(3-4):251--–273, 1991.

\bibitem{Rieszweigh2}
D.~Cruz-Uribe and K.~Moen.
\newblock {One and two weight norm inequalities for Riesz potentials}.
\newblock {\em Illinois J. Math.}, 57(1):295--323, 2013.

\bibitem{JavierHarmo}
J.~Duoandikoetxea.
\newblock {\em Fourier Analysis}, volume~29 of {\em Grad. Stud. Math.}
\newblock Amer. Math. Soc., 2001.

\bibitem{DucRoud}
T.~Duyckaerts and S.~Roudenko.
\newblock {Going Beyond the Threshold: Scattering and Blow-up in the Focusing
  NLS Equation}.
\newblock {\em Comm. Math. Phys.}, 334(3):1573–--1615, 2015.

\bibitem{Holmer_2010}
J.~Holmer, R.~Platte, and S.~Roudenko.
\newblock {Blow-up criteria for the 3D cubic nonlinear Schrödinger equation}.
\newblock {\em Nonlinearity}, 23(4):977--1030, 2010.

\bibitem{LACEY}
M.~T. Lacey, K.~Moen, C.~P\'erez, and R.~H. Torres.
\newblock {Sharp weighted bounds for fractional integral operators}.
\newblock {\em J. Funct. Anal.}, 259(5):1073--1097, 2010.

\bibitem{LinaresMiyaGus}
F.~Linares, H.~Miyazaki, and G.~Ponce.
\newblock {On a class of solutions to the generalized KdV type equation}.
\newblock {\em Comm. Contemp. Math.}, 21(07):1850056, 2019.

\bibitem{LINARES2020}
F.~Linares, A.~Pastor, and J.~{Drumond Silva}.
\newblock {Dispersive blow-up for solutions of the Zakharov-Kuznetsov
  equation}.
\newblock {\em Ann. Henri Poincar\'{e} Anal. Non Lin\'{e}aire}, 2020.

\bibitem{LinaresI}
F.~Linares, G.~Ponce, and G.~Santos.
\newblock {On a Class of Solutions to the Generalized Derivative Schrödinger
  Equations}.
\newblock {\em Acta. Math. Sin.-English Ser}, 35(06):1057–--1073, 2019.

\bibitem{LinaresII}
F.~Linares, G.~Ponce, and G.~Santos.
\newblock {On a class of solutions to the generalized derivative Schrödinger
  equations II}.
\newblock {\em J. Diff. Eq.}, 267(1):97--118, 2019.

\bibitem{PavelI}
P.~M. Lushnikov.
\newblock {Collapse of Bose-Einstein condensates with dipole-dipole
  interactions}.
\newblock {\em Phys. Rev. A}, 66:051601, 2002.

\bibitem{PavelII}
P.~M. Lushnikov.
\newblock {Collapse and stable self-trapping for Bose-Einstein condensates with
  $1/{r}^{b}$-type attractive interatomic interaction potential}.
\newblock {\em Phys. Rev. A}, 82:023615, Aug 2010.

\bibitem{Miyazaki2020}
H.~Miyazaki.
\newblock {Lower bound for the lifespan of solutions to the generalized KdV
  equation with low degree of nonlinearity}.
\newblock {\em Advanced Studies in Pure Mathematics}, 85:303--313, 2020.

\bibitem{MOROZ}
V.~Moroz and J.~{Van Schaftingen}.
\newblock {Groundstates of nonlinear Choquard equations: Existence, qualitative
  properties and decay asymptotics}.
\newblock {\em J. Funct. Anal.}, 265(2):153--184, 2013.

\bibitem{MuckApcond}
B.~Muckenhoupt.
\newblock {Weighted Norm Inequalities for the Hardy Maximal Function}.
\newblock {\em Trans. Amer. Math. Soc.}, 165:207--226, 1972.

\bibitem{Rieszweigh}
B.~Muckenhoupt and R.~L. Wheeden.
\newblock {Weighted Norm Inequalities for Fractional Integrals}.
\newblock {\em Trans. Amer. Math. Soc.}, 192:261--274, 1974.

\bibitem{NahasPo}
J.~Nahas and G.~Ponce.
\newblock {On the Persistent Properties of Solutions to Semi-Linear
  Schrödinger Equation}.
\newblock {\em Comm. PDE}, 34(10):1208--1227, 2009.

\bibitem{SteinThe}
E.~M. Stein.
\newblock {The characterization of functions arising as potentials}.
\newblock {\em Bull. Amer. Math. Soc.}, 67(1):102--104, 1961.

\bibitem{stein1993harmonic}
E.~M. Stein.
\newblock {\em {Harmonic Analysis: Real-variable Methods, Orthogonality, and
  Oscillatory Integrals}}.
\newblock Monographs in harmonic analysis. Princeton University Press, 1993.

\bibitem{SGU}
E.~M. Stein and G.~Weiss.
\newblock {Fractional Integrals on $n$-dimensional Euclidean Space}.
\newblock {\em J. Math. Mechanics}, 7(4):503--514, 1958.

\bibitem{YRZ2020}
K.~Yang, S.~Roudenko, and Y.~Zhao.
\newblock Stable blow-up dynamics in the ${L}^2$-critical and
  ${L}^2$-supercritical generalized {H}artree equations.
\newblock {\em Stud. Appl. Math.}, 145(4):647--695, 2020.

\end{thebibliography}

\end{document}